\documentclass[10pt,a4paper]{amsart}
\usepackage{hyperref}
\usepackage{mathrsfs}
\usepackage{amsmath}
\usepackage{amssymb}
\usepackage[all]{xy}
\usepackage{latexsym}
\usepackage[T1]{fontenc}
\usepackage{datetime}
\usepackage{pgf,tikz}
\usetikzlibrary{arrows}

 \usepackage[latin1]{inputenc}

\mathchardef\dashmod="2D

\newcommand{ \inj}{ \hookrightarrow}
\newcommand{ \surj}{ \twoheadrightarrow}
\newcommand{ \Hom}{ \mathrm{Hom}}
\newcommand{ \rp}{ R_+}
\newcommand{ \rn}{ R_-}

\newcommand{\kr}{k\R}
\newcommand{\E}{\mathcal{E}}
\newcommand{\qo}{\mathbb{Q}_0}
\newcommand{\qi}{\mathbb{Q}_1}
\newcommand{\li}{\Lambda_1}
\newcommand{\lo}{\Lambda_0}
\newcommand{\lamqo}{\Lambda(Q_0)}

\newcommand{\kera}{\mathrm{Soc}}

\newcommand{\sqop}{\widetilde{Sq^2}}

\newcommand{\Le}{\mathbb{L}}

\newcommand{\extrel}{\mathrm{Ext}_{\mathcal{R}el}}
\newcommand{\torrel}{\mathrm{Tor}^{\mathcal{R}el}}

\newcommand{\R}{\mathbb{R}}
\newcommand{\Z}{\mathbb{Z}}
\newcommand{\F}{\mathbb{F}}

\newcommand{\mf}{\underline{\F}}
\newcommand{\mz}{\underline{\Z}}
\newcommand{\ste}{\mathcal{A}}

\newcommand{\ab}{\mathcal{B}}
\newcommand{\Ccal}{\mathcal{C}}
\newcommand{\tri}{\mathcal{T}}
\newcommand{\sh}{\mathcal{SH}}
\newcommand{\Sp}{\mathcal{S}p}
\newcommand{\gr}{Q}

\theoremstyle{definition}
\newtheorem{de}{Definition}[section]
\newtheorem{nota}[de]{Notation}

\theoremstyle{plain}
\newtheorem{thm}[de]{Theorem}
\newtheorem{lemma}[de]{Lemma}
\newtheorem{pro}[de]{Proposition}
\newtheorem{corr}[de]{Corollary}

\newtheorem*{thm*}{Theorem}
\newtheorem*{lemma*}{Lemma}
\newtheorem*{pro*}{Proposition}
\newtheorem*{corr*}{Corollary}

\theoremstyle{remark}
\newtheorem{rk}[de]{Remark}
\newtheorem{ex}[de]{Example}

\title[$ k\R^{\star}(BV)$]{Detection of height $h$ and connective Real $K$-theory of elementary abelian $2$-groups}
\author{Nicolas Ricka}
\address{Institut de Recherche Mathématique Avancée \\
 UMR 7501 \\
 7 rue René-Descartes \\
 67084 Strasbourg, France}
\email{ricka@math.univ-paris13.fr}
\date{\today}
\thanks{ I want to thank my avisor Geoffrey Powell for both the very stimulating mathematical discussions, and his careful readings of early expositions of the results presented in this paper. I also want to thank the reviewer, whose remarks have improved the clarity of the exposition.}

\begin{document}

\begin{abstract}
In this paper, we determine the connective K-theory with reality of elementary abelian $2$-groups as a module over $\Z[v_1,a]$, where $v_1$ is the equivariant Bott class and $a$ the Euler class of the sign representation. This gives in particular a new approach to the computation of the connective real K-theory of such groups. The originality here is to make all computations in the equivariant stable category, considering only equivariant cohomology theories, and to use relative homological algebra over some subalgebras of the equivariant Steenrod algebra to perform explicit computations.
\end{abstract}

\maketitle

During the last few years, the study of equivariant stable homotopy theory has proven itself to be efficient in solving stable homotopy theoretic problems. For example, Voevodsky's $ \R$-realization functor \cite[section 3.3]{MV01} provided inspiration to the development of equivariant tools such as Hill-Hopkins-Ravenel's slice filtration in equivariant stable homotopy theory in the proof of the Kervaire invariant one problem in \cite{HHR}.

Before \cite{HHR}, a particular case of the slice filtration already appeared for the group with two elements $ \gr$: Dugger considered the slice filtration for Atiyah's K-theory with reality spectrum $K\R$ in  \cite{Du03}, and identifies its $k$-invariants. This tower consists only in shifts of $k\R$, the connective cover of $K\R$, and in particular is far simpler than the equivariant Postnikov tower of $K\R$.
This indicates that the study of the $k\R$-cohomology of a spectrum $X$ should rely on the slice tower of $k\R$, and thus on the action of the equivariant modulo $2$ Steenrod algebra on the cohomology of $X$.

Throughout the paper, $\gr$ denotes the group with two elements. Let $RO( \gr)$ be the Grothendieck group of real representations of $\gr$. Let $1$ be the the trivial representation and $\alpha$ the sign representatin of $\gr$, so that $RO(\gr)$ is the free abelian group on  $1$ and $ \alpha$. 

We will use the convention of Hu and Kriz in \cite{HK01} for graded objects: a $*$ indicates a $\Z$-grading and a $\star$ indicates a $RO( \gr)$-grading. For example, $\gr$-equivariant cohomology theories are naturally $RO( \gr)$-graded objects. In particular, ordinary equivariant cohomology with respect to the constant Mackey functor $\underline{\Z}$ is denoted $H\underline{\Z}^{\star}$, and connective K-theory with reality is denoted $\kr^{\star}$. Note that these objects are not mere abelian groups, but they are the more structured corresponding $\gr$-equivariant analogue of abelian groups: Mackey functors fot the group $\gr$.

Let $V$ be an elementary abelian $2$-group. Consider the classifying space $BV$ as a $ \gr$-space with trivial $\gr$-action. The aim of this article is to set up an equivariant machinery general enough to get an explicit description of $\kr^{ \star}(BV)$: the connective $K$-theory with reality of $BV$.

The study of $ \kr^{ \star}(BV)$ provides in particular an new and unified computation of the $\Z$-graded abelian groups $ko^*(BV)$ and $ku^*(BV)$ provided by Ossa \cite{Os89} for $ku^*$ and Yu's thesis \cite{Yu} for $ko^*$, correcting Ossa's computation of $ko^*(BV)$ in \cite{Os89}.

These groups where studied extensively by Bruner and Greenlees in \cite{BG03} for $ku^*$ of groups and \cite{BG10} for $ko^*$. In \cite{Po11} and \cite{Po12}, Powell studied their functorial behavior (as a functor of $V$).

In these arguments, the realification-complexification exact sequence
$$ ko \stackrel{c}{ \rightarrow} ku \stackrel{R}{ \rightarrow} \Sigma^2ko,$$
is always at the center of the computation. Since we know by \cite{At66} that this sequence is of $\gr$-equivariant nature, this provides another motivation to study these objects from an equivariant point of view. \\

The aim of this paper is to compute the $\gr$-equivariant k-theory with reality of groups, \textit{i.e.} the $RO(\gr)$-graded Mackey functor $\kr^{\star}(BV)$.
The main result is an explicit computation of the $\Z[a,v_1]$-module structure of $\kr^{\star}_{\gr}(BV)$, that is equivariant classes of maps $BV \rightarrow \Sigma^{\star}\kr$.

\begin{thm*}[\emph{Theorem \ref{thm:krbvformula}}]
There is a $\Z[a,v_1]$-module splitting of $ \kr^{\star}_{\gr}(BV)$ as
$$ \kr^{\star}_{\gr}(BV) \cong co\Gamma_{v_1}(  \kr^{\star}_{\gr}(BV)) \oplus F^1(V) \oplus F^2(V) \otimes_{\Z} \Lambda(v_1),$$
where $co\Gamma_{v_1}$ denotes the cotorsion part, that is, for a $v_1$-module $M$, the quotient of $M$ by its $v_1$-torsion part: $M/\Gamma_{v_1}M$.
and isomorphisms:
\begin{enumerate}
\item $F^1(V) \cong Im( \qi : H \mf^{\star}_{\gr}(BV) \rightarrow H \mf^{\star+2+\alpha}_{\gr}(BV))$,
\item $F^2(V) \cong Sq^2Sq^2Sq^2 F$ where $F$ is the largest free $ \ste(1)$-module contained in $ H \F^*(BV)$,
\item and $$ \frac{\Phi_n}{\Phi_{n+1}} \cong  \bigoplus_{i=1}^n \left(( \Sigma^{-i(1+ \alpha)}HP^{ \star})_{twist \geq 0} \oplus ( \Sigma^{-i(1+ \alpha)-1}HP^{ \star})_{twist \leq -2}, \right)^{ \oplus \begin{pmatrix} n \\ i \end{pmatrix}}$$
where $$\Phi_n = Im(v_1^n : co\Gamma_{v_1}( \kr^{\star+n(1+ \alpha)}_{\gr}(BV)) \rightarrow co\Gamma_{v_1}( \kr^{\star+n(1+ \alpha)}_{\gr}(BV))),$$ for $n \geq 0$ defines a decreasing exhaustive filtration of the $\Z[a,v_1]$-module $co\Gamma_{v_1}(  \kr^{\star}_{\gr}(BV))$.
\end{enumerate}
\end{thm*}
With $HP^{\star}$ some explicit $\Z[a]$-module defined in Notation \ref{nota_HPstar}.

As the $RO(\gr)$-graded abelian group $\kr^{\star}_{e}(BV) = ku^*(BV)[\sigma^{\pm 1}]$ is quite easy to determine using various techniques (see for example \cite{BG03}), and the action of the Euler class $a$ of the representation $\alpha$ suffices to understand the restriction and transfer of a Mackey functor obtained via a cohomology theory, this is really a determination of $\kr^{\star}(BV)$.

During the trip, we obtain two families of result of independent interest.
First we show that the equivariant Steenrod algebra is a flat Hopf algebroid, by the general machinery developed by Boardman, this gives a nice duality between modules over the equivariant Steenrod algebra and comodules over its dual:

\begin{thm*}[\emph{Theorem \ref{thm:boardman_hmf}}]
Denote
\begin{itemize}
\item $ \ste^{\star} = (H\mf^{\star}_{\gr}(H\mf))$
\item $\ste_{\star} = (H\mf_{\star}^{\gr}(H\mf))$.
\end{itemize}
Then, , the category of Boardman $\ste^{ \star}$-modules is equivalent to the category of Boardman $\ste_{ \star}$-comodules.
\end{thm*}

Using these property, we then show that the equivariant analogues of the Cartan formulae holds with respect to multiplication with elements in the coefficient ring $H\mf^{\star}_{\gr}$.

\begin{pro*}[\emph{Proposition \ref{pro_cartancoeff}}]

\begin{enumerate}
 \item For all $h \in H \mf^{ \star}_{\gr}$,  $ \eta_R(h) = \sum_{h' \in H \mf^{ \star}_{\gr}, x \in \ste_{ \star}} (x^{ \vee}h)x$.
\item Let $M$ be a Boardman $ \ste^{ \star}$-module and $x^{ \vee} \in \ste^{ \star}$.
Define $x'_i$ and $x''_i$ in $\ste_{\star}$ by
$$ \sum_{i \geq 0} x'_i \otimes x''_i = \sum_{h,y,z| pr_x(yz) = \eta_R(h)x} hy \otimes z \in \ste_{\star}$$
where the second sum is over $h \in H \mf_{ \star}^{\gr}$ and $y,z$ in $\mathcal{BM}$.
Then, for all $h \in H \mf^{ \star}_{\gr}$ and $m \in M$,
$$ x(hm) = \sum_{i \geq 0} x'_i(h) x''_i(m).$$
\end{enumerate}
\end{pro*}

During the proof of our main theorem, we have to study the structure of the category of $\E$-modules, where $\E = \Lambda_{\F}(\qo,\qi)$, where $\qo$ and $\qi$ are two equivariant stable cohomology operations related to the classical Milnor operations. We show a surprising relationship between this category and the well known category of modules over the subalgebra of the (non-equivariant) Steenrod algebra $\ste(1)$ generated by the two first Steenrod operations $Sq^1$ and $Sq^2$.

\begin{thm*} [\emph{Theorem \ref{thm:formula_qoi}}]
Let $H\mf^{\star} \otimes(-) : \F\dashmod mod^{\Z} \rightarrow H\mf^{\star}\dashmod mod$, where $H\mf^{\star}\dashmod mod$ denotes the category of $H\mf^{\star}$-modules in $\mathcal{M}^{RO(\gr)}$, be the extension of scalars functor.
Then the following diagram is commutative up to natural isomorphism
$$ \xymatrix@=2cm{ \sh \ar[r]  \ar[d]_{H\F^*} & \gr\sh \ar[d]^{H\mf^{\star}} \\
\F\dashmod mod^{\Z} \ar[r]_{H\mf \otimes(-)} & H\mf^{\star}\dashmod mod }$$
where the top arrow is the inflation functor.

Moreover, this refines to the following commutative diagram up to natural isomorphism
$$ \xymatrix@=2cm{ \sh \ar[r]  \ar[d]_{H\F^*} & \gr\sh \ar[d]^{H\mf^{\star}} \\
\ste(1)\dashmod mod \ar[r]_{R} & \E\dashmod mod, }$$
Where $R$ denotes the extension of scalars from $\F$ to $H\mf^{\star}$ together with a natural action of $\E$ coming from the action of $\ste(1)$ on the source.

\end{thm*}

In the first section of the paper, we consider towers of objects $k_{\bullet}$ in a triangulated category $\tri$ and an exact functor $(-)^* : \tri \rightarrow \ab$. Note that $(-)^*$ denotes a covariant functor. The notation is intended to follow the usual convention in stable homotopy theory, where the functor $(-)^*$ which associated a cohomology theory to a spectrum is indeed covariant. 

We then introduce a property of this tower, called the detection of height $h$, where $h$ is a nonnegative integer. We show that, when $h \leq 2$, the objects $k^*_{\bullet}$ can be recovered entirely from a $\Lambda(\theta)$-module structure on $C_{\bullet}^*$, where $C_{\bullet}$ are the layers of $k_{\bullet}$. Precisely, $k_{\bullet}$ can be understood in terms of the homology of the complex $(C_{\bullet}^*, \theta)$. We also provide some general tools to check detection of height $h$ when $h \leq 2$.

We then recall a result of Dugger's work in \cite{Du03} which identifies the slice tower of Atiyah's k-theory with reality $\gr$-equivariant spectrum. This enables the use of the first section, when the tower $k_{\bullet}$ is the slice tower of $\kr$, and the functor $(-)^*$ is $[-,X]^*$, where $X$ is some fixed space. Consequently, the first section brings the problem of computing $\kr^{\star}_{\gr}(BV)$ back to understanding the homology of $H\underline{\Z}^{\star}_{\gr}(BV)$ with respect to an integral lift of the second equivariant Milnor operation $\qi$. We call $\mathcal{H}^{\star}_{\gr}(V)$ this bigraded object.

It is easy to see that $\mathcal{H}^{\star}(V)$ can be recovered from the $\E$-module structure on $H\mf^{\star}_{\gr}(BV)$, where $\E := \Lambda_{\F}(\qo,\qi)$. The subject of the third section is precisely the determination of the $\E$-module structure on $H\mf^{\star}_{\gr}(X)$, for $X$ any non-equivariant spectrum. It turns out that this fits into an interesting picture, as everything can be determined from the $\ste(1)$-module structure on the non-equivariant cohomology of $X$: 
$$H\mf^{\star}_{\gr}(X) = R(H\F^*(X))_{\gr}$$
as $\E$-modules, for some functor $R : \ste(1)\dashmod mod \rightarrow \E \dashmod mod$.

In the fourth section, we interpret $\mathcal{H}^{\star}(V)$ as a relative extention group
$$ \mathcal{H}^{\star}(V) = \extrel(\F, H\mf^{\star}_{\gr}(BV)) = \extrel(\F, R(H\F^{*}(BV))_{\gr}),$$
in the context of relative homological algebra with respect to the pair $\E(1) \subset \ste(1)$. 
We then use the classical tools of homological algebra (in the relative context) to produce general tools for the computation of the composite $\extrel(\F, R(-))$. The main computational tool is given in Theorem \ref{thm_secR}. This result relates short exact sequences of $\ste(1)$-modules to long exact sequences in $\extrel(\F,R(-))$. 

The fifth section is devoted to the computation of $\extrel(\F,R(-))$ for the most simple $\ste(1)$-modules: the free ones. The result is presented in Corollary \ref{pro_h01libre}. 

The point is that  $\extrel(\F,R(\ste(1)))$ is so small that $\extrel(\F,R(M))$ depends in general mostly on the stable isomorphism class of a $\ste(1)$-module $M$. Expliciting this fact is the subject of section 6.

In section 7 we use the explicit form of the $\ste(1)$-modules $H\F^*(BV)$ to perform the complete computation of $\mathcal{H}^{\star}(V)$. As both $\extrel$ and $R$ are additive functors, the computation is done for each indecomposable non free summand of the $\ste(1)$-module $H\F^*(BV)$. The result of this computation is Theorem \ref{pro_calculh01pn}.

The last section of the paper recollects the machinery of section 1 together with the actual computation of $\mathcal{H}^{\star}(V)$. The proof divides into two steps: first showing that the slice tower for $\kr$ detects at height $2$, and then use the actual computation of $\mathcal{H}^{\star}(V)$ with the detection property of height $2$ to finish the computation.

\tableofcontents

\section{Detection of height $h$} \label{section:detection}

\subsection{Definition} Let $\tri$ be a triangulated category, and $\Sigma : \tri \rightarrow \tri$ its shift functor. Let $\ab$ be an abelian category, and $\ab^{\Z}$ denote the category of $\Z$-graded objects of $\ab$.
For this article, we consider a (covariant) exact functor $(-)^* : \tri \rightarrow \ab^{\Z}$, {\it i.e.} $(-)^*$ sends distinguished triangles into long exact sequences of objects of $\ab$. 

\begin{ex} 
The main examples we have in mind are the stable homotopy category $\sh$ and the functor which associates to a spectrum $X$ the cohomology theory it represents, and in general the equivariant stable homotopy category (see \cite[Theorem 9.4.3]{HPS}).
\end{ex}

\begin{de}
Let $K$ be an object of $\tri$. A {\em tower over $K$} is a diagram of the form
$$ \xymatrix{ \hdots \ar[r]^{e_{n+2}} & k_{n+1} \ar[r]^{e_{n+1}} \ar[drrr]_{f_{n+1}}& k_{n} \ar[r]^{e_{n}} \ar[drr]^{f_{n}} & k_{n-1} \ar[r]^{e_{n-1}} \ar[dr]^{f_{n-1}}& \hdots \\
& & & & K }$$
where the indices run through $\Z$.
\end{de}

\begin{ex}
A $t$-structure on the category $\tri$ give such towers, for example the Postnikov towers in equivariant stable homotopy theory (see \cite[IV,4.1]{GM95}). The slice tower (see \cite{HHR}) also give such an example.
\end{ex}

Let $k_{\bullet}$ be a tower over an object $K$ of $\tri$. One want to use the triangulated structure of the category $\tri$ to compute as much information as possible about the various stages $k_{\bullet}^*$.

\begin{nota}
Complete the map $k_{n+1} \stackrel{e_{n+1}}{\longrightarrow} k_n$ into a distinguished triangle
$$ k_{n+1} \stackrel{e_{n+1}}{\longrightarrow} k_n \stackrel{c_{n}}{\longrightarrow} C_n \stackrel{\delta_{n}}{\longrightarrow} \Sigma k_{n+1}$$
and denote $\theta_n : C_n \rightarrow \Sigma C_{n+1}$ the composite $\delta_nc_n$. The situation is summarized in the following diagram.
$$ \xymatrix@=1,5cm{\hdots \ar[r]^{e_{n+2}} & k_{n+1} \ar[r]^{e_{n+1}}  \ar[d]^{c_{n+1}} & k_{n} \ar[r]^{e_{n}}  \ar[d]^{c_{n}} & k_{n-1} \ar[r]^{e_{n-1}}   \ar[d]^{c_{n-1}} & \hdots \ar[r] &  K \\
 \hdots & C_{n+1} \ar@{-->}[l]^{\theta_{n+1}}   \ar@{-->}[ul]_{\delta_{n+1}} & C_{n} \ar@{-->}[l]^{\theta_{n}}   \ar@{-->}[ul]_{\delta_{n}} & C_{n-1} \ar@{-->}[l]^{\theta_{n-1}}   \ar@{-->}[ul]_{\delta_{n-1}} & \hdots \ar@{-->}[l]^{\theta_{n-2}}&  }$$
where a dotted arrow from $X$ to $Y$ represents a map $X \rightarrow \Sigma Y$.
\end{nota}

\vspace{0,5cm}

For the application, we consider a tower $k_{\bullet}$ over $K$, when the object $K^*$ is completely understood. Our goal is to exhibit a property of the tower which enables us to compute explicitly $k_{\bullet}^*$. We now introduce this key property.

\begin{de}
\begin{enumerate}
\item Let $k_{\bullet}$ be a tower over an object $K$. For an integer $n$, define $$T_n(k_{\bullet}) = Ker(f_n^* : k_n^* \rightarrow K^*).$$
\item We say that $k_{\bullet}$ has the detection of height $h$ and level $n$ (for the functor $(-)^*$, if there is an ambiguity) if the morphism
$$T_n(k_{\bullet}) \rightarrow Coker(e_{n+1}^*e_{n+2}^* \hdots e_{n+h}^* : k^*_{n+h} \rightarrow k_n^*)$$
is injective.
We say that $k_{\bullet}$ has the detection of height $h$ if $k_{\bullet}$ has the detection of height $h$ and level $n$ for all $n \in \Z$.
\end{enumerate}
\end{de}

\begin{nota}
Let $R$ be a ring, $M$ a $R$-module and $x \in R$. Denote
\begin{itemize}
\item $\Gamma_x(M)$ the $x$-torsion sub-module of $M$, \\
\item $co\Gamma_x(M)$ the cokernel of the map $\Gamma_x(M) \inj M$.
\end{itemize}
\end{nota}

\begin{ex} \label{ex:periodic_tower}
Let $k$ be a ring spectrum and $x \in k_d$. Then, multiplication by $x$ gives a tower over $$K = k[x^{-1}] := \underset{n \rightarrow \infty}{hocolim} \ \Sigma^{-n d} k:$$
$$  \xymatrix{ \hdots \ar[r]^{x} & \Sigma^{(n+1)d} k \ar[r]^{x} \ar[drrr] & \Sigma^{nd} k \ar[r]^{x} \ar[drr] & \Sigma^{(n-1)d} k \ar[r]^{x} \ar[dr] & \hdots \\
& & & & K. }$$
Let $X$ be a spectrum, and consider detection properties with respect to the exact functor $[X,-]^*$.
The graded abelian group $T_x( \Sigma^{d}k^*_{\bullet}(X))$ is then exactly a shift of $T_x(k^*(X))$. It is easy to see that detection of height $h$ for this tower is equivalent to $Ker(x^h : k^{*+hd}(X) \rightarrow k^*(X)) \subset \Gamma_x(k^*(X))$ being an isomorphism.
\end{ex}

\begin{lemma}
Let $h \geq 0$ and $n \in \Z$. If a tower $k_{\bullet}$ over $K$ satisfies the detection  of height $h$ at level $n$, then it also satisfies the detection  of height $(h+1)$ at level $n$.
\end{lemma}

\begin{proof}
Consider the commutative diagram 

$$ \xymatrix{ T_x(k_n^*) \ar@{^(->}[r] \ar[dr]^f & k_n^* \ar@{->>}[d] & \\
 & Coker(e^*_{n+1}e_{n+2}^* \hdots e^*_{n+h+1}) \ar@{->>}[r]^{\pi} & Coker(e^*_{n+1}e_{n+2}^* \hdots e^*_{n+h}). } $$
If $k_{\bullet}$ satisfies the detection of height $h$, then $\pi \circ f$ is injective, so $f$ is also injective.
\end{proof}

\begin{rk}
The case $h=1$ already appeared in the literature.
Let $X$ be an object of $\tri$. The property of detection  of height $1$ at level $n$ for the functor $[X,-]^*$ in our sense is equivalent to the detection property of level $n$ detection for $X$ as defined in \cite[Definition 2.2]{Po12}.
\end{rk}

\subsection{Checking detection of height $h$ for low $h$} Our next objective is to provide some tools to prove detection properties. We are mainly interested in $h = 1$ and $2$.
\begin{nota} \label{nota:k_T}
Let $k_{\bullet}$ be a tower over $K$ and denote simply $T_n = T_n(k_{\bullet}^*)$. Our first tool is some natural filtration of $T_n$.
\end{nota}

\begin{de} \label{de:filtrationkernel}
Let $F^0_n(k_{\bullet}^*)$, $F^1_n(k_{\bullet}^*)$ and $F^2_n(k_{\bullet}^*)$, or simply $F^0_n$, $F^1_n$ and $F^2_n$ if it is clear by the context, the sub-quotients corresponding to the following filtration of $T_n$:
$$ \xymatrix{ 0 \ar@{^(->}[r] & Ker(e_n^*)\cap{Im(e_{n+1}^*)} \ar@{^(->}[r]  \ar@{->>}[d] & Ker(e_n^*) \ar@{^(->}[r]  \ar@{->>}[d]  & T_n  \ar@{->>}[d]  \\
 & F^0_n & F^1_n & F^2_n } .$$
\end{de}

\begin{pro} \label{pro:detect_1_2}
With Notation \ref{nota:k_T}, the following properties are satisfied.
\begin{enumerate}
\item The injection $Ker(e_n^*e_{n+1}^*) \subset T_{n+1}$ induces a monomorphism
\begin{eqnarray*}
\iota_{n+1} : F^0_n & \inj & F^2_{n+1} \\
e_{n+1}^*x & \mapsto & [x].
\end{eqnarray*}
\item The tower $k_{\bullet}$ satisfies the detection  of height $1$ at level $n$ if and only if $F^2_n = 0$.
\item The tower $k_{\bullet}$ satisfies the detection  of height $1$ at level $n$ if and only if $F^0_n = 0$.
\item The tower $k_{\bullet}$ satisfies the detection  of height $2$ at level $n$ if and only if the maps $\iota_n$ are isomorphisms.
\end{enumerate}
\end{pro}

\begin{proof}
The map $i_{n+1}$ is well defined by construction of $F^2_{n+1}$. The rest of the first assertion is clear.

Detection of height $1$ is equivalent to $T_n \cap Im(e^*_{n})=0$ for all $n$. Clearly this is equivalent to both $T_n = Ker(e^*_n)$ and $Ker(e_n^*) \cap Im(e_n^*)=0$. This proves the second and third points.

We now prove the last point. First, if some map $i_n$ is not an isomorphism, then let $x \in T_n$ such that $[x] \in F^2_n$ is not in the image of $i_n$. Then $e_{n+1}^* e_n^*(x) \neq 0$, so $k_{\bullet}$ does not satisfies detection of height $2$ at level $n$.
Conversly, if the tower satisfies detection of height $2$, then $\forall n$, any $ x\in F^2_n$ is the image of $e_n^*x \in Ker(e^*_{n-1}) \cap Im(e^*_n)$ by $\iota_n$, so $\iota_n$ is surjective. 
\end{proof}

\subsection{Pushing the detection property along morphisms of towers}
\begin{lemma} \label{lemma:tri_towers}
Let $i_{\bullet}$, $j_{\bullet}$ and $k_{\bullet}$ be three towers. Suppose given two morphisms $f_{\bullet} : i_{\bullet} \rightarrow j_{\bullet}$ and $g_{\bullet} : j_{\bullet} \rightarrow k_{\bullet}$ of towers over some object $K$ of $\tri$. This provides a commutative diagram

$$  \xymatrix@=1,5cm{ \hdots \ar[r]^{e_{n+2}} & i_{n+1} \ar[r]^{e_{n+1}} \ar[d]^{f_{n+1}} & i_{n} \ar[r]^{e_{n}} \ar[d]^{f_{n}} & i_{n-1} \ar[r]^{e_{n-1}} \ar[d]^{f_{n-1}} & \hdots \\
\hdots \ar[r]^{e'_{n+2}} & j_{n+1} \ar[r]^{e'_{n+1}} \ar[d]^{g_{n+1}} & j_{n} \ar[r]^{e'_{n}} \ar[d]^{g_{n}} & j_{n-1} \ar[r]^{e'_{n-1}} \ar[d]^{g_{n-1}} & \hdots \\
\hdots \ar[r]^{e''_{n+2}} & k_{n+1} \ar[r]^{e''_{n+1}} & k_{n} \ar[r]^{e''_{n}} & k_{n-1} \ar[r]^{e''_{n-1}} & \hdots  }$$
where the columns are cofiber sequences.

If the tower $i_{\bullet}$ satisfies detection  of height $h_i$, and the tower $k_{\bullet}$ satisfies detection of height $h_k$, then the tower $j_{\bullet}$ satisfies detection  of height $(h_i+h_k)$.
\end{lemma}

\begin{proof}
The proof is a diagram chase.

Let $x \in T_n(j_{\bullet}^*)$. Then $g_n^*(x) \in T_n(k_{\bullet}^*)$, so ${e''}_{n-h_k+1}^* \hdots {e''}_n^* (g_n^*(x)) = 0$ by the detection of height $h_k$ for the tower $k_{\bullet}$. By exactness of $(-)^*$, the element ${e'}_{n-h_k+1}^* \hdots {e'}_n^* (x)$ comes from $i_{n-h_k+1}^*$. Now, the morphism $f_{\bullet}$ is over $K$, so $x \in T_n(j_{\bullet}^*)$ implies that ${e'}_{n-h_k+1}^* \hdots {e'}_n^* (x)$ comes from an element $y \in T_{n-h_k}(i_{\bullet}^*)$. By the detection of height $h_i$ for the tower $i_{\bullet}^*$, one has $e_{n-h_k-h_i+1}^* \hdots e_{n-h_k}^* (y) = 0.$

By commutativity of the diagram given in the lemma, we have $${e'}_{n-h_k-h_i+1}^* \hdots {e'}_n^* (x) = 0.$$ This concludes the proof.
\end{proof}

The following proposition is an important consequence of lemma \ref{lemma:tri_towers}. The situation considered in the proposition is inspired by the isotropy separation sequence in equivariant stable homotopy theory.

\begin{pro} \label{pro:suchgreatheights}
Let $E, \tilde{E} : \tri \rightarrow \tri$ two exact functors and $$E \rightarrow id_{\tri} \rightarrow \tilde{E}$$ a natural distinguished triangle. 
Let $k_{\bullet}$ be a tower over $K$.
Suppose moreover that $EK \rightarrow K$ is the identity, and that $\tilde{E}e_n$ is trivial for all $n \in \Z$.
Then if the tower $E(k_{\bullet})$ satisfies the detection  of height $h$, $k_{\bullet}$ satisfies the detection of height $(h+1)$.
\end{pro}

\begin{proof}
Consider the tower $\tilde{E}k_{\bullet}$ as a tower over $K$ with the trivial maps $\tilde{E}k_{n} \rightarrow K$. Then the morphism of towers $k_{\bullet} \rightarrow \tilde{E}k_{\bullet}$ induced by the natural transformation $id_{\tri} \rightarrow \tilde{E}$ is a morphism over $K$ because $\tilde{E}e_n =0$ for all $n$.

Now, $\tilde{E}k_{\bullet}$ satisfies trivially the detection  of height $1$ because $T_n(\tilde{E}k_{\bullet}^*) = \tilde{E}k_{n}^* = Coker(\tilde{E}e_{n+1}^*)$. The proposition is now a consequence of lemma \ref{lemma:tri_towers}.
\end{proof}

\subsection{Detection as a computational tool} 
We now show how the detection of height $2$ for a tower $k_{\bullet}$ helps to gain control over the objects $k_{\bullet}^*$.
Recall that, for all $n \in \Z$, we have have defined a filtration of $k_{n}^*$ of length four in the last section.

\begin{de} \label{de:filtrationcotors}
Let $\phi_n$ denote $Im(f_n^* : k_n^* \rightarrow K^*)$.
\end{de}

We will now give as much information as possible about each sub-quotients of $k_{n}^*$. Although the computation is not complete in the general case, it is sufficient for our application.

\begin{thm} \label{thm:chaincplx}
Let $k_{\bullet}$ be a tower over $K$ and $n \in \Z$.
\begin{enumerate}
\item The map $c_n$ induces an isomorphism $$F^1_n \overset{\sim}{\longrightarrow} Im(\theta_{n-1}^* : C_{n-1}^* \rightarrow (\Sigma C_n)^*).$$
\item There is a chain complex $$ \xymatrix@=2cm{ F_n^2 \ar@{^(->}[r]^{\overline{c_n^*}} &  \frac{Ker(\theta_n^*)}{Im(\theta_{n-1}^*)} \ar@{->>}[r]^{\overline{\delta_n^*}} & \Sigma F^0_{n+1},}$$
where the first morphism is induced by $c_n$ and the second one is induced by $\delta_n$, and whose homology is isomorphic to $ \frac{\phi_n}{\phi_{n+1}}$.
\item Suppose that $k_{\bullet}$ satisfies detection of height $2$. Then the previous chain complex is isomorphic to
$$ \xymatrix@=2cm{ F_n^2 \ar@{^(->}[r]^{\overline{c_n^*}} &  \frac{Ker(\theta_n^*)}{Im(\theta_{n-1}^*)} \ar@{->>}[r]^{\iota_{n+2}\overline{\delta_n^*}} & \Sigma F^2_{n+2}.}$$
\end{enumerate}
\end{thm}

\begin{proof}
\begin{enumerate}
\item By exactness, $Ker(e_n^*) = Im(\delta_{n-1}^*)$, giving a morphism $$Ker(e_{n}^*) = Im(\delta_{n-1}^*) \overset{c_{n-1}^*}{\longrightarrow} Im(\theta^*_{n-1}),$$ which is surjective by definition. Its kernel is precisely $Im(e_{n+1}^*) \cap Ker(e_n^*)$.
\item By exactness, $Ker(e_n^*)  = Im( \delta_{n-1}^*)$, so $c_{n}^*$ gives a well-defined map
$$ \frac{T_n}{Ker(e_n^*)} \overset{\overline{c_n^*}}{\longrightarrow} \frac{Ker(\theta_n^*)}{Im(\theta_{n-1}^*)}.$$
The second arrow is well-defined because $\delta_n^* \circ \theta_{n-1}^* = 0$.
To conclude the proof of $(2)$, we will construct a map $$ \psi : \Phi_n/ \Phi_{ n+1} \rightarrow Ker( \overline{ \delta_n^*} ) / Im(  \overline{c_n^*})$$ and show it is injective and surjective.

Let $[f_{ n}^*(x)] \in \Phi_n/ \Phi_{ n+1}$ where $x \in k_n^*(X)$. By construction, $c_n^*(x) \in Ker( \delta_n^*) \subset Ker( \theta_n^*)$, so $c_n^*(x)$ defines a class $[c_n^*(x)] \in  \frac{Ker( \theta_n^*)}{Im( \theta_{ n-1}^*)}$. Moreover, by exactness, $ \delta_n^* \circ c_n^* = 0$, so $[c_n^*(x)] \in Ker( \overline{ \delta_n}^*)$. Define $ \psi( [f_n^*(x)]) = [c_n^*(x)] \in Ker( \overline{ \delta_n}^* ) / Im(  \overline{c_n}^*)$. \\
This defines a morphism because for all $t \in T_n$ and $y \in \Phi_{n+1}$, we have $c_n^*(t+ e_n^*y) = c_n^*(t) \in Im(  \overline{c_n^*})$.
\begin{itemize}
 \item Injectivity: let $[f_n^*(x)] \in \Phi_n/ \Phi_{n+1}$ such that $c_n^*(x) \in Im( \overline{c_n}^*)$. Then $ \exists t \in T_n$ and $ y \in k_{n+1}^*$ such that $x = t + e_{n+1}^*y$. Thus $[f_n^*(x)] = [f_n^*(t + e_{n+1}^*y)] = [f_n^*(e_{n+1}^*y)] = 0$.
\item Surjectivity: let $[y] \in Ker( \overline{ \delta_n}^*)/ Im(  \overline{c_n}^*)$, where $y \in Ker( \delta_n^*)$. By exactness $ \exists x \in k_n^*$ such that $c_n^*(x) = y$. Then $[f_n^*(x)]$ is a preimage of $[y]$.
\end{itemize}
\item This is a consequence of $(2)$ and the isomorphism provided by the third point of Proposition \ref{pro:detect_1_2}.
\end{enumerate}
\end{proof}

\begin{rk}
Let $k_{\bullet}$ be a tower obtained as in Example \ref{ex:periodic_tower}, satisfying the detection of height $2$. Then there is an isomorphism $$ \Sigma F^2_{n+2} \cong \Sigma^{1+2|x|} F^2_n.$$  Thus the third point of Theorem \ref{thm:chaincplx} gives a strong condition on $\frac{Ker(\theta_n^*)}{Im(\theta_{n-1}^*)}$ in this case.
\end{rk}

\section{The slice tower for $K$-theory with reality} \label{section:slicetower}

In this section, we recall the tower we are interested in, and give an interpretation of the various constructions of section \ref{section:detection}. The point of this section is also to identify explicitly the central term of the chain complex of Theorem \ref{thm:chaincplx} when the tower under consideration is the slice tower of periodic K-theory with reality, with respect to the functor $[BV,-]_{\gr}^{\star}: \gr\sh \rightarrow \ab^{RO(\gr)}$, which associates to a $Q$-equivariant spectrum $E$ the $RO(\gr)$-abelian group $E_{\gr}^{\star}(BV)$. This is the abelian group $\mathcal{H}^{\star}(V)$ of Notation \ref{nota:hv}.

\subsection{Conventions for $\gr$-equivariant stable homotopy theory}

\subsection{Equivariant stable homotopy theory and Mackey functors} \label{sub:conventions}

We refer to \cite{MM} for the constructions and definitions in the equivariant stable homotopy category. When $X,Y$ are equivariant objects, the symbol $[X,Y]^{\star}_{\gr}$ will always denote the abelian group of $\gr$-equivariant maps, as opposed to the more structured object $[X,Y]^{\star}$, which is a Mackey functor as we will see in this section.

\begin{nota}
Let $\gr\mathcal{T}$ be the category of $\gr$-spaces and $\gr\mathcal{H}$ its homotopy category with respect to the usual fine model structure, where weak equivalences (resp. cofibrations) are maps which are non-equivariant weak equivalences (resp. cofibrations) on all fixed points. This category is called {\it~the $\gr$-equivariant homotopy category}.
\end{nota}

We now define the spheres and suspension functors we need to set up the equivariant stable homotopy category.

\begin{de} \label{de:repsphere}
A {\em representation sphere} is a pointed $\gr$-space of the form $S^W$ (the one point compactification of $W$), for $W$ an orthogonal representation of $\gr$.
\end{de}

Forcing the suspension functors $\Sigma^W = S^W \wedge (-)$ to be invertible up to weak equivalence, for all orthogonal representation $W$, yields the category of $\gr$-equivariant spectra, denoted $\gr\Sp$.

Since the functors $S^W \wedge (-)$ are now weakly invertible in the equivariant stable homotopy category, one point compactification assembles to a nice functor $S^{(-)}$, from $RO(\gr)$ to $\gr\sh$.

\begin{nota}
Let $\mathcal{M}$ be the category of Mackey functors and natural transformations between them.
Let $\mathcal{M}^{RO(\gr)}$ be the category of $RO(\gr)$-graded Mackey functors.
\end{nota}

The interested reader can consult \cite{FL04} for Mackey functors and the relationship between Mackey functors and equivariant stable homotopy theory. In the particular case where the ambient group is $\gr$, we have the following convenient graphical representation.

\begin{nota} \label{nota:mackey}
Let $M$ be a Mackey functor. Let $\theta_M :  M_e \rightarrow M_e$ the action of the non-trivial element of $\gr$ on $M_e$, we represent $M$ by the following diagram:
$$\xymatrix{ M_{\gr}  \ar@/^/[d]^{ \rho_M } \\ 
M_e. \ar@/^/[u]^{\tau_M} } $$
Observe that, since $\theta_M = \rho_M\tau_M -1$ this diagram suffices to determine $M$, and in particular the action of $\gr$ on $M_e$.
\end{nota}

Let $X,Y$ be $\gr$-spectra. The stable homotopy classes of morphisms $[X,Y]_{\star}$ is naturally a $RO(\gr)$-graded Mackey functor.

With the previous notations, $[-,-]_{\star}$ defines a functor
$$\gr\sh^{op} \times \gr\sh \rightarrow \mathcal{M}^{RO(\gr)}.$$

\begin{de}
Let $E$ be a $\gr$-spectrum. The {\it~homotopy Mackey functor} of $E$ is by definition the $RO(\gr)$-graded Mackey functor $\underline{\pi}_{\star}(E) := [S^0,E]_{\star}$.
\end{de}

\begin{pro} \label{pro:functorH}
The $\gr$-equivariant Postnikov tower defines a $t$-structure on the $\gr$-equivariant stable homotopy category whose heart is isomorphic to the category $\mathcal{M}$ of Mackey functors for the group $\gr$. In particular, one has an Eilenberg-MacLane functor
$$H : \mathcal{M} \rightarrow \gr\sh$$
which sends a short exact sequences of Mackey functors to a distinguished triangle of $\gr$-equivariant spectra.
\end{pro}

\begin{proof}
This proposition summarize the results of \cite[Proposition I.7.14]{LMS} and \cite[Theorem 1.13]{Le95} in the particular case of the group with two elements.
\end{proof}

\begin{de}
Denote $\mz$ the Mackey functor
$$\xymatrix{ \Z \ar@/^/[d]^{ =} \\ 
\Z \ar@/^/[u]^{ 2} } $$
and $\mf$ the Mackey functor
 $$\xymatrix{ \F \ar@/^/[d]^{ =} \\ 
\F. \ar@/^/[u]^{0} } $$
\end{de}

\begin{rk}
The Mackey functors whose restrictions are monomorphisms in general, and $\mf$, $\mz$ in particular, play a special role in this context. One reason is that equivariant Eilenberg-MacLane spectra with coefficients in constant Mackey functors are exactly the $0th$-slices (see \cite[Proposition 4.47]{HHR}). Moreover $H\mz = Cofiber(P^1(S^0) \rightarrow S^0)$ with the notations of \textit{loc cit}, by \cite[Corollary 4.51]{HHR}.
\end{rk}

Proposition \ref{pro:functorH} provides a distinguished triangle
$$ \xymatrix{ H\mz \ar[r]^{\times 2} & H\mz \ar[d]^p \\
 & H\mf \ar@{-->}[ul]^{\partial}}$$

\begin{de} \label{de:qo}
Denote $\qo : H\mf \rightarrow H\mf$ the composite $\Sigma p \circ \partial$. It defines a cohomology operation satisfying $\qo^2=0$.
\end{de}

\subsection{Connective $K$-theory with reality and equivariant Milnor operations}

Recall that Atiyah \cite{At66} defined a $\gr$-equivariant cohomology theory called $K$-theory with reality, which is represented by the $\gr$-equivariant spectrum indexed over a complete universe denoted $K\R$. Let $k\R$ be its connective cover.

Dugger studies in \cite[p.21]{Du03} the slice tower for $K\R$. His results are summarized in the following proposition.

\begin{pro}[\emph{\cite[Theorem 1.5]{Du03}}] \label{pro:de_q1}
There is an equivariant lift of the complex Bott map $v_1$ in degree $(1+ \alpha)$ such that the slice tower of $K\R$ is the following tower over $K\R$:
$$ \xymatrix@=0.9cm{\hdots \ar[r]^{v_1 \quad \quad} & \Sigma^{(n+1)(1+\alpha)}k\R \ar[r]^{v_1}  \ar[d] & \Sigma^{n(1+\alpha)}k\R \ar[r]^{v_1}   \ar[d] & \Sigma^{(n-1)(1+\alpha)}k\R \ar[r]^{\quad \quad v_1}    \ar[d] & \hdots  \\
 \hdots & \Sigma^{(n+1)(1+\alpha)}H\underline{\Z} \ar@{-->}[l]^{\overline{\qi} \quad \quad}   \ar@{-->}[ul]_{\delta_{n+1}} & \Sigma^{(n)(1+\alpha)}H\underline{\Z} \ar@{-->}[l]^{\overline{\qi} }   \ar@{-->}[ul]_{\delta_{n}} & \Sigma^{(n-1)(1+\alpha)}H\underline{\Z} \ar@{-->}[l]^{\overline{\qi} }   \ar@{-->}[ul]_{\delta_{n-1}} & \hdots \ar@{-->}[l]^{\quad \quad \overline{\qi} }&  }$$
where $H\underline{\Z}$ is the Eilenberg-MacLane spectrum with coefficients the constant Mackey functor $\Z$, and $\overline{\qi} $ is some degree $2+ \alpha$ map.
\end{pro}

\begin{lemma}
There is a cohomology operation $\qi$ of degree $2+\alpha$ such that $\overline{\qi}$ is the unique integral lift of $\qi$. Moreover $\qo$ and $\qi$ generates an exterior sub-algebra of the $\gr$-equivariant Steenrod algebra.
\end{lemma}

\begin{proof}
If such a lift exists, then it is unique since there is only one nonzero operation in this degree by the results of Hu and Kriz \cite{HK01}.

Now, the map $\overline{\qi} : H\mz \rightarrow \Sigma^{2+\alpha} H\mz$ commutes with multiplication by two, so there exists $\qi$ completing
$$ \xymatrix@=1.5cm{ H\mz \ar[r]^{\times 2} \ar[d]^{\overline{\qi}} & H\mz \ar[r]^p  \ar[d]^{\overline{\qi}} & H\mf \ar[r]^{\partial} & \Sigma H\mz \ar[d]^{\overline{\qi}} \\  \Sigma^{2+\alpha} H\mz \ar[r]^{\times 2} & \Sigma^{2+\alpha} H\mz \ar[r]^p  & \Sigma^{2+\alpha} H\mf \ar[r]^{\partial} & \Sigma^{3+\alpha}  H\mz }$$
into a map of distinguished triangles. This proves the first point. The commutation of the squares involving the maps $p$ and $\partial$ into the resulting commutative diagram concludes the proof.
\end{proof}

\begin{de}
Define $\E$ the subalgebra  $ \Lambda_{\F}(\qo, \qi)$ of the $\gr$-equivariant Steenrod algebra. Define also $\Lambda_0$ and $\Lambda_1$ as the exterior algebras over $\F$ generated by $\qo$ and $\qi$ respectively.
\end{de}

Let $V$ be an elementary abelian $2$-group. One wants to understand $k\R^{\star}(BV)$ as a Mackey functor. In order to use the results of section \ref{section:detection}, {\em i.e.} to prove a detection property and to make the actual computation, we need to understand the action of $\overline{\qi}$ on $H\mz^{\star}_{\gr}(BV)$, and in particular the Margolis homology with respect to $\overline{\qi}$: $$ \frac{Ker(\overline{\qi} : H\mz^{\star}_{\gr}(BV) \rightarrow H\mz^{\star+2+\alpha}_{\gr}(BV))}{Im(\overline{\qi} : H\mz^{\star-2-\alpha_{\gr}}(BV) \rightarrow H\mz^{\star}_{\gr}(BV))}.$$

\begin{nota} \label{nota:hv}
Denote $\mathcal{H}^{\star}_{\gr}(V)$ the abelian group $$ \frac{Ker(\overline{\qi} : H\mz^{\star}(BV)_{\gr} \rightarrow H\mz^{\star+2+\alpha}(BV)_{\gr})}{Im(\overline{\qi} : H\mz^{\star-2-\alpha}(BV)_{\gr} \rightarrow H\mz^{\star}(BV)_{\gr})}.$$
\end{nota}

Now, the multiplication by $2$ on $BV$ is nulhomotopic, so the long exact sequence
$$ \xymatrix{ \hdots \ar[r] & H\mz^{\star}_{\gr}(BV) \ar[r]^{\times 2} & H\mz^{\star}_{\gr}(BV) \ar[r] & H\mf^{\star}_{\gr}(BV) \ar[r] & \hdots}$$
is split and $H\mz^{\star}_{\gr}(BV) = Ker(\qo : H\mf^{\star}_{\gr}(BV) \rightarrow H\mf^{\star+1}_{\gr}(BV))$. Thus $\mathcal{H}^{\star}(V)$ depends only on the $\E$-module structure of $H\mf^{\star}_{\gr}(BV)$. The determination of this structure is our next objective.

\section{The action of $ \E$ on the cohomology of non-equivariant spectra}

This section is independent from the rest of the paper. The aim here is to understand the action of the equivariant Milnor derivations $\qo,\qi$ on the cohomology of $\gr$-spectra, and especially to have an explicit description of this action for spectra induced from non-equivariant ones.

The main results of this section are the Cartan formulae in equivariant cohomology in Corollary \ref{cor:cartan}, and the comparison between the action of the Steenrod algebra on the cohomology of a non-equivariant spectrum and the action of the equivariant Steenrod algebra on its equivariant cohomology in Theorem \ref{thm:formula_qoi}.

\subsection{Recollections and observations about the coefficient ring for equivariant cohomology}

In order to be self-contained, we recall the structure of the coefficient ring for $H\mf$-cohomology. The computation is done in \cite{FL04} where it is attributed to Stong. The case we are interested in here is also given in \cite[p.371]{HK01}.

\begin{nota}
There are two particular elements in the cohomology of the point: the Euler class of the map $S^0 \inj S^{\alpha}$, denoted $a$, and the orientation class for the sign representation, denoted $\sigma^{-1}$. These elements are in $H\mf^{\alpha}$ and $H\mf^{\alpha -1}$ respectively.
\end{nota}

\begin{pro} \label{pro_cohompoint}
 The $RO( \gr)$-graded Mackey functor $H \mf_{ \star}$ is represented in the following picture. The symbol $ \bullet$ stands for the Mackey functor $$ \xymatrix{ \gr \ar@/^/[d] \\ 
0, \ar@/^/[u] } $$ 
and $L$ stands for $$\xymatrix{ \gr \ar@/^/[d]^{ 0} \\ 
\gr. \ar@/^/[u]^{=} } $$
A vertical line represents the product with the Euler class $a$. This product induces one of the following Mackey functor maps:
\begin{itemize}
\item the identity of $ \bullet$, 
\item the unique non-trivial morphism $ \mf \rightarrow \bullet$,
\item the unique non-trivial morphism $ \bullet \inj \mf$.
\end{itemize}
 
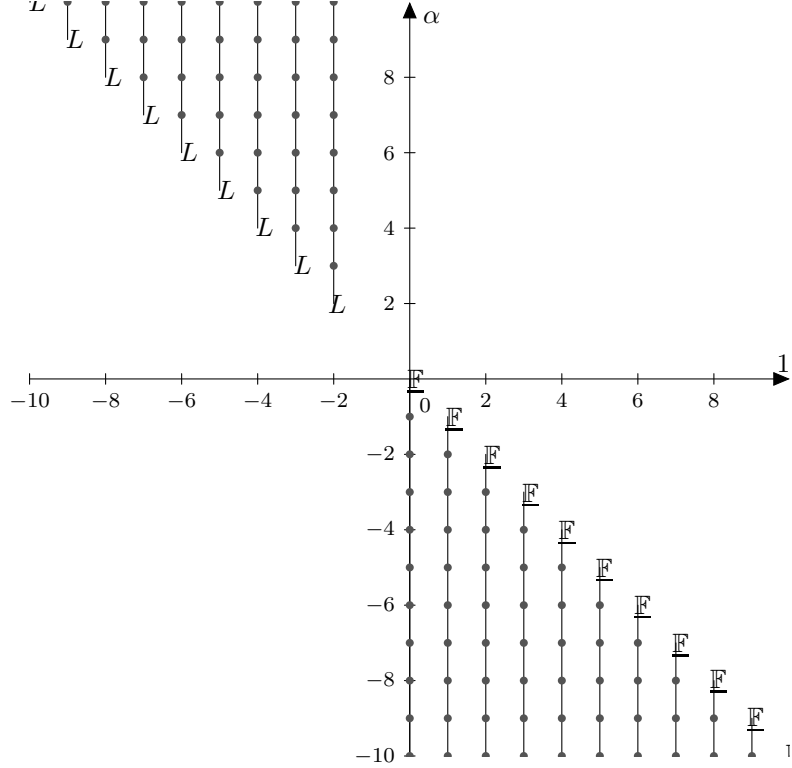
 \begin{center}
\definecolor{cqcqcq}{rgb}{0.75,0.75,0.75}
\definecolor{qqqqff}{rgb}{0.33,0.33,0.33}


\begin{tikzpicture}[line cap=round,line join=round,>=triangle 45,x=0.5cm,y=0.5cm]
\draw[->,color=black] (-10,0) -- (10,0);
\foreach \x in {-10,-8,-6,-4,-2,2,4,6,8}
\draw[shift={(\x,0)},color=black] (0pt,2pt) -- (0pt,-2pt) node[below] {\footnotesize $\x$};
\draw[->,color=black] (0,-10) -- (0,10);
\foreach \y in {-10,-8,-6,-4,-2,2,4,6,8}
\draw[shift={(0,\y)},color=black] (2pt,0pt) -- (-2pt,0pt) node[left] {\footnotesize $\y$};
\draw[color=black] (0pt,-10pt) node[right] {\footnotesize $0$};
\clip(-10,-10) rectangle (10,10);
\draw (-0.31,0.5) node[anchor=north west] {$\underline{ \mathbb{F}}$};
\draw (0.7,-0.51) node[anchor=north west] {$\underline{ \mathbb{F}}$};
\draw (1.7,-1.52) node[anchor=north west] {$\underline{ \mathbb{F}}$};
\draw (2.71,-2.52) node[anchor=north west] {$\underline{ \mathbb{F}}$};
\draw (3.67,-3.51) node[anchor=north west] {$\underline{ \mathbb{F}}$};
\draw (4.65,-4.51) node[anchor=north west] {$\underline{ \mathbb{F}}$};
\draw (5.66,-5.49) node[anchor=north west] {$\underline{ \mathbb{F}}$};
\draw (6.66,-6.5) node[anchor=north west] {$\underline{ \mathbb{F}}$};
\draw (7.65,-7.46) node[anchor=north west] {$\underline{ \mathbb{F}}$};
\draw (8.63,-8.47) node[anchor=north west] {$\underline{ \mathbb{F}}$};
\draw (9.64,-9.45) node[anchor=north west] {$\underline{ \mathbb{F}}$};
\draw (-12.25,12.46) node[anchor=north west] {$L$};
\draw (-11.24,11.45) node[anchor=north west] {$L$};
\draw (-10.28,10.49) node[anchor=north west] {$L$};
\draw (-9.3,9.51) node[anchor=north west] {$L$};
\draw (-8.29,8.5) node[anchor=north west] {$L$};
\draw (-7.29,7.5) node[anchor=north west] {$L$};
\draw (-6.3,6.51) node[anchor=north west] {$L$};
\draw (-5.32,5.51) node[anchor=north west] {$L$};
\draw (-4.31,4.53) node[anchor=north west] {$L$};
\draw (-3.31,3.52) node[anchor=north west] {$L$};
\draw (-2.4,2.49) node[anchor=north west] {$L$};
\draw (0,0) -- (0,-10);
\draw (1,-1) -- (1,-10);
\draw (2,-2) -- (2,-10);
\draw (3,-3) -- (3,-10);
\draw (4,-4) -- (4,-10);
\draw (5,-5) -- (5,-10);
\draw (6,-6) -- (6,-10);
\draw (7,-7) -- (7,-10);
\draw (8,-8) -- (8,-10);
\draw (9,-9) -- (9,-10);
\draw (10,-10) -- (10,-10);
\draw (-2,2) -- (-2,10);
\draw (-3,3) -- (-3,10);
\draw (-4,4) -- (-4,10);
\draw (-5,5) -- (-5,10);
\draw (-6,6) -- (-6,10);
\draw (-7,7) -- (-7,10);
\draw (-8,8) -- (-8,10);
\draw (-9,9) -- (-9,10);
\draw (-10,10) -- (-10,10);
\draw (-11,11) -- (-11,10);
\draw (9.4,0.9) node[anchor=north west] {$1$};
\draw (0.1,10) node[anchor=north west] {$\alpha$};
\begin{scriptsize}
\fill [color=qqqqff] (1,-2) circle (1.5pt);
\fill [color=qqqqff] (1,-3) circle (1.5pt);
\fill [color=qqqqff] (1,-4) circle (1.5pt);
\fill [color=qqqqff] (1,-5) circle (1.5pt);
\fill [color=qqqqff] (1,-6) circle (1.5pt);
\fill [color=qqqqff] (1,-7) circle (1.5pt);
\fill [color=qqqqff] (1,-8) circle (1.5pt);
\fill [color=qqqqff] (1,-9) circle (1.5pt);
\fill [color=qqqqff] (1,-10) circle (1.5pt);
\fill [color=qqqqff] (1,-11) circle (1.5pt);
\fill [color=qqqqff] (0,-1) circle (1.5pt);
\fill [color=qqqqff] (0,-2) circle (1.5pt);
\fill [color=qqqqff] (0,-3) circle (1.5pt);
\fill [color=qqqqff] (0,-4) circle (1.5pt);
\fill [color=qqqqff] (0,-5) circle (1.5pt);
\fill [color=qqqqff] (0,-6) circle (1.5pt);
\fill [color=qqqqff] (0,-7) circle (1.5pt);
\fill [color=qqqqff] (0,-8) circle (1.5pt);
\fill [color=qqqqff] (0,-9) circle (1.5pt);
\fill [color=qqqqff] (0,-10) circle (1.5pt);
\fill [color=qqqqff] (0,-11) circle (1.5pt);
\fill [color=qqqqff] (2,-3) circle (1.5pt);
\fill [color=qqqqff] (2,-4) circle (1.5pt);
\fill [color=qqqqff] (2,-5) circle (1.5pt);
\fill [color=qqqqff] (2,-6) circle (1.5pt);
\fill [color=qqqqff] (2,-7) circle (1.5pt);
\fill [color=qqqqff] (2,-8) circle (1.5pt);
\fill [color=qqqqff] (2,-9) circle (1.5pt);
\fill [color=qqqqff] (2,-10) circle (1.5pt);
\fill [color=qqqqff] (2,-11) circle (1.5pt);
\fill [color=qqqqff] (3,-4) circle (1.5pt);
\fill [color=qqqqff] (3,-5) circle (1.5pt);
\fill [color=qqqqff] (3,-6) circle (1.5pt);
\fill [color=qqqqff] (3,-7) circle (1.5pt);
\fill [color=qqqqff] (3,-8) circle (1.5pt);
\fill [color=qqqqff] (3,-9) circle (1.5pt);
\fill [color=qqqqff] (3,-10) circle (1.5pt);
\fill [color=qqqqff] (3,-11) circle (1.5pt);
\fill [color=qqqqff] (4,-5) circle (1.5pt);
\fill [color=qqqqff] (4,-6) circle (1.5pt);
\fill [color=qqqqff] (4,-7) circle (1.5pt);
\fill [color=qqqqff] (4,-8) circle (1.5pt);
\fill [color=qqqqff] (4,-9) circle (1.5pt);
\fill [color=qqqqff] (4,-10) circle (1.5pt);
\fill [color=qqqqff] (4,-11) circle (1.5pt);
\fill [color=qqqqff] (5,-6) circle (1.5pt);
\fill [color=qqqqff] (5,-7) circle (1.5pt);
\fill [color=qqqqff] (5,-8) circle (1.5pt);
\fill [color=qqqqff] (5,-9) circle (1.5pt);
\fill [color=qqqqff] (5,-10) circle (1.5pt);
\fill [color=qqqqff] (5,-11) circle (1.5pt);
\fill [color=qqqqff] (6,-7) circle (1.5pt);
\fill [color=qqqqff] (6,-8) circle (1.5pt);
\fill [color=qqqqff] (6,-9) circle (1.5pt);
\fill [color=qqqqff] (6,-10) circle (1.5pt);
\fill [color=qqqqff] (6,-11) circle (1.5pt);
\fill [color=qqqqff] (7,-8) circle (1.5pt);
\fill [color=qqqqff] (7,-9) circle (1.5pt);
\fill [color=qqqqff] (7,-10) circle (1.5pt);
\fill [color=qqqqff] (7,-11) circle (1.5pt);
\fill [color=qqqqff] (8,-9) circle (1.5pt);
\fill [color=qqqqff] (8,-10) circle (1.5pt);
\fill [color=qqqqff] (8,-11) circle (1.5pt);
\fill [color=qqqqff] (9,-10) circle (1.5pt);
\fill [color=qqqqff] (9,-11) circle (1.5pt);
\fill [color=qqqqff] (10,-11) circle (1.5pt);
\fill [color=qqqqff] (-2,11) circle (1.5pt);
\fill [color=qqqqff] (-2,10) circle (1.5pt);
\fill [color=qqqqff] (-2,9) circle (1.5pt);
\fill [color=qqqqff] (-2,8) circle (1.5pt);
\fill [color=qqqqff] (-2,7) circle (1.5pt);
\fill [color=qqqqff] (-2,6) circle (1.5pt);
\fill [color=qqqqff] (-2,5) circle (1.5pt);
\fill [color=qqqqff] (-2,4) circle (1.5pt);
\fill [color=qqqqff] (-2,3) circle (1.5pt);
\fill [color=qqqqff] (-3,11) circle (1.5pt);
\fill [color=qqqqff] (-3,10) circle (1.5pt);
\fill [color=qqqqff] (-3,9) circle (1.5pt);
\fill [color=qqqqff] (-3,8) circle (1.5pt);
\fill [color=qqqqff] (-3,7) circle (1.5pt);
\fill [color=qqqqff] (-3,6) circle (1.5pt);
\fill [color=qqqqff] (-3,5) circle (1.5pt);
\fill [color=qqqqff] (-3,4) circle (1.5pt);
\fill [color=qqqqff] (-4,11) circle (1.5pt);
\fill [color=qqqqff] (-4,10) circle (1.5pt);
\fill [color=qqqqff] (-4,9) circle (1.5pt);
\fill [color=qqqqff] (-4,8) circle (1.5pt);
\fill [color=qqqqff] (-4,7) circle (1.5pt);
\fill [color=qqqqff] (-4,6) circle (1.5pt);
\fill [color=qqqqff] (-4,5) circle (1.5pt);
\fill [color=qqqqff] (-5,11) circle (1.5pt);
\fill [color=qqqqff] (-5,10) circle (1.5pt);
\fill [color=qqqqff] (-5,9) circle (1.5pt);
\fill [color=qqqqff] (-5,8) circle (1.5pt);
\fill [color=qqqqff] (-5,7) circle (1.5pt);
\fill [color=qqqqff] (-5,6) circle (1.5pt);
\fill [color=qqqqff] (-6,11) circle (1.5pt);
\fill [color=qqqqff] (-6,10) circle (1.5pt);
\fill [color=qqqqff] (-6,9) circle (1.5pt);
\fill [color=qqqqff] (-6,8) circle (1.5pt);
\fill [color=qqqqff] (-6,7) circle (1.5pt);
\fill [color=qqqqff] (-7,11) circle (1.5pt);
\fill [color=qqqqff] (-7,10) circle (1.5pt);
\fill [color=qqqqff] (-7,9) circle (1.5pt);
\fill [color=qqqqff] (-7,8) circle (1.5pt);
\fill [color=qqqqff] (-8,11) circle (1.5pt);
\fill [color=qqqqff] (-8,10) circle (1.5pt);
\fill [color=qqqqff] (-8,9) circle (1.5pt);
\fill [color=qqqqff] (-9,11) circle (1.5pt);
\fill [color=qqqqff] (-9,10) circle (1.5pt);
\fill [color=qqqqff] (-10,11) circle (1.5pt);
\end{scriptsize}
\end{tikzpicture}

\end{center}

\end{pro}

We finish this subsection by a lemma about Mackey functors, relating the $ \Z[a]$-module structure on $([-,-]^{\star})_{\gr}$ with the $RO( \gr)$-graded Mackey functor structure of $ [-,-]^{\star}$.

\begin{lemma} \label{lemma_mackey_a}
 Let $E$ be a $ \gr$-spectrum such that $2a$ acts trivially on $E_{\star}^{\gr}$.
\begin{enumerate}
 \item $Im(a)= Ker( \rho)$ where $ \rho $ is the restriction of the Mackey functor $E_{ \star}$.
\item $ Ker(a) = Im( \tau)$ where $ \tau $ is the transfer.
\end{enumerate}
\end{lemma}

\begin{proof}
 These are consequences of the existence of the long exact sequence associated to the distinguished triangle $$ \gr_+ \rightarrow S^0 \rightarrow S^{ \alpha}$$
in the stable $\gr$-equivariant category.
\begin{enumerate}
 \item Apply the exact functor $ [ -, \Sigma^{ - \star} E]_{e}$ to the triangle. We have:
$$ \xymatrix{  [S^{ \alpha}, \Sigma^{- \star} E]_{e} \ar[r] \ar@{ = }[d] & [S^0, \Sigma^{- \star} E]_{e} \ar[r] \ar@{ = }[d] & [ \gr_+, \Sigma^{- \star} E]_{e} \ar@{ = }[d] \\
\underline{ \pi}_{ \star+ \alpha}(E)_{e} \ar[r]^a & \underline{ \pi}_{ \star}(E)_{e} \ar[r]^{ \rho} & \underline{ \pi}_{ \star}(E)_{\gr} \\ }$$
where lines are exact. The first point follows.
\item Apply the exact functor $ [ S^{ \star}, (-) \wedge E]_{e}$ to the triangle. We have:
$$ \xymatrix{  [S^{ \star}, \gr_+ \wedge E]_{e} \ar[r] \ar@{ = }[d] & [S^{ \star}, E]_{e} \ar[r] \ar@{ = }[d] & [ \Sigma^{ \star - \alpha}, E]_{e} \ar@{ = }[d] \\
\underline{ \pi}_{ \star}(E)_{\gr} \ar[r]^{ \tau} & \underline{ \pi}_{ \star}(E)_{e} \ar[r]^{a} & \underline{ \pi}_{ \star}(E)_{e} \\ }$$
where lines are exact. The second point follows.
\end{enumerate}
\end{proof}

\subsection{Cartan formulae in $\E$ for the $H\mf^{\star}$-module structure}

Recall that Hu and Kriz computed a presentation of the $\gr$-equivariant modulo $2$ dual Steenrod algebra $\ste_{\star} := (H\mf_{\star}^{\gr}H\mf)$ from which we can deduce the following result.

\begin{pro} \label{pro_hypolibrepourhfd}
 The  $ H \mf_{ \star}^{\gr}$-module $(H \mf_{ \star}^{\gr} H \mf)$ is free over $$ \mathcal{BM} := \{ \Pi_{i,j} \tau_i^{ \epsilon_i} \xi_{j}^{n(j)}, n(j) \in \mathbb{N}, \epsilon(i) \in \{ 0,1 \} \},$$
 where the gradings are $|\tau_i|=(2^i-1)(1+\alpha)+1$ and $|\xi_i|=(2^i-1)(1+\alpha)$.
We call $\mathcal{BM}$ the \textit{monomial basis} of $(H \mf_{ \star} H \mf)_{e}$.
\end{pro}

\begin{proof}
We show that the $H \mf_{ \star}^{\gr}$-module morphism 
$$ \phi : H \mf_{ \star} \{ \mathcal{BM} \} \rightarrow \ste_{ \star}$$
is an isomorphism. \\ 
Let $R$ be the ideal generated by $ a \tau_{k+1}+ \eta_R( \sigma^{-1}) \xi_{k+1} - \tau_k^2$ for $k \geq 0$ so that $$ \ste_{ \star} \cong H \mf_{ \star}^{\gr}[ \xi_{i+1}, \tau_{i} | i \geq 0] / R.$$
\begin{itemize}
\item { \bf Surjectivity: } surjectivity follows from the definition of $\mathcal{BM}$. Let $ \xi_1^{i_1} \hdots \xi_n^{i_n} \tau_0^{j_1} \hdots \tau_m^{j_m}$ be an element of $\mathcal{BM}$. For all $k$ such that $j_k \geq 2$, write  $$ \xi_1^{i_1} \hdots \xi_n^{i_n} \tau_0^{j_1} \hdots \tau_m^{j_m} \equiv \xi_1^{i_1} \hdots \xi_n^{i_n}( \Pi_{k|j_k \leq 1} \tau_{j_k} )( \Pi_{k|j_k \geq 2} \tau_{k}^{j_k-2}(a \tau_{k+1}+ \eta_R( \sigma^{-1}) \xi_{k+1})$$ modulo $R$. By induction over $max \{ j_k \}$, there is an element of $H \mf_{ \star}^{\gr} \{ \mathcal{BM} \}$ whose image by $ \phi$ is $ \xi_1^{i_1} \hdots \xi_n^{i_n} \tau_0^{j_1} \hdots \tau_m^{j_m}$.
\item { \bf Injectivity: }  this is shown analogously to the non-equivariant odd case. First, see that $Ker( \phi) \cong H \mf_{ \star}^{\gr} \{ \mathcal{BM} \} \cap R$. But for all $0 \neq r \in R$, $ \exists i_1, \hdots, i_n, j_1, \hdots, j_k$ and $ \exists k \geq k_0 \geq 0$ such that $j_k \geq 2$ and $pr_{ H \mf^{ \star}_{\gr} \xi_1^{i_1} \hdots \xi_n^{i_n} \tau_0^{j_1} \hdots \tau_m^{j_m}}(r) \neq 0$. By definition of $ \mathcal{BM}$,  $ H \mf_{ \star}^{\gr} \{ \mathcal{BM} \} \cap R = 0$.
\end{itemize}
\end{proof}

To set up the Cartan formulae, we need to refine the previous result a little.

\begin{corr} \label{corr_hypolibertehfd}
There is an isomorphism of Mackey functors $$ H \mf_{ \star}(H \mf) \cong \bigoplus_{b \in \mathcal{BM}} \Sigma^{ |b|} H \mf_{ \star}. $$
\end{corr}

\begin{proof}
Proposition \ref{pro_hypolibrepourhfd} gives a map
$$ \bigvee_{b\in \mathcal{BM}} S^{|b|} \rightarrow H\mf \wedge H\mf.$$
By adjunction, we get a map of $H\mf$-modules
$$ \phi : \bigvee_{b\in \mathcal{BM}} \Sigma^{|b|}H\mf \rightarrow H\mf \wedge H\mf.$$
By Proposition \ref{pro_hypolibrepourhfd}, this is an equivalence in fixed points, and this is clearely an equivalence of underlying non-equivariant spectra. Consequently, $\phi$ is a weak equivalence, so it induces an isomorphism of $RO(\gr)$-graded Mackey functors.
\end{proof}

We recall the following result of Boardman.

\begin{de}[\emph{\cite[Definitions \S 10 and Definition 11.11]{Bo95}}] \label{de_modulecomoduleboard}
\begin{enumerate}
\item Let $A^{ \star}$ be a monoid in the category of $H$-bimodules. A Boardman module over $A^{ \star}$ is a $RO(\gr)$-graded filtered $H$-module $M$, which is complete and Hausdorff, and a continuous $H$-module morphism $ \lambda : A^{ \star} \underset{(r,l)}{ \otimes} M \rightarrow M$ making the appropriate coherence diagram commute.
\item Let $A_{ \star}$ be a $RO( \gr)$-graded Hopf algebroid. A Boardman $A_{ \star}$-comodule is a $RO( \gr)$-graded filtered $H$-module $M$, which is complete and Hausdorff, together with a continuous $H$-module morphism $ \rho : M \rightarrow M \hat{ \underset{(l,l)}{ \otimes}} A_{ \star}$, where the action of $H$ on $M \hat{ \underset{(l,l)}{ \otimes}} A_{ \star}$ is defined by $h ( m \otimes s) = m \otimes \eta_R(h)s$, for $m \otimes s \in M \hat{ \underset{(l,l)}{ \otimes}} A_{ \star}$ and $h \in H$.
\end{enumerate}
\end{de}

\begin{pro}[\emph{\cite[Theorem 11.13]{Bo95}}] \label{pro_thmboardman}
Suppose that $A_{ \star}$ is a free Boardman $H$-module, and denote $A^{ \star} = \Hom_{H\dashmod mod}(A_{ \star}, H)$. Then, the category of Boardman $A^{ \star}$-modules is equivalent to the category of Boardman $A_{ \star}$-comodules.
\end{pro}

Now, we turn to the most important result of this section. Recall that, for a ring $\gr$-spectrum $E$, the pair $(E_{\star}^{\gr}, E_{\star}^{\gr}E)$ has a natural Hopf algebroid structure.

\begin{thm} \label{thm:boardman_hmf}
Denote
\begin{itemize}
\item $ \ste^{\star} = (H\mf^{\star}_{\gr}(H\mf))$
\item $\ste_{\star} = (H\mf_{\star}^{\gr}(H\mf))$.
\end{itemize}
Then, , the category of Boardman $\ste^{ \star}$-modules is equivalent to the category of Boardman $\ste_{ \star}$-comodules.
\end{thm}

\begin{proof}
There are two distinct parts to the proof.
Firstly, Proposition  \ref{corr_hypolibertehfd} gives the freeness of $H\mf \wedge H\mf$ as a $H\mf$-module, so that we have an explicit isomorphism, say $\phi$, between $H\mf \wedge H\mf$ and a sum of shifts of $H\mf$. Consequently, we have an isomorphism

\begin{eqnarray*}
 \ste^{\star} & =  & (H\mf^{\star}_{\gr}(H\mf)) \\
& = &  [H\mf, H\mf]^{\star}_{\gr} \\
& \cong & ( \Hom_{H\mf\dashmod mod}(H\mf\wedge H\mf, H\mf))^{\star}_{\gr} \\
& \stackrel{\phi^*}{\cong}  & ( \Hom_{H\mf\dashmod mod}(\bigvee_{x \in \mathcal{BM}} \Sigma^{|x|}H\mf\dashmod mod, H\mf))^{\star}_{\gr} \\
& = & \Hom_{H\mf^{\star}_{\gr}}((H\mf_{\star}^{\gr}H\mf), H\mf^{\star}_{\gr}) \\
& = & \Hom_{H\mf^{\star}_{\gr}}(\ste_{\star}, H\mf^{\star}_{\gr}). \\
\end{eqnarray*}

Secondly, Proposition \ref{pro_hypolibrepourhfd} allows us to apply Boardman's result: Proposition \ref{pro_thmboardman}, for $H = H\mf^{\star}_{\gr}$ and $\ste_{\star} = (H\mf_{\star}^{\gr}(H\mf))$. The first point of this proof gives the identification $ \ste^{\star} = (H\mf^{\star}_{\gr}(H\mf))$, and concludes the proof.
\end{proof}

We conclude this section by exhibiting some equivariant Cartan formulae using Theorem \ref{thm:boardman_hmf}.

\begin{de}
For $x \in \ste_{\star}$, denote $x^{\vee} \in \ste^{\star}$ the dual of $x$, that is the preimage of $x$ by the isomorphism
$\ste^{\star} \cong \Hom_{H\mf^{\star}_{\gr}}(\ste_{\star}, H\mf^{\star}_{\gr})$
described in the proof of Theorem \ref{thm:boardman_hmf}.
\end{de}

\begin{pro} \label{pro_cartancoeff}

\begin{enumerate}
 \item For all $h \in H \mf^{ \star}_{\gr}$,  $ \eta_R(h) = \sum_{h' \in H \mf^{ \star}_{\gr}, x \in \ste_{ \star}} (x^{ \vee}h)x$.
\item Let $M$ be a Boardman $ \ste^{ \star}$-module and $x^{ \vee} \in \ste^{ \star}$.
Define $x'_i$ and $x''_i$ in $\ste_{\star}$ by
$$ \sum_{i \geq 0} x'_i \otimes x''_i = \sum_{h,y,z| pr_x(yz) = \eta_R(h)x} hy \otimes z \in \ste_{\star}$$
where the second sum is over $h \in H \mf_{ \star}^{\gr}$ and $y,z$ in $\mathcal{BM}$.
Then, for all $h \in H \mf^{ \star}_{\gr}$ and $m \in M$,
$$ x(hm) = \sum_{i \geq 0} x'_i(h) x''_i(m).$$
\end{enumerate}
\end{pro}

\begin{proof}
Recall Proposition \ref{pro_thmboardman}, which gives the formula
$ \lambda(m) = \sum_{x \in \mathcal{BM}} x^{ \vee}m \otimes x$. \\
Denote the structure morphism in the following way
\begin{itemize}
 \item $ \mu : \ste^{ \star} \otimes M \rightarrow M$ for the Boardman $ \ste^{ \star}$-module structure on $M$, and $xm$ for the $x \in \ste^{ \star}$ action on $m \in M$.
\item $ \lambda : M \rightarrow M \hat{ \otimes} \ste_{ \star}$ for the Boardman $ \ste_{ \star}$-comodule on $M$, defined by Proposition \ref{pro_thmboardman}. In particular, it is a Boardman $H \mf^{ \star}_{\gr}$-module morphism, and the action of $H \mf^{ \star}_{\gr}$ on $M \otimes  \ste_{ \star}$ is induced by $ \eta_L$.
\end{itemize}

Let $m \in M$ and $h \in H \mf^{ \star}_{\gr}$. Write $ \eta_R(h) = \sum_{i \geq 0} h'_i x_{h,i}$. Then,
\begin{eqnarray*}
\sum_x x^{ \vee}(hm) \otimes x & = & \lambda(hm) \\
 & = & \left[ \sum_{x} (x^{ \vee}m \otimes x) \right] h \\
& = &  \sum_{x} x^{ \vee}m \otimes \eta_R(h) x \\
& = & \sum_{i,m', x|x^{ \vee}m = m'} m' \otimes x h'_i x_{h,i} \\
& = & \sum_{i,m', x|x^{ \vee}m = m'} h'_i m' \otimes x  x_{h,i} .
\end{eqnarray*}

In particular, for $M = H \mf^{ \star}_{\gr}$ and $h = 1$, one has $ \sum_x x^{ \vee}(m) \otimes x =  \sum_x x^{ \vee}(1) \otimes \eta_R(m) = 1 \otimes \eta_R(m)$, this gives the first point. \\

We prove the second point. By $(1)$, we rewrite the sum
$$ \lambda(hm) = \sum_{i,m', x|x^{ \vee}m = m'} h'_i m' \otimes x  x_{h,i} = \sum_{x,x' \in \mathcal{BM}} x'^{ \vee}(h) x^{ \vee}(m) \otimes x  x',$$
thus, for $y^{ \vee} \in \ste^{ \star}$, $y(hm) = \sum_{pr_y(xx') = \eta_R(h')y} h'x^{ \vee}(h) x'^{ \vee}(m).$
\end{proof}

We now turn to the particular algebra generated by $\qo$ and $\qi$. First, the operation $\qi$ being build as a Bockstein coming from an exact couple, it has trivial square. Thus, it is the only element of $\ste^{\star}$ satisfying this property: $\tau_1^{\vee}$ in the notations of \cite{HK01}. The operation $\qo$ correspond to $\tau_0^{\vee}$, the only non trivial operation in degree one. We deduce two important corollaries from the previous proposition.

\begin{corr}[Cartan formulae] \label{cor:cartan}
Let $X$ be a $\gr$-spectrum, $x \in H\mf^{\star}_{\gr}(X)$ and $h \in H\mf^{\star}_{\gr}$.
Then
\begin{itemize}
\item $\qo(hx) = \qo(h)x + h \qo(x)$,
\item $\qi(hx) = \qi(h)x + a \qo(h) \qo(x) + h \qi(x)$.
\end{itemize}
\end{corr}

\begin{proof}
Follows from the proposition, and the identification $\tau_0^{\vee} = \qo$ and $\tau_1^{\vee} = \qi$.
\end{proof}

\begin{corr} \label{corr_actionstecoeff}
Let $k,n \geq 0$.
\begin{enumerate}
\item For $i=0$ and $1$, the operation $\mathbb{Q}_i$ induce a $ \F[a]$-module morphism on the $ H \mf$-cohomology of any $\gr$-spectrum. 
\item $ \qo(a^k \sigma^{-n}) = \left\{ \begin{matrix} a^{k+1} \sigma^{-n+1} \text{ if n odd} \\ 0 \text{ if n even} \end{matrix} \right.$
\item $ \qo(a^{-k} \sigma^{n}) = \left\{ \begin{matrix} a^{-k+1} \sigma^{n+1} \text{ if n even} \\ 0 \text{ if n odd} \end{matrix} \right.$
\item $ \qi(a^k \sigma^{-n}) = \left\{ \begin{matrix} a^{k+3} \sigma^{-n+2} \text{ if n is 2 or 3 modulo 4} \\ 0 \text{ if n is 0 or 1 modulo 4} \end{matrix} \right.$
\item $ \qi(a^{-k} \sigma^{n}) = \left\{ \begin{matrix} a^{k+1} \sigma^{-n+1} \text{ if n is 2 or 3 modulo 4} \\ 0 \text{ if n is 0 or 1 modulo 4} \end{matrix} \right.$
\end{enumerate}
\end{corr}

\begin{proof}
\begin{itemize}
\item We know that $ \eta_R(a) = a$ by \cite[Theorem 6.41]{HK01}. But $ \eta_R$ is a ring morphism. The first point now follows from Proposition \ref{pro_cartancoeff}.
\item We compute $ \eta_R( a^k \sigma^{-n})$ modulo $( \xi_1, \xi_2, \hdots)$. As $ \tau_0^{2^i}= a^{2^i} \tau_i$ modulo this ideal, one has $ \eta_R( a^k \sigma^{-n}) = a^k \eta_R( \sigma^{-n}) = a^k ( \sigma^{-1} + a \tau_0)^n = a^k \sum_{i = 0}^n \begin{pmatrix} n \\ i \end{pmatrix} \sigma^{-n+i} a^i \tau_0^i$ and in particular, the coefficient in front of $ \tau_i$ is $ \begin{pmatrix} n \\ 2^i \end{pmatrix}a^{k+2^{i+1}-1} \sigma^{-n+2^i} $. The assertions $(2)$ and $(4)$ now follows.
\item For $(3)$ and $(5)$, we apply Cartan's formula to the equality $0 = \sigma^{-n+1} a^{-k} \sigma^{n}$ where the product on the right hand side is considered as the action of $H \mf^{\star}_{\gr}$ on itself via its ring structure. One finds
\begin{eqnarray*}
0 & = & \qo(\sigma^{-n+1} a^{-k} \sigma^{n}) \\
& = & \qo( \sigma^{-n+1} ) a^{-k} \sigma^{n} + \sigma^{-n+1} \qo(a^{-k} \sigma^{n}) 
\end{eqnarray*}
now $(2)$ gives $(3)$.
At last,
\begin{eqnarray*}
0 & = & \qi(\sigma^{-n+1} a^{-k} \sigma^{n}) \\
& = & \qi( \sigma^{-n+1} ) a^{-k} \sigma^{n} + a \qo( \sigma^{-n+1}) \qo(a^{-k} \sigma^{n}) + \sigma^{-n+1} \qi(a^{-k} \sigma^{n}) 
\end{eqnarray*}
now, $(2)$ and $(4)$ together gives $(5)$.
\end{itemize}
\end{proof}

\subsection{Equivariant cohomology of non-equivariant spectra}

Recall that $\ste(1)$ denotes the sub-algebra of the non-equivariant modulo $2$ Steenrod algebra generated by the first two Steenrod squares $Sq^1$ and $Sq^2$.
\begin{de} \label{de:R}
Define $$R : \ste(1)\dashmod mod \rightarrow \E\dashmod mod$$ by the following formulae:
\begin{itemize}
\item for $M \in \ste(1)\dashmod mod$, $R(M) = H\mf^{\star}_{\gr} \otimes_{\F} M$ as a $RO(\gr)$-graded $\F$-vector space,
\item for all $x \in M$, $\qo(x) = Sq^1(x) \in R(M)$,
\item for all $x \in M$, $\qi(x) = a Sq^2(x)  + \sigma^{-1} Q_1(x) \in R(M)$,
\item the action of $\qo$ and $\qi$ satisfies the Cartan formulae.
\end{itemize}
\end{de}

The following result gives a strikingly powerful tool that relates the category of $\ste(1)$-modules to the category of modules over the exterior algebra generated by $\qo$ and $\qi$.

\begin{thm} \label{thm:formula_qoi} \label{thm:hmfx_vs}
Let $H\mf^{\star} \otimes(-) : \F\dashmod mod^{\Z} \rightarrow H\mf^{\star}\dashmod mod$, where $H\mf^{\star}\dashmod mod$ denotes the category of $H\mf^{\star}$-modules in $\mathcal{M}^{RO(\gr)}$, be the extension of scalars functor.
Then the following diagram is commutative up to natural isomorphism
$$ \xymatrix@=2cm{ \sh \ar[r]  \ar[d]_{H\F^*} & \gr\sh \ar[d]^{H\mf^{\star}} \\
\F\dashmod mod^{\Z} \ar[r]_{H\mf \otimes(-)} & H\mf^{\star}\dashmod mod }$$
where the top arrow is the inflation functor.

Moreover, this refines to the following commutative diagram up to natural isomorphism
$$ \xymatrix@=2cm{ \sh \ar[r]  \ar[d]_{H\F^*} & \gr\sh \ar[d]^{H\mf^{\star}} \\
\ste(1)\dashmod mod \ar[r]_{R} & \E\dashmod mod }$$
\end{thm}

\begin{proof}
We first show the commutativity of the first diagram. The fixed points spectrum of $H\mf$ is $H\F$. The left adjoint of the functor $\sh \rightarrow \gr\sh$ is the fixed points $(-)^{\gr} : \sh \rightarrow \gr\sh$, so there is an isomorphism of $\Z$-graded abelian groups
$$ {H \mf}^{*}_{\gr}(X) = [X, \Sigma^* H\mf] \cong [X, (\Sigma^* H\F)^{\gr}] = H\F^*(X),$$
and thus a morphism $f : H \mf^{\star} \otimes_{\F} H\F^*(X) \rightarrow H\mf^{\star}(X)$. This natural transformation of cohomology theories is an isomorphism for $X = S^0$. \\

Observe that $\qo$ and $\qi$ define, via the commutativity of the first diagram, two natural maps
$$\qo : H\mf^{*}_{\gr} \rightarrow H\mf^{*+1}_{\gr},$$
and
$$\qi : H\mf^{*}_{\gr} \rightarrow H\mf^{*+2+\alpha}_{\gr} \cong a H\F^{*+2}_{\gr} \oplus \sigma^{-1} H\F^{*+3}_{\gr}.$$
Now, by naturality, these gives non-equivariant modulo $2$ Steenrod operations $y_0$ and $y_1$,$y'_1$ in degrees $1$, $2$ and three respectively, such that, for all non-equivariant spectrum $X$, and all $x \in H\F^*(X) \subset H\mf^{\star}(X)_{\gr}$,
$$\qo(x) = y_0(x)$$
and
$$\qi(x) = a y_1(x) + \sigma^{-1} y'_1(x).$$

We determine these operations.

The only non-trivial operation possible for $y_0$ is $Sq^1$, for dimensional reasons, this concludes the first identification, because $\qo \neq 0$.

There exists $\epsilon_1,\epsilon_2,\epsilon_3 \in \F$ such that
$$ \qi(x) = \epsilon_1 a Sq^2(x) + \epsilon_2 \sigma^{-1} Sq^2Sq^1(x) + \epsilon_3 \sigma^{-1} Sq^1Sq^2(x)$$
because these operations form a basis of the non-equivariant Steenrod algebra in the appropriate dimension.

Now, at least one of the coefficients is non zero because $\qi$ is non trivial ({\em e.g.} because $KU$ is not split). We will determine the three coefficients using the commutativity of $\qo$ and $\qi$ and the Cartan formulae.
Recall the Adem relation
$$  Sq^2Sq^2 = Sq^1Sq^2 Sq^1.$$
Compute $\qo \qi(x)$: 
\begin{eqnarray*}
 &  & \qo \qi(x) \\
& = &  \qo \left( \epsilon_1 a Sq^2(x) + \epsilon_2 \sigma^{-1} Sq^2Sq^1(x) + \epsilon_3 \sigma^{-1} Sq^1Sq^2(x) \right) \\
& = &  \epsilon_1 a \qo Sq^2(x) + \epsilon_2 \qo ( \sigma^{-1} Sq^2Sq^1(x)) + \epsilon_3 \qo( \sigma^{-1} Sq^1Sq^2(x) \\
& = & \epsilon_1 a Sq^1 Sq^2(x) + \epsilon_2 a Sq^2Sq^1(x)) + \epsilon_2  \sigma^{-1} Sq^1Sq^2Sq^1(x) + \epsilon_3 a Sq^1Sq^2(x) \\
& = & \epsilon_1 a Sq^1 Sq^2(x) + \epsilon_2 a Sq^2Sq^1(x)) + \epsilon_2  \sigma^{-1} Sq^2Sq^2(x) + \epsilon_3 a Sq^1Sq^2(x) \\
& = & ( \epsilon_1+ \epsilon_3) a Sq^1 Sq^2(x) + \epsilon_2 a Sq^2Sq^1(x)) + \epsilon_2  \sigma^{-1} Sq^2Sq^2(x) \\
\end{eqnarray*}

Compute $\qi \qo(x)$: 
\begin{eqnarray*}
 &  & \qi \qo(x) \\
& = & \epsilon_1 a Sq^2Sq^1(x) + \epsilon_2 \sigma^{-1} Sq^2Sq^1Sq^1(x) + \epsilon_3 \sigma^{-1} Sq^1Sq^2Sq^1(x) \\
& = &   \epsilon_1 a Sq^2Sq^1(x) + \epsilon_3 \sigma^{-1} Sq^2Sq^2 \\
\end{eqnarray*}

Now, $\qi \qo(x) = \qo \qi(x)$, so $\epsilon_1 = \epsilon_2 = \epsilon_3$, thus $\qi(x) =  a Sq^2(x) + \sigma^{-1} (Sq^2Sq^1 + Sq^1Sq^2)(x) = (aSq^2 + \sigma^{-1}Q_1)(x)$.

This proves that both $R(H\F^{*}(X))$ and $H\mf^{\star}(X)$ satisfies the $(2)$ and $(3)$ of Definition \ref{de:R}. The two remaining properties for $H\mf^{\star}(X)$ are the subject of the corollary \ref{cor:cartan} and the commutativity of the first diagram up to natural isomorphism.
We conclude by the unicity of a $H\mf^{\star}_{\gr}$-module and a $\E$-module satisfying these four properties.
\end{proof}

\begin{rk}
This is a particular case of a much more general situation studied by the author in his PhD thesis \cite[Section 3]{RicPhD}. Let, as before, $\ste^{\star}$ denotes the $\gr$-equivariant Steenrod algebra, and let $\ste_*$ denotes the non-equivariant Steenrod algebra. Then, for any subalgebra $\mathcal{B}$ of the non-equivariant Steenrod algebra $\ste^*$, there exists a sub-algebra $\mathbb{X}$ of the equivariant Steenrod algebra $\ste^{\star}$ and a functor $$R_{\mathcal{B}}^{\mathcal{X}} : \mathcal{B}\dashmod mod \rightarrow \mathcal{X}\dashmod mod$$ such that the first commutative diagram of the previous theorem refines to
$$ \xymatrix@=2cm{ \sh \ar[r]  \ar[d]_{H\F^*} & \gr\sh \ar[d]^{H\mf^{\star}} \\
\mathcal{B}\dashmod mod \ar[r]_{R_{\mathcal{B}}^{\mathcal{X}} } & \mathcal{X}\dashmod mod. }$$

The hope is that the structure of the category of modules over the subalgebra $\ste(n)$ of the Steenrod algebra can be seen through the eyes of a simpler subalgebra of the equivariant Steenrod algebra.
\end{rk}

\subsection{Duality and the functor $R$} \label{subsection_dualitee1}

As $ \E$ and $ \ste(1)$ are both Hopf algebras, the categories $  \E\dashmod mod$ and $ \ste(1)\dashmod mod$ have a $ \F$-linear duality functor
$$(-)^{ \vee} :  \E\dashmod mod^{op} \rightarrow \E\dashmod mod$$
and
$$(-)^{ \vee} :  \ste(1)\dashmod mod^{op} \rightarrow \ste(1)\dashmod mod,$$

defined via $Hom_{ \F}(-, \F)$.

We want to understand the relationship between $ R : \ste(1)\dashmod mod \rightarrow \E\dashmod mod$ and these two duality functors. The principal result is the following.

\begin{pro} \label{pro_compatibilite_dualites}
The diagram
$$ \xymatrix{ \ste(1)\dashmod mod^{op}  \ar[d]_{ R} \ar[rr]^{(-)^{ \vee}} & & \ste(1)\dashmod mod \ar[d]^{ R} \\  \E\dashmod mod^{op} \ar[rr]_{ \Sigma^{2-2 \alpha} (-)^{ \vee}}  & &  \E\dashmod mod } $$
commutes up to a natural isomorphism.
\end{pro}

The key point is the case $M = \F$, which correspond to the following lemma.

\begin{lemma} \label{lemma_isodualitecoeff}
The pairing
\begin{eqnarray*}
 H \mf^{ \star}_{\gr} \otimes H \mf^{ \star}_{\gr}  & \rightarrow & \Sigma^{2-2 \alpha} \F \\
h \otimes k & \mapsto & \pi_{ \sigma^2}(hk) \\
\end{eqnarray*}
induces a $\E$-module isomorphism
$$ w : H \mf^{ \star}_{\gr} \stackrel{ \simeq}{ \rightarrow} \Sigma^{2-2 \alpha} (H \mf^{ \star}_{\gr})^{ \vee}.$$
This isomorphism satisfies the following formulae, for all $m, n \geq 0$,
\begin{align*}
a^m \sigma^{-n}  & \mapsto  \pi_{a^{-m} \sigma^{n+2}} \\
a^{-m} \sigma^{n+2} & \mapsto  \pi_{a^m \sigma^{-n}} \\
\end{align*}
where, for $h \in H \mf^{ \star}_{\gr}$,  $ \pi_h : H \mf^{ \star}_{\gr} \rightarrow \F$ stands for the projection on $h$.
Moreover, for $h, k,l \in H \mf^{ \star}_{\gr}$, $w(hk)(l) = w(h)(kl)$. 
\end{lemma}

\begin{proof}
The map $w$ is a $ \F$-vector space isomorphism by Proposition \ref{pro_cohompoint}.
By corollary \ref{corr_actionstecoeff} we have
$$ \mathbb{Q}_i w h = w \mathbb{Q}_i  h $$
for $i= 0$ or $1$ and for all $h \in H \mf^{ \star}_{\gr}$.

The last assertion comes from the fact that the isomorphism is induced by the pairing
\begin{eqnarray*}
 H \mf^{ \star}_{\gr} \otimes H \mf^{ \star}_{\gr}  & \rightarrow & \Sigma^{2-2 \alpha} \F \\
h \otimes k & \mapsto & \pi_{ \sigma^2}(hk) \\
\end{eqnarray*}
which is associative because it correspond to the natural product on $H \mf^{ \star}_{\gr}$.
\end{proof}

\begin{rk}
This result is a consequence of a much more general result about the relationship between the duality theory of Mackey functors and equivariant Anderson duality in the $\gr$-equivariant stable homotopy category. This relationship was studied by the author in \cite{Ric3}.
\end{rk}

\begin{proof}[Proof of Proposition \ref{pro_compatibilite_dualites}]
Recall Definition \ref{de:R}, which gives a $ \F$-vector space isomorphism $RM \cong H \mf^{ \star}_{\gr} \otimes M$. 
Consider the natural transformation $ \psi : R \circ (-)^{ \vee} \rightarrow \Sigma^{2-2 \alpha} (-)^{ \vee} \circ R$ defined for all $M \in \ste(1)\dashmod mod$ by
$$ \psi_M(h \otimes f) = w(h) \ f $$
for all $h \in H \mf^{\star}_{\gr}$ and $f : M \rightarrow \F$.

The morphism $ \psi_M$ is clearly a $ \F$-vector space isomorphism. It remains to show that $ \psi_M$ is a $ \E$-module isomorphism.

Let $h \in H \mf^{ \star}_{\gr}$ and $f : M \rightarrow \F \in M^{ \vee}$.

We first show the commutativity with the action of $ \qo$:
\begin{itemize}
\item $ \qo(h \otimes f) = \qo(h) \otimes f + h \otimes (f \circ Q_0)$, thus $ \psi_M( \qo(h \otimes f)) = w( \qo(h)) f + w(h) (f \circ Q_0)$,
\item on the other hand, by Definition of the action of $\qo$, for all $k \in H\mf^{\star}_{\gr}$ and $m \in M$, we have $ \qo( \psi_M(h \otimes f)) ( k \otimes m) = \psi_M(h \otimes f)( \qo(k \otimes m))$. Moreover,

 \begin{eqnarray*}
\psi_M(h \otimes f)( \qo(k \otimes m))  & = & \psi_M(h \otimes f)( \qo(k) \otimes m + k \otimes Q_0(m)) \\
& = & w(h)( \qo(k)) f(m) + w(h)(k) f( Q_0(m)) \\
& = & \qo(w(h))(k) f(m) + w(h)(k) (f \circ Q_0)(m) \\
& = & w( \qo(h))(k) f(m) + w(h)(k) (f \circ Q_0)(m), \\
\end{eqnarray*}

where the last equality comes from the first assertion of Lemma \ref{lemma_isodualitecoeff}. 
\end{itemize}
We deduce from that $ \qo( \psi_M(h \otimes f)) = \psi_M( \qo(h \otimes f)) $.

We now show that $ \psi_M$ is a  $ \Lambda_{ \F}( \qi)$-module morphism.

\begin{itemize}
 \item By the Cartan formulae, $ \qi(h \otimes f) = \qi(h) \otimes f + a \qo(h) \otimes (f \circ Q_0) + ah \otimes (f \circ Sq^2) + \sigma^{-1}h \otimes ( f \circ Q_1)$, thus $ \psi_M(  \qi(h \otimes f) ) = w( \qi(h))  f + w(a \qo(h)) (f \circ Q_0) + w(ah)  (f \circ Sq^2) + w( \sigma^{-1}h)  ( f \circ Q_1)$,
\item let $k \in H \mf^{ \star}_{\gr}$ and $m \in M$, we have
\end{itemize}

\begin{eqnarray*}
& & \qi( \psi_M( h \otimes f) )(k \otimes m) \\
& = & \psi_M( h \otimes f)( \qi( k \otimes m)) \\
 & = & \psi_M( h \otimes f)(  \qi(k) \otimes m + a \qo(k) \otimes Q_0m + ak \otimes Sq^2m + \sigma^{-1}k \otimes Q_1m) \\
& = & w(h)( \qi(k)) f( m) + w(h)(a \qo(k)) f( Q_0m) \\ & & + w(h)(ak) f( Sq^2m) + w(h)( \sigma^{-1}k) f( Q_1m). \\
\end{eqnarray*}

But by the first assertion of Lemma \ref{lemma_isodualitecoeff}, $w(h)( \mathbb{Q}_i k) = w( \mathbb{Q}_i h)(k)$, for $i =0$ or $1$, and by the last assertion of Lemma \ref{lemma_isodualitecoeff}, $ w(a \qo(h))(k) = w( \qo(h))(ak)$, $w(ah)(k) = w(h)(ak)$ and $w( \sigma^{-1}h)(k) = w(h)( \sigma^{-1}k)$, thus
$$ \psi_M(  \qi(h \otimes f) ) = \qi(  \psi_M( h \otimes f) ).$$
The result follows.
\end{proof}

\section{The computation of $\mathcal{H}^{\star}(V)$: the functor $H_{01}^{\star}$} \label{sec_hoistar}

The aim of this section is to provide tools for the computation of $\mathcal{H}^{\star}(V)$. The main result of this section is Theorem \ref{thm_secR} which enables a computation of $\mathcal{H}^{\star}(V)$ by induction on the dimension of $V$.
In the process, we exhibit a surprising relationship between the stable categories of $\ste(1)$-modules and $\E$-modules with an extra action by $H\mf^{\star}_{\gr}$. This result is useful since $\E$ is simply an exterior algebra over $\F$ generated by the two equivariant Milnor operations $\qo$ and $\qi$.

\subsection{$(\E,\lo)$-relative homological algebra} \label{subsection:hoi}

The aim of this section is to provide tools to make accessible an explicit computation of $\mathcal{H}^{\star}(V)$ from Notation \ref{nota:hv}, for all $V$ elementary abelian $2$-group.

Let $A$ be a unital ring and $B \subset A$ a subring.
Relative homological algebra, introduced by Hochschild in \cite{Hoch56}, and studied in its general form by Eilenberg-Moore in \cite{EM65}, consists in changing the model structure on the category of $A$-modules, to one which neglect the homological properties of the underlying $B$-module. \\

The original paper of Hochschild \cite{Hoch56} and the book of Enochs and Jenda \cite{EJ11} are good references for our use of relative homological algebra. 
\begin{nota}
Let $\lo = \Lambda_{\F}(\qo)$ and $\li = \Lambda_{\F}(\qi)$.
\end{nota}
For our purpose, we stick to the case of $(\E,\lo)$-relative homological algebra.

\begin{de} \label{de_exactrelatif}
We say that a sequence of $\E$-modules
$$ \hdots \rightarrow M_i \stackrel{d_i}{ \rightarrow} M_{i-1} \rightarrow \hdots$$
is $( \E, \lo)$-exact if it is an exact sequence of $\E$-modules such that the underlying sequence of $ \lo$-modules is split.
\end{de}

\begin{rk} \label{rk:exactrelatif}
In particular, any short exact sequence $$ M \inj M' \surj M''$$ such that $M$ or $M''$ is free as a $ \lo$-module is an $( \E, \lo)$-exact sequence. 
\end{rk}

All the usual notions of homological algebra, such as projectivity, injectivity, resolutions, are obtained by replacing the notion of exactness by this  $( \E, \lo)$-exactness.

%
%

\begin{pro} \label{pro_carac_proinjrel}
The classes consisting of $( \E, \lo)$-injective $ \E$-modules and $( \E, \lo)$-projective $\E$-modules coincide. Moreover, this common class is the one consisting of $\E$-modules of the form $$ \E \otimes_{ \F} V_F \oplus \li \otimes_{ \F} V_T$$
for $V_F$ and $V_T$ some $\Z$-vector spaces.

Thus, a $\E$-module $M$ is in this class if and only if it is of the form $ \li \otimes M'$ for some $ \lo$-module $M'$.
\end{pro}

\begin{proof}
This is an immediate consequence of \cite[lemmas 1 and 2]{Hoch56}
%
%
\end{proof}
%
%

\begin{pro} \label{corr_caracfpreserveproj}
Let $F : \E\dashmod mod \rightarrow \E\dashmod mod$ be an exact functor. Then $F$ preserve the $( \E, \lo)$-projective $\E$-modules if and only if, for all $M \in \E\dashmod mod$, $F( \E \otimes_{ \lo}M)$ is a $( \E, \lo)$-projective $\E$-module.
\end{pro}

\begin{proof}
The implication $ \Rightarrow$ is clear. \\
For the other direction, let $M$ be a $( \E, \lo)$-projective $\E$-module. Then, the canonical projection $$ \E \otimes_{ \lo} M \surj M$$ is split, so $F(M)$ is a summand of the $( \E, \lo)$-projective $\E$-module $$F( \E \otimes_{ \lo} M).$$ Thus, $F(M)$ is $( \E, \lo)$-projective.
\end{proof}

\begin{corr} \label{corr_resfd}
\begin{itemize}
\item The complex $I^{ \bullet}_{ \F}$:
$$ \hdots \stackrel{ \qi}{ \rightarrow} \Sigma^{(n+1) | \qi|} \li \stackrel{ \qi}{ \rightarrow} \Sigma^{n | \qi|} \li  \stackrel{ \qi}{ \rightarrow} \hdots \stackrel{ \qi}{ \rightarrow} \li \surj \F \rightarrow 0,$$
where all maps are induced by $ \qi$ except the last one which is the surjection $ \li \surj \F$  is a $ ( \E, \lo)$-projective resolution of $ \F$.
\item The complex $P_{ \bullet}^{ \F}$:
$$ \F \inj \Sigma^{-| \qi|} \li \stackrel{ \qi}{ \rightarrow}  \Sigma^{-2| \qi|} \li \hdots $$
where all maps are induced by $ \qi$ except the first one, which is $ \F \inj \li$ is a $ ( \E, \lo)$-injective resolution of $ \F$.
\item Let $T_{\F}^{\bullet}$ be the periodic complex
$$ \hdots \stackrel{ \qi}{ \rightarrow} \Sigma^{(n+1) | \qi|} \li \stackrel{ \qi}{ \rightarrow} \Sigma^{n | \qi|} \li  \stackrel{ \qi}{ \rightarrow} \hdots$$
with $\li$ in degree $0$. It is an $(\E,\lo)$-exact complex.
\end{itemize}
\end{corr}

\begin{proof}
Is is a consequence of Proposition \ref{pro_carac_proinjrel}.
\end{proof}

The point of corollary \ref{corr_resfd} is that it gives functorial resolutions of any $\E$-module $M$, by tensoring up with $I^{\bullet}_{\F}$ and $P^{\bullet}_{\F}$.

\begin{corr} \label{cor_resbeta0acyclique} \label{de_resolutions_fonctorielles}
\begin{itemize}
\item The functor $I_{ \bullet} :=(-) \otimes I^{ \bullet}_{ \F}$ defines a functorial $( \E, \li)$-injective resolution in $\E$-modules.
\item The functor $P_{ \bullet} :=(-) \otimes P_{ \bullet}^{ \F}$ defines a functorial $( \E, \li)$-projective resolution.
\item the functor $T_{ \bullet} := T_{ \bullet}^{ \F} \otimes (-)$ has values in $(\E,\lo)$-exact complexes of $\E$-modules. It is called the Tate complex functor, and, for an $\E$-module $M$, $T_{\bullet}(M)$ is called the Tate complex of $M$.
\end{itemize}
\end{corr}

\begin{proof}
First, by Proposition \ref{pro_carac_proinjrel}, the complexes $M \otimes I^{ \bullet}_{ \F}$ and $M \otimes P_{ \bullet}^{ \F}$ contain only $( \E, \li)$-injective and $( \E, \li)$-projective $\E$-modules.

Moreover, the functor $M \otimes -$ is exact, so these complexes are exact. 

Finally, the functor $ UM \otimes - : \lo\dashmod mod \rightarrow \lo\dashmod mod$, where $U : \E\dashmod mod \rightarrow \lo\dashmod mod$ stands for the forgetful functor, is an additive functor. Therefore the underlying long exact sequences of $ \lo$-modules are split.
\end{proof}

\begin{de}[\emph{ \cite[section 8.1]{EJ11}}]
\begin{itemize}
\item Let $F$ be a left $ ( \E, \lo)$-exact functor. Define the $n$th right $ ( \E, \lo)$-derived functor of $F$,  $\R^nF$, to be $H^n(F(I^{ \bullet}))$.
\item Let $G$ be a right $ ( \E, \lo)$-exact functor. Define the $n$th left $ ( \E, \lo)$-derived functor of $G$, $\Le_nG$, to be $H_n(F(P_{ \bullet}))$.
\end{itemize}
\end{de}

For computational simplicity, we will now introduce a "Tate homology" version of these functors.

\begin{de} \label{de_tatehomology}
Call Tate complex the functor $T_{ \bullet}$.
 Let $F : \E\dashmod mod \rightarrow \ab$ be a right (resp. left) $( \E, \lo)$-exact functor.

For $i \in \Z$, define the $i$th left (resp. right) Tate derived functor of $F$, $\widehat{\Le_i}F$ (resp. $ \widehat{\R^i}F$) by $ \widehat{\Le_i}F = H_i( F( T_{ \bullet}(-))$ (resp. $ \widehat{\R^i}F = H_i( F( T_{ - \bullet}(-))$).
\end{de}

There is the following comparison between the Tate derived functors and the relative derived functors.

\begin{pro} \label{pro_tate}
\begin{enumerate}
\item Let $F : \E\dashmod mod \rightarrow \ab$ be a left $( \E, \lo)$-exact functor, then $ \widehat{\R^i}F$ is naturally isomorphic to $\R^iF$ for all $i \geq 1$.
\item Let $G : \E\dashmod mod \rightarrow \ab$ be a right $( \E, \lo)$-exact functor, then $ \widehat{\Le_i}F$ is naturally isomorphic to $\Le_iF$ for all $i \geq 1$.
\item Let $A \rightarrow B \rightarrow C$ be a short  $( \E, \lo)$-exact sequence, then there are long exact sequences of the form $$  \hdots \rightarrow \widehat{\R^i}F(A) \rightarrow \widehat{\R^i}F(B) \rightarrow \widehat{\R^i}F(C) \rightarrow \widehat{\R^{i-1}}F(A) \rightarrow \hdots $$
and
$$  \hdots \rightarrow \widehat{\Le_i}F(A) \rightarrow \widehat{\Le_i}F(B) \rightarrow \widehat{\Le_i}F(C) \rightarrow \widehat{\Le_{i+1}}F(A) \rightarrow \hdots .$$
\end{enumerate}
\end{pro}

\begin{proof}
The two first points are consequences of the definition of $T_{ \bullet}$ and unicity of relative derived functors.
The third point is a consequence of the snake lemma.
\end{proof}

\begin{de}[\emph{ \cite[section 8.1]{EJ11}}]
\begin{enumerate}
\item Define $ \extrel^i(-,-) : \E\dashmod mod^{op} \times \E\dashmod mod \rightarrow \E\dashmod mod $ as the $(\E,\lo)$-derived functor of $Hom_{ \E}(-, N)$.
\item Define $$ \torrel_i(-,-) : \E\dashmod mod \times \E\dashmod mod \rightarrow \E\dashmod mod$$ as the $ ( \E, \lo)$-derived tensor product.
\end{enumerate}
\end{de}

As usual, the snake lemma provides long exact sequences in $\torrel$ and $\extrel$ induced by short $(\E, \lo)$-exact sequences.

\subsection{An interpretation of $H_{01}^{\star}$ in relative homological algebra}

We now use relative homological algebra to give a better interpretation of the functor $H_{01}^{\star}$.
First, observe that the Cartan formulae \ref{cor:cartan} implies that $\mathcal{H}^{\star}(V)$ has the structure of a $RO(\gr)$-graded $\F[a]$-module. We want to explicit this structure as well. 
We start by the study of a functor strongly related to $\mathcal{H}^{\star}$.

\begin{nota}
Let $A$ be a commutative ring and $x \in A$. Denote by
 \begin{enumerate}
\item $Ker_x$ the functor $A\dashmod mod \rightarrow A\dashmod mod$ defined by the Kernel $x$,
\item $Im_x$ the image of $x$,
\item $Coker_x = id/Im_x$.
\end{enumerate}
\end{nota}

\begin{de} \label{de_h01}
Let $H_{ 01}^{ \star} : \E\dashmod mod \rightarrow \F\dashmod mod$ be the functor
$$H_{01}^{ \star} = (Ker_{ \qi} \cap Ker_{ \qo}) / (Im_{ \qi} \circ Ker_{ \qo}).$$
\end{de}

Of course, by definition, one has 
\begin{equation}
\mathcal{H}^{\star}(V) = H_{01}^{\star}(R(H\F^{*}(BV))). \label{eqn:h_hrhf}
\end{equation}

We now develop several tools for the computation of $H_{01}^{\star}$ in general, and apply it to the particular $\E$-module $R(H\F^{*}(BV))$.


\begin{pro} \label{pro_identh01}
There is an isomorphism $$ \extrel^0( \F, -) \cong Ker_{ \qo} \cap Ker_{ \qi},$$
and, for all $n \geq 1$, there are natural isomorphisms
$$ H_{01}^{ \star} \cong \Sigma^{-n | \qi|} \extrel^n( \F,-).$$
\end{pro}

\begin{proof}
Consider the $( \E, \lo)$-projective resolution of corollary \ref{corr_resfd}.
One has
$$Hom_{ \E}( \F , -) = Ker_{ \qo} \cap Ker_{ \qi},$$
so $ \extrel^0( \F, -) \cong Ker_{ \qo} \cap Ker_{ \qi}.$
Moreover, by adjunction,
$$ Hom_{ \E}( \li , -) \cong Hom_{ \lo}( \F, -) \cong Ker_{ \qo}(-),$$
and precomposing by $ \qi$ makes the diagram
$$ \xymatrix{ Hom_{ \E}( \li , -) \ar[r]^{ \qi^{*}} \ar[d]_{ \cong} & Hom_{ \E}( \Sigma^{| \qi|} \li , -) \ar[d]^{ \cong} \\
Ker_{ \qo}(-) \ar[r]_{ \qi|_{Ker_{ \qo}}} & Ker_{ \qo}(-) } $$
commutative. Thus, for $n \geq 1$,
$$ \extrel^n( \F, M) =  Ker_{ \qi}( Ker_{ \qo}(  \Sigma^{-n | \qi|} M))/ \qi(Ker_{ \qo}(  \Sigma^{-n | \qi|} M)) = H_{01}^{ \star}(M).$$
\end{proof}

%

There is the following comparison between the Tate derived functors (from Proposition \ref{pro_tate}) and the relative derived functors.

\begin{pro} \label{pro_tateh01}
\begin{enumerate}
\item For all $n \in \Z$, there are natural isomorphisms $ \widehat{\R^i}Hom_{ \E}( \F, -) \cong H_{01}^{ \star - i| \qi|}$.
\item Let $A \rightarrow B \rightarrow C$ be a short $( \E, \lo)$-exact sequence, then, there is a long exact sequence of the form $$ \hdots H_{01}^{ \star}(A) \rightarrow H_{01}^{ \star}(B) \rightarrow H_{01}^{ \star}(C) \rightarrow H_{01}^{ \star + | \qi|}(A) \rightarrow \hdots .$$
\end{enumerate}
\end{pro}

\begin{proof}
Proposition \ref{pro_identh01} provides the isomorphism $ \widehat{\R^i}Hom_{ \E}( \F, -) \cong H_{01}^{ \star - i| \qi|}$.

The first assertion and Proposition \ref{pro_tate} gives $(2)$. 
\end{proof}

\subsection{The composite $H_{01}^{\star} \circ R$}

We now turn to the additional structure of $H_{01}^{\star}(M)$, when $M$ is in the image of the functor $R$. 

\begin{de}
Let $\E_{H\mf^{\star}_{\gr}}\dashmod mod$ the category of $\E_{H\mf^{\star}_{\gr}}\dashmod mod$ of $H\mf^{\star}_{\gr}$-modules in $\E\dashmod mod$, \textit{i.e.} $\Lambda_{H\mf^{\star}_{\gr}}(\qo,\qi)$-modules.
\end{de}

By Definition \ref{de:R}, this functor can be seen as taking its values in $\E_{H\mf^{\star}_{\gr}}\dashmod mod$.

\begin{lemma} \label{lemma_H01R}
The restriction of $H_{01}^{ \star}$ to $\E_{H\mf^{\star}_{\gr}}\dashmod mod$ provides a functor denoted again $H_{01}^{ \star}$:
$$ H_{01}^{ \star} :  \mathcal{E}^{ \star}(1)\dashmod mod \rightarrow \F[a, \sigma^{-4}]\dashmod mod.$$
\end{lemma}

\begin{proof}
 Proposition \ref{pro_cartancoeff} implies that the elements of $ \F[a, \sigma^{-4}]$ are in $Ker_{ \qo}(H \mf^{ \star}_{\gr}) \cap Ker_{ \qi}(H \mf^{ \star}_{\gr})$. Let $M$ be a  $H \mf^{ \star}_{\gr}$-module and $x$ representing a class in $H_{01}^{ \star}(M)$. 
\begin{itemize}
\item By the Cartan formulae, $ \forall h \in  \F[a, \sigma^{-4}]$, $hx \in Ker_{ \qo}(M) \cap Ker_{ \qi}(M)$ so $[hx] \in H_{01}^{ \star}(M)$,
\item moreover, for $ \qi(y) \in Im_{ \qi} \circ Ker_{ \qo}(M)$, the Cartan formulae implies that $h \qi(y) = \qi(hy) \in Im_{ \qi} \circ Ker_{ \qo}(M)$. Thus, the cohomology class $[hx]$ does not depends on the choice of $x$, and thus the morphism is well defined.
\end{itemize}
\end{proof}

\begin{thm} \label{thm_secR}
Let $ A \stackrel{f}{ \rightarrow} B  \stackrel{g}{ \rightarrow} C$ be a short exact sequence of $ \ste(1)$-modules, which is split as an exact sequence of $ \lamqo$-modules. Then
$$ RA \rightarrow RB \rightarrow RC$$
is a short $( \E, \lo)$-exact sequence.
\end{thm}

\begin{proof}
Let $ i : C \rightarrow B$ be the $ \lamqo$-module morphism which splits the short exact sequence.
Then $Ri : RC \rightarrow RB$ satisfies $RgRi = Id_{RC}$, and for all $h \in H \mf^{ \star}_{\gr}$ and $c \in C$,  $\qo(Ri(hc)) = \qo( h i(c)) = \qo(h)i(c) + h Q_0(i(c)) = RI( \qo(h)c + hQ_0(c)) = Ri( \qo( hc))$ by Proposition \ref{pro_cartancoeff}. Consequently, the short exact sequence $ RA \rightarrow RB \rightarrow RC$ is split in $ \lamqo\dashmod mod$.
\end{proof}

\begin{pro} \label{pro_selh01}
Let $ A \rightarrow B \rightarrow C$ be a short exact sequence of $ \ste(1)$-modules, split as a short exact sequence of $ \lamqo$-modules.
Then, there is a long exact sequence of $ \F[a, \sigma^{-4}]$-modules induced by the functor $H_{01}$:

$$ \hdots \rightarrow H_{01}^{ \star}(RA) \rightarrow H_{01}^{ \star}(RB) \rightarrow H_{01}^{ \star}(RC) \rightarrow H_{01}^{ \star+2+ \alpha}(RA) \rightarrow \hdots.$$
In particular, is $C$ is a $Q_0$-acyclic $ \ste(1)$-module, then the hypothesis of the proposition are satisfied.
\end{pro}

\begin{proof}
Theorem \ref{thm_secR} implies that $RA \rightarrow RB \rightarrow RC$ is a short $( \E, \lo)$-exact sequence, allowing us to use Proposition \ref{pro_tateh01} to obtain long exact sequences of $\F$-vector spaces in $H_{01}^{\star}$-homology.
The morphisms induced by $A \rightarrow B$ and $B \rightarrow C$ are $ \F[a, \sigma^{-4}]$-module morphisms by Lemma \ref{lemma_H01R}.

Consider now the edge morphism $ \partial$ of the long exact sequence.
Let $[x] \in H_{01}^{ \star}(RC)$ represented by some $x \in RC$. Let $y \in RB$ be a lift of $x$ to $RB$.

Then $ \partial([x]) = [ \qi(y)] \in H_{01}^{ \star}(RA)$ by construction of the edge morphism. Now, let $h \in  \F[a, \sigma^{-4}]$. A lift of $hx$ to $RB$ is $hy$. Moreover $ \qi(hy) = h \qi(y)$ by the Cartan formulae. Consequently $ \partial h[x] = h \partial [x]$. Thus, the edge morphism is a $ \F[a, \sigma^{-4}]$-module morphism.

The last assertion is a consequence of remark \ref{rk:exactrelatif}.
\end{proof}

\begin{de} \label{de_rp}
Define $$ \rp (-) : \ste(1)\dashmod mod \rightarrow \E_{H\mf^{\star}_{\gr}}\dashmod mod$$ as the sub-functor of $R$ consisting of the sub-$ \E_{H\mf^{\star}_{\gr}}$-module generated by elements of the form $ h \otimes m$, for $m \in M$ and $h \in \F[a, \sigma^{-1}] \subset H \mf^{ \star}_{\gr}$. Denote $ \rn(-)$ the quotient.
\end{de}

\begin{rk}
For degree reasons, there is a splitting $R \cong \rp(-) \oplus \rn(-)$ as $\E$-modules since for any $\ste(1)$-module $M$, no non-trivial action of either $\qo$ or $\qi$ could exist between $\rp(M)$ and $\rn(M)$ because of the cancellation in twist $-1$.
\end{rk}

Our next result is the Proposition \ref{pro_aactionsurh01} which is the main computational tool we will be considering. The objective is to be able to compute the value of the composite $H_{01}^{\star}\circ R$ on the free $\ste(1)$-module of rank one.

\begin{nota} \label{nota_h01stara1}
Denote $H_{01}^* : \ste(1)\dashmod mod \rightarrow \F\dashmod mod$ the functor $$ \Sigma^{-3} Ext_{( \Lambda_{ \F}( Q_0, Q_1), \Lambda_{ \F}( Q_0 ))}^1( \F, -).$$ The grading comes from the internal grading on $Ext_{( \Lambda_{ \F}( Q_0, Q_1), \Lambda_{ \F}( Q_0 ))}^1( \F, -)$
\end{nota}

\begin{lemma} \label{lemma_identh01classique}
For all $ i \geq 1$, there are natural isomorphisms
$$ H_{01}^* \cong \Sigma^{-3i} Ext_{( \Lambda_{ \F}( Q_0, Q_1), \Lambda_{ \F}( Q_0 ))}^i( \F, -) \cong Ker_{Q_0} \cap Ker_{Q_1}/(Im_{Q_1} \circ Ker_{Q_0}).$$
\end{lemma}

\begin{proof}
 The proof of Proposition \ref{pro_identh01} gives, mutatis mutandis, the proof of the lemma.
\end{proof}

\begin{lemma} \label{lemma_sq2diff} \label{lemma_actionsq2sq2}
Let $M$ be a $ \ste(1)$-module.
Then,
\begin{enumerate}
 \item $Ker_{ Q_1}(M) \cap Ker_{Q_0}(M)$ has a natural $ \Lambda_{ \F}( Sq^2)$-module structure.
\item If moreover $M$ is $Q_0$-acyclic, then there is a natural $ \Lambda_{ \F}( \sqop)$-module structure on $H_{01}^*(M)$, where $\sqop$ is defined by $ \sqop([Q_0(m)]) = [Q_0Sq^2m]$ for $[Q_0m] \in H_{01}^*(M)$.
\end{enumerate}
Let $ \Lambda$ be the free product of the algebras $ \Lambda_{ \F}( Sq^2)$ and $ \Lambda_{ \F}(\sqop)$, where $Sq^2$ and $\sqop$ are of degree $2$ and $H_{01}^*$ the restriction of $ H_{01}^*$ to the full subcategory $ \ste(1)\dashmod mod_{Q_0}$ consisting of $Q_0$-acyclic $ \ste(1)$-modules, then $H_{01}^*$ lifts to a functor $$ H_{01}^* : \ste(1)\dashmod mod_{Q_0} \rightarrow \Lambda\dashmod mod.$$
\end{lemma}

\begin{proof}
The two first point is proved using basic manipulations in the Steenrod algebra.
The second part of the lemma is a consequence of the two first points.
\end{proof}

%

The Euler class $a$ is in the image of the Hurewicz map. Consequently, multiplication by $a$ is an $\E$-module morphism. Therefore this map induces an injection $ \rp(-) \stackrel{a}{ \inj} \rp(-)$.

\begin{de}
Let $F$ be the functor  $$ \rp (-) / a \rp (-): \ste(1)\dashmod mod \rightarrow  \E[ \sigma^{-1}]\dashmod mod.$$ 
\end{de}

\begin{rk}
This is well defined since the action of $ \sigma^{-1}$ commute with $ \qo$ and $ \qi$.
Indeed, by the Cartan formulae, given in Proposition \ref{pro_cartancoeff}, $\forall x \in M$, $$ \qo( \sigma^{-1}x) =  \qo( \sigma^{-1})x + \sigma^{-1} Q_0x \equiv \sigma^{-1} Q_0 x \text{ mod } a,$$ and $$ \qi( \sigma^{-1}x) = \qi( \sigma^{-1})x + a \qo( \sigma^{-1}) Q_0x + \sigma^{-1} \qi(x) \equiv \sigma^{-1} \qi(x)\text{ mod } a.$$
\end{rk}

For $M \in \ste(1)\dashmod mod_{Q_0}$, the short exact sequence $ \rp \stackrel{a}{ \inj} \rp \surj F$ is $( \E, \lo)$-exact. We want to understand the long exact sequence in $H_{01}^{\star}$ associated to it.

\begin{rk}
The following results can be seen as the algebraic consequence of the cofibre sequence of $\gr$-spaces
$$ \gr_+ \rightarrow S^0 \rightarrow S^{\alpha}.$$ In the case when the $\ste(1)$-module considered is the cohomology of a space $X$, this can be obtained by keeping track of the cofibre sequence 
$$ \gr_+ \wedge X \rightarrow S^0 \wedge X \rightarrow S^{\alpha} \wedge X$$
through the entire construction.
\end{rk}

\begin{lemma} \label{lemma_h01msigma}
There is a natural isomorphism of functors $ \E\dashmod mod \rightarrow \F\dashmod mod$
$$ H_{01}^{ \star} \circ i \circ F \cong Ker_{Q_0} \cap Ker_{ Q_1} \oplus \sigma^{-1}H_{01}^*(-)[ \sigma^{-1}],$$
where $ \sigma^{-1}H_{01}^*(-)[ \sigma^{-1}]$ denotes the $RO( \gr)$-graded $\F$-vector space valued functor  $H_{01}^*(-) \otimes \sigma^{-1} \F[ \sigma^{ -1}]$, and $i :  \E[ \sigma^{-1}]\dashmod mod \rightarrow \E\dashmod mod$ is the forgetful functor.
\end{lemma}

\begin{proof}
Let $M$ be a $ \ste(1)$-module. Proposition \ref{pro_cartancoeff} and the Corollary \ref{corr_actionstecoeff} provides the action of $ \qi$ and $ \qo$ on $ \rp M$. We now give an explicit description modulo $a$:
let $ \sigma^{-n} \otimes m \in \rp M$. One has
\begin{eqnarray*}
\qo( \sigma^{-n} \otimes m) &=& \qo( \sigma^{-n}) \otimes m + \sigma^{-n} \otimes Q_0m \\
& \equiv & \sigma^{-n} \otimes Q_0m \text{ mod } a
\end{eqnarray*}
because $Im_{ \qo}( \F[a , \sigma^{-1}]) \subset a \F[a , \sigma^{-1}]$, thus $Ker_{ \qo} \circ F \cong Ker_{Q_0}(-)[ \sigma^{-1}]$.
Moreover,  
\begin{eqnarray*}
 & & \qi( \sigma^{-n} \otimes m)  \\ &=& \qi( \sigma^{-n}) \otimes m + a \qo( \sigma^{-n}) \otimes Q_0m + a \sigma^{-n} \otimes Sq^2m + \sigma^{-n-1} \otimes Q_1m \\
& \equiv & \sigma^{-n-1} \otimes Q_1m \text{ mod } a,
\end{eqnarray*}
so $Ker_{ \qi} \circ Ker_{ \qo} = Ker_{ Q_1} \circ Ker_{ Q_0}(-)[ \sigma^{-1}]$ and $Im_{ \qi} \circ Ker_{ \qo} = \sigma^{-1} Im_{ Q_1} \circ Ker_{ Q_0}(-)[ \sigma^{-1}].$
The natural isomorphism $H_{01}^* \cong (Ker_{Q_0} \cap Ker_{ Q_1}) / (Im_{Q_1} \circ Ker_{Q_0})$ given in Lemma  \ref{lemma_identh01classique} then provides the asserted isomorphism. \\
\end{proof}

\begin{lemma} \label{lemma_sel_xa}
Let $M$ be a $Q_0$-acyclic $ \ste(1)$-module.
Then, there is a long exact sequence

$$ \hdots \rightarrow H_{01}^{ \star- \alpha}( \rp M) \stackrel{a}{ \rightarrow} H_{01}^{ \star}( \rp M) \stackrel{ \rho}{ \rightarrow} H_{01}^{ \star}( FM) \stackrel{ \beta}{ \rightarrow}  H_{01}^{ \star+2}( \rp M) \rightarrow \hdots $$
where $\rho$ denotes reduction modulo $a$.
\end{lemma}

\begin{proof}
The $ \ste(1)$-module $M$ being $Q_0$-acyclic, $(a) \rp M \cong \Sigma^{ \alpha} \rp(-)$ is $ \qo$-acyclic, and thus injective as a $ \lo$-module. Thus the underlying exact sequence of $ \lo$-modules is split, and $ 0 \rightarrow (a) \rp M \rightarrow  \rp M \rightarrow \rp M / a \rightarrow 0$ is a $( \E, \lo)$-exact sequence.

Consequently, Proposition \ref{pro_selh01} provides a long exact sequence:

$$ \hdots \rightarrow H_{01}^{ \star}( (a) \rp M) \stackrel{ }{ \rightarrow} H_{01}^{ \star}( \rp M) \stackrel{ \rho}{ \rightarrow} H_{01}^{ \star}( FM) \stackrel{ \beta}{ \rightarrow}  H_{01}^{ \star+2+ \alpha}((a) \rp M) \rightarrow \hdots.$$

Denote $M[\sigma^{-1}] = \bigoplus_{n\geq 0} M \sigma^{-n}$ (adjoining a formal variable $\sigma^{-1}$ to $M$). This makes sense since, when working modulo $a$, all spheres are orientable.
With this notation, $ \E$-module isomorphism $ (a) \rp M \cong \Sigma^{ \alpha} \rp M$ gives:
$$ \hdots \rightarrow H_{01}^{ \star - \alpha }( \rp M) \stackrel{a}{ \rightarrow} H_{01}^{ \star}( \rp M) \stackrel{ \rho}{ \rightarrow} H_{01}^{ \star}( M[ \sigma^{ -1}]) \stackrel{ \beta}{ \rightarrow}  H_{01}^{ \star+2}( \rp M) \rightarrow \hdots.$$
\end{proof}

One can reinterpret the previous lemma as an exact couple, and consider the associated spectral sequence. It is a Bockstein spectral sequence whose first page is isomorphic to $(H_{01}^{ \star} \circ F)(M)[ \tilde{a}]$, where $ \tilde{a}$ is an element of degree $ - \alpha \in RO( \gr)$ and homological degree $1$, and which converges to $(H_{01}^{ \star} \circ R)(M)$.

%

\begin{lemma} \label{lemma_bockstein_xa}
Let $M$ be a $Q_0$-acyclic $ \ste(1)$-module .
Consider the natural isomorphism $$ H_{01}^{ \star}( M[ \sigma^{-1}]) \cong Ker_{Q_0}(M) \cap Ker_{Q_1}(M) \oplus \sigma^{-1} H_{01}^*(M)[ \sigma^{-1}]$$ provided by Lemma \ref{lemma_h01msigma}.
Then, the first differential $d_1$ of the Bockstein spectral sequence associated to the multiplication by the Euler class $a$ in $H_{01}^{ \star}$ acts on $ H_{01}^{ * + k \alpha}( M[ \sigma^{-1}])$ for all $k \geq 0$:
\begin{itemize}
\item as $ \tilde{a} Sq^2$ if $k$ is even,
\item as $ \tilde{a} Sq^{ \rotatebox{180}{$^2$}}$ if $k$ is odd.
\end{itemize}
\end{lemma}

\begin{proof}
With the notations of Lemma \ref{lemma_sel_xa}, the morphism $d_1$ is the composite $$H_{01}^{ \star}( M[ \sigma^{ -1}]) \stackrel{ \beta}{ \rightarrow} H_{01}^{ \star+2}( \rp M)  \stackrel{ \rho}{ \rightarrow} H_{01}^{ \star + 2}( M[ \sigma^{ -1}]).$$
The edge of the exact sequence is given explicitely by the (well-known) description of the Bockstein spectral sequence (see for example \cite[4.1.A]{BG10}).
The end of the proof is a straightforward computation using this description.
\end{proof}

\begin{pro} \label{pro_aactionsurh01}
Let $M$ be a $Q_0$-acyclic $ \ste(1)$-module.
Suppose
\begin{enumerate}
 \item  $Ker_{Q_0}(M) \cap Ker_{Q_1}(M)$ is $Sq^2$-acyclic,
 \item $H_{01}^*(M)$ is $Sq^2$-acyclic
 \item and $H_{01}^*(M)$ is $Sq^{ \rotatebox{180}{$^2$}}$-acyclic.
\end{enumerate}
Then $$Ker_a( H_{01}^{ \star}( \rp M)) = H_{01}^{ \star}( \rp M)$$
and
$$ H_{01}^{ \star}( \rp M) = Ker_{d_1}(H_{01}^{ \star}( M[ \sigma^{-1}])).$$
\end{pro}

\begin{proof}
Consider the  Bockstein spectral sequence associated to the multiplication by the Euler class $a$ in $H_{01}^{ \star}$.
By definition, we have an isomorphism $E^1 \cong H_{01}^{ \star}( M[ \sigma^{ -1}])[ \tilde{a}]$, and the first differential $d_1$ is identified in Lemma \ref{lemma_bockstein_xa}.
Consequently, the hypothesis of the proposition are equivalent to: the Bockstein spectral sequence collapses at page $E^2$, because $E^2$ is concentrated in degrees of the form $ \{0 \} \times RO( \gr) \subset \mathbb{Z} \times RO( \gr)$. The product with $ \tilde{a}$ increases the homological degree, and the $E_2$ page is concentrated in homological degree $0$so, product with $ \tilde{a}$ is trivial on $E^2 = E^{ \infty}$. Therefore, product with $a$ on $H_{01}^{ \star}( \rp M)$, induced by the product with $ \tilde{a}$ on $E^2 = E^{ \infty}$ is trivial too. \\

Thus we have also identified the $E^{ \infty}$ page:
$$ E^{ \infty} = E^{2} = Ker_{d_1}( H_{01}^{ \star}(M[ \sigma^{-1}]).$$
\end{proof}

\section{The computation of $\mathcal{H}^{\star}(V)$: $H_{01}^{\star}R$ on free $\ste(1)$-modules}

The aim of this section is to compute $H^{\star}_{01}\circ R$ on free $\ste(1)$-modules. The result is expressed in Corollary \ref{pro_h01libre}.

\subsection{Duality}

A natural question we address now is the relationship between $H^{\star}_{01}$ and the $ \F$-linear duality functor, using Proposition \ref{pro_compatibilite_dualites}. \\

\begin{lemma} \label{lemma_dualrelatif}
Consider the functor $(-)^{ \vee} : \E\dashmod mod^{op} \rightarrow \E\dashmod mod$.
\begin{enumerate}
 \item $(-)^{ \vee}$ is $( \E, \lo)$-exact,
 \item $(-)^{ \vee}$ sends  $( \E, \lo)$-projective (resp. $( \E, \lo)$-injective) $\E$-modules on $( \E, \lo)$-projective (resp. $( \E, \lo)$-injective) $\E$-modules.
\end{enumerate}
\end{lemma}

\begin{proof}
 The first point is a consequence of the exactness of $(-)^{ \vee}$ and the definition of relative exactness.

The second point uses Proposition \ref{pro_carac_proinjrel}. There is a $ \E$-module isomorphism $( \li )^{ \vee}  \cong \Sigma^{-2- \alpha} \li$.
Thus, for $M$ a $ \E$-module, the dual $( \E \otimes_{ \lo} M)^{ \vee} \cong ( \li \otimes_{ \F} M)^{ \vee} \cong \Sigma^{-2- \alpha} \E \otimes_{ \lo} (M^{ \vee})$ (because $ \li$ is of finite dimension) is also a $( \E, \lo)$-projective module by the last point of Proposition \ref{pro_carac_proinjrel}. Thus, the functor $(-)^{ \vee}$ preserves $( \E, \lo)$-projective $ \E$-modules. \\
\end{proof}

The use of duality makes appear another functor, related to $H_{01}^{ \star}$:

\begin{nota}
 Denote $H^{01}_{ \star}$ the functor $$ \Sigma^{| \qi|}\Le_1( (Id/(Im_{ \qo}  + Im_{ \qi})) ) : \E\dashmod mod \rightarrow \F\dashmod mod.$$
\end{nota}

%

\begin{pro} \label{pro_commh01dualite}
\begin{enumerate}
 \item The following diagram commutes up to natural isomorphism:
$$ \xymatrix{ \E\dashmod mod^{op} \ar[d]_{H_{01}^{op}} \ar[r]^{ \vee} & \E\dashmod mod \ar[d]^{H^{01}} \\ \F\dashmod mod^{op} \ar[r]_{ \vee} &  \F\dashmod mod.} $$
\item Moreover, for all $i \in \Z$, there is a natural isomorphism between $H^{01}_{ \star}$ and $$ \Sigma^{i | \qi|}\hat{\Le_i}( (Id/(Im_{ \qo}  + Im_{ \qi})) ) : \E\dashmod mod \rightarrow \F\dashmod mod.$$
\end{enumerate}
\end{pro}

\begin{proof}
Consider the diagram
$$ \xymatrix{ \E\dashmod mod^{op}  \ar[d]_{Hom_{ \E}( \F, -)} \ar[r]^{ \vee} & \E\dashmod mod \ar[d]^{(Id/(Im_{ \qo}  + Im_{ \qi}))} \\ 
\F\dashmod mod^{op} \ar[r]_{ \vee} &  \F\dashmod mod, } $$
which is commutative because of the natural isomorphisms $Hom_{ \E}( \F, -)^{ \vee} \cong (Ker_{ \qo} \cap Ker_{ \qi})^{\vee} \cong (Id/(Im_{ \qo}  + Im_{ \qi})) ((-)^{ \vee})$.
By Lemma \ref{lemma_dualrelatif}, the dual of a $( \E, \lo)$-projective resolution is a $( \E, \lo)$-injective resolution, consequently there is a natural isomorphism
\begin{equation} \label{eqn_comparaisonliri}
 ( \R^i(\Hom{ \E}( \F, -)))^{ \vee} \cong \Le_i( (Id/(Im_{ \qo}  + Im_{ \qi}))) \circ (-)^{ \vee}
\end{equation}
the first point now follows from the definition of $H_{ 01}$ and $H^{01}$. \\

For the second point, for $ i \geq 1$ the result follows from the same isomorphism and Proposition \ref{pro_identh01}. We deduce an isomorphism for all $i \in \Z$ between $ \widehat{\Le_i}( (Id/(Im_{ \qo}  + Im_{ \qi})))$ and $ \Sigma^{| \qi|} \widehat{\Le_{i+1}}( (Id/(Im_{ \qo}  + Im_{ \qi})))$. The result follows.
\end{proof}

\begin{de}
Define $ \E\dashmod mod_{ \qo}$ to be the full subcategory of $ \E\dashmod mod$ consisting of $ \qo$-acyclic objects.
\end{de}

\begin{lemma} \label{lemma_deriveresbeta0}
Let $F, G : \E\dashmod mod \rightarrow \ab$ be two left or right  $( \E, \lo)$-exact functors such that the restriction of $F$ and $G$ to $ \E\dashmod mod_{ \qo}$ are the same.

Then, the Tate derived functors of $F$ and $G$ coincide on $ \E\dashmod mod_{ \qo}$.
\end{lemma}

\begin{proof}
By Lemma \ref{cor_resbeta0acyclique}, the functor $T_{ \bullet}$ restricts to $$ \E_{ \qo} \rightarrow Ch_{ \Z}( \E\dashmod mod_{ \qo}).$$

Consider the case when $F$ and $G$ are right $( \E, \lo)$-exact. Then, the functors $F( T_{ \bullet})$ and $G(T_{ \bullet})$ are naturally isomorphic, so there is a natural isomorphism $ \widehat{\Le}_i(F) = H_i(F( T_{ \bullet})) \cong H_i(G(T_{ \bullet})) = \widehat{\Le}_i(G)$.

The other case, $F$ and $G$ being left $( \E, \lo)$-exacts is analogous.
\end{proof}

\begin{lemma} \label{lemma_calculh01}
For all $i \geq 0$, there is a natural isomorphism
$$ H_{01}^{ \star} \cong \Sigma H^{01}_{ \star}$$
as functors $$ \E\dashmod mod_{ \qo} \rightarrow \F[a, \sigma^{-4}]\dashmod mod.$$
\end{lemma}

\begin{proof}
Restricting to the category $ \E\dashmod mod_{ \qo}$, there is a natural isomorphism between functors $ \E\dashmod mod_{ \qo} \rightarrow \F[a, \sigma^{ -4}]\dashmod mod$:

$$ \Sigma^{| \qo|} Coker_{ \qo} \stackrel{ \qo}{ \rightarrow} Ker_{ \qo}.$$

Consider the diagram
$$ \xymatrix{ \Sigma^{ | \qi|}Ker_{ \qo} \ar[r]^{ \qi} & Ker_{ \qo}\\
 \Sigma^{ 1+| \qi|}Coker_{ \qo} \ar[u]_{ \cong}  \ar[r]_{ \qi} & \Sigma Coker_{ \qo},  \ar[u]^{ \cong} }$$

Where the vertical isomorphisms are induced by $ \qo$. Thus, the commutativity of $ \qo$ and $ \qi$ imply the commutativity of the diagram. Thus, there are isomorphisms $ \Sigma Id/(Im_{ \qo}  + Im_{ \qi}) \cong Coker_{ \qi} \circ Ker_{ \qo} $. This gives a $4$-terms exact sequence of functors $  \E\dashmod mod_{ \qo} \rightarrow \F[a, \sigma^{ -4}]\dashmod mod$
\begin{equation} \label{eqn_k01k0k0c01}
 \Sigma^{ | \qi|} Ker_{ \qi} \circ Ker_{ \qo} \inj \Sigma^{ | \qi|} Ker_{ \qo} \stackrel{ \qi}{ \rightarrow} Ker_{ \qo} \stackrel{ \qo^{-1}}{ \surj} \Sigma Id/(Im_{ \qo}  + Im_{ \qi}).
\end{equation}

And Lemma \ref{lemma_deriveresbeta0} applies to the natural isomorphism $ \Sigma Id/(Im_{ \qo}  + Im_{ \qi}) \cong Coker_{ \qi} \circ Ker_{ \qo}$ and provides a natural isomorphism $\Le_i( \Sigma Id/(Im_{ \qo}  + Im_{ \qi})) \cong \Le_i(Coker_{ \qi} \circ Ker_{ \qo})$. \\

Denote for short $T^M_{\bullet} = T^{\F}_{\bullet} \otimes M$ the Tate resolution for $M$. The study of the two spectral sequences associated to the bicomplex

$$ \xymatrix{ 
\vdots & \vdots \\
\Sigma^{ | \qi|}Ker_{ \qo}( T^{M}_{ \bullet +1}) \ar[r]^{ \qi} \ar[u] & Ker_{ \qo}( T^{M}_{ \bullet +1})  \ar[u] \\
\Sigma^{ | \qi|}Ker_{ \qo}( T^{M}_{ \bullet})  \ar[r]^{ \qi} \ar[u]  & Ker_{ \qo}( T^{M}_{ \bullet})  \ar[u] \\
\Sigma^{ | \qi|}Ker_{ \qo}( T^{M}_{ \bullet -1})  \ar[r]^{ \qi}  \ar[u] & Ker_{ \qo}( T^{M}_{ \bullet -1}) \ar[u]  \\ 
 \vdots \ar[u] & \vdots \ar[u] }$$
 
concludes the proof.

\end{proof}

\subsection{Free $\ste(1)$-modules}

The aim of this subsection is to compute $H_{01}^{ \star}( R F)$, for free $ \ste(1)$-modules $F$. The result is given in Proposition \ref{pro_h01libre} which states that free $\ste(1)$ module have almost no $H_{01}^{ \star}( R(-))$. This is essential to the end of our computation, as it allows to understand $H_{01}^{ \star}( RM)$, for any $\ste(1)$-module $M$, knowing only $M$ up to stable isomorphism in a large range of bidegrees.

\begin{lemma} \label{lemma_h01ste1}
\begin{enumerate}
 \item There is a $ \Lambda_{ \F}(Sq^2)$-module isomorphism 
 $$Ker_{Q_0}( \ste(1)) \cap Ker_{Q_1}( \ste(1)) \cong \{  Sq^2  Sq^2 , Sq^2  Sq^2 Sq^2 \} \F,$$ 
 with $Sq^2 ( Sq^2 Sq^2 ) = Sq^2 Sq^2 Sq^2$. In particular, this module is $Sq^2$-acyclic.
 \item The $ \Lambda(Sq^2)$-module $H_{01}^*( \ste(1))$ is trivial.
\end{enumerate}
\end{lemma}

\begin{proof}
We use that the image of $ \ste(1)$ by the forgetful functor$ \ste(1)\dashmod mod \rightarrow \Lambda_{\F}(Q_0, Q_1)\dashmod mod$ is isomorphic to $ \Lambda_{\F}(Q_0, Q_1) \oplus \Sigma^2 \Lambda_{\F}(Q_0, Q_1)$, a basis of this $ \Lambda_{\F}(Q_0, Q_1)$-module consists on $1$ and $Sq^2$. The $\F$-vector space structure of$Ker_{Q_0}( \ste(1)) \cap Ker_{Q_1}( \ste(1))$ and $H_{01}^*( \ste(1))$ follows.

Now, the action of $Sq^2$ on $Ker_{Q_0}( \ste(1)) \cap Ker_{Q_1}( \ste(1))$ is induced by the action of $Sq^2$ on the generators of $  \ste(1)$ as a $ \Lambda_{\F}(Q_0, Q_1)$-module: $1$ and $Sq^2$.
\end{proof}

\begin{pro} \label{pro_h01ste1}
 There is an identification
$$ H_{01}^{ \star}( R \ste(1)) = Sq^2  Sq^2 Sq^2 \F \oplus \sigma^2  Sq^1 \F.$$
\end{pro}

\begin{proof}
For degree reasons, there is  splitting $R \cong \rp(-) \oplus \rn(-)$. As the functor $H_{01}^{ \star}$ is additive, there is also a splitting $$H_{01}^{ \star} \circ R \cong H_{01}^{ \star} \circ \rp(-) \oplus  H_{01}^{ \star} \circ \rn(-).$$

Then, Lemma \ref{lemma_h01ste1} provides the hypothesis of Proposition \ref{pro_aactionsurh01} giving $$ H_{01}^{ \star}( \rp \ste(1)) = Sq^2  Sq^2 Sq^2 \F .$$ \\

To complete the computation, we use the duality properties of Proposition \ref{pro_commh01dualite} and Proposition \ref{pro_compatibilite_dualites}. We get
\begin{eqnarray*}
 H_{01}^{ \star}(R \ste(1))^{ \vee} & \cong & H^{01}_{ \star}( (R \ste(1))^{ \vee}) 
 \end{eqnarray*}
by the first point of Proposition \ref{pro_commh01dualite},
 
 \begin{eqnarray*}
H^{01}_{ \star}( (R \ste(1))^{ \vee}) & \cong & H^{01}_{ \star}( \Sigma^{-2+2 \alpha}R ( \ste(1)^{ \vee})) 
\end{eqnarray*}
by Proposition \ref{pro_compatibilite_dualites},

 \begin{eqnarray*}
H^{01}_{ \star}( \Sigma^{-2+2 \alpha}R ( \ste(1)^{ \vee}))  & \cong & H^{01}_{ \star}( \Sigma^{-2+2 \alpha}R ( \Sigma^{-6} \ste(1)))
\end{eqnarray*}
because $ \ste(1)^{ \vee} \cong \Sigma^{-6} \ste(1)$. Finally, Lemma \ref{lemma_calculh01} gives
 \begin{eqnarray*}
H^{01}_{ \star}( \Sigma^{-2+2 \alpha}R ( \Sigma^{-6} \ste(1))) & \cong & \Sigma^{-1}H_{01}^{ \star}( \Sigma^{-2+2 \alpha}R ( \Sigma^{-6} \ste(1))) \\
& \cong &  \Sigma^{-9+2 \alpha} H_{01}^{ \star}(R \ste(1)).
\end{eqnarray*}

For degree reasons, this isomorphism is compatible with the splitting $$H_{01}^{ \star} \circ R \cong H_{01}^{ \star} \circ \rp(-) \oplus  H_{01}^{ \star} \circ \rn(-),$$ and gives two isomorphisms

$$ H_{01}^{ \star} ( \rp( \ste(1)))^{ \vee} \cong \Sigma^{-9+2 \alpha} H_{01}^{ \star}( \rn( \ste(1)))$$
and
$$ H_{01}^{ \star} ( \rn( \ste(1)))^{ \vee} \cong \Sigma^{-9+2 \alpha} H_{01}^{ \star}( \rp( \ste(1))).$$

Consequently, $H_{01}^{ \star}( \rn( \ste(1)))$ is a one dimensional vector space, generated by an element in degree $-6+9-2 \alpha = 3-2 \alpha$. To conclude, see that $ \sigma^2 Sq^1$ is in $$Ker_{ \qo}(  \rn( \ste(1))) \cap Ker_{ \qi} ( \rn( \ste(1))),$$ which is a consequence of the fact that $H\mf_{\gr}^{*-\alpha}$ is trivial, and that it cannot be in the image of $ \qi$. The class it represents is thus a generator of $H_{01}^{ \star}( \rn( \ste(1)))$.
\end{proof}

\begin{corr} \label{pro_h01libre}
Let $F$ be a free $ \ste(1)$-module. Then
$$ H_{01}^{ \star} ( F)  \cong (F \otimes_{ \ste(1)} \F) \otimes_{ \F} \left( Sq^2  Sq^2 Sq^2 \F \oplus \sigma^2 Sq^1 \F \right).$$
In particular, this $ \F[a, \sigma^{-4}]$-module is concentrated in degrees of the form $ \mathbb{Z} \subset RO( \gr)$ and $ \mathbb{Z} - 2 \alpha \subset RO( \gr)$.
\end{corr}

\begin{proof}
The result is essentially given by Proposition \ref{pro_h01ste1} and the additivity of the functors $H_{01}^{ \star}$ and $R$.
\end{proof}

\section{The computation of $\mathcal{H}^{\star}(V)$ : the stable category} \label{section_stable}

In this section, we study the relationship between $H^{\star}_{01} R$ and the stable category of $\ste(1)$-modules. When the $\ste(1)$-module under consideration is $Q_0$-acyclic, there are nice tools to perform the computation, given in Proposition \ref{pro_calculh01} and Proposition \ref{pro_amodsurh01}.

\subsection{The computational tools}

The computation we did in Proposition \ref{pro_h01libre} says that free $\ste(1)$-modules have very small $H_{01}^{\star}R$-homology. This motivates the study of $H_{01}^{\star}R$ by neglecting free modules as a first approximation, that is to study the stable category of $\ste(1)$-modules. Good references are Margolis' book \cite[chapter 14]{Mar83}, and Palmieri \cite{Pal01}. \\

For a quick overview of the required material about the stable category, the reader can consult \cite[Definition 2.4, Propositions 2.5 and 2.6]{BrOssa} and the subsequent paragraphs. We follow the notations of \textit{loc. cit.}.

\begin{nota} \label{de_omega}
Let $M$ and $N$ be two $\ste(1)$-modules.
\begin{enumerate}
\item Let $\Ccal$ be a subcategory of $\ste(1) \dashmod mod$. $St(\Ccal)$ denotes the stable category associated to $\Ccal$. We use $\cong$ and $\simeq$ for the equivalences and the stable equivalences respectively.
\item $M^{red}$ denotes the reduced module associated to $M$, that is the isomorphism class of any module $M'$ stably isomorphic to $M$ which contains no free summands.
\item The desuspension in the stable category is denoted $\Omega$, and $\Omega_r := (\Omega(-))^{red}$.
\end{enumerate}
\end{nota}

The work we did in Section \ref{sec_hoistar}$\dashmod$\ref{section_stable} gives a very precise description of $H_{01}^{\star}\circ R$ for reduced $Q_0$-acyclic $\ste(1)$-modules, given in Proposition \ref{pro_calculh01}, which computes the $\F$-vector space structure of $H_{01}^{\star}\circ R$, and Proposition \ref{pro_amodsurh01}, which recover the $\F[a]$-module structure of it.

\begin{pro} \label{corr_h01loop}
Let $M$ be a $Q_0$-acyclic $ \ste(1)$-module. Let $F \surj M$ be the projective cover of $M$. There are isomorphisms
\begin{enumerate}
\item $H_{01}^{*}(R \Omega_r M) \cong H_{01}^{*}(R F),$
\item $ H_{01}^{*+2-2 \alpha}(R F) \cong H_{01}^{*+2-2 \alpha}(R M),$
\item and, for all $k \not\in \{ -2,-1 \}$, $ H_{01}^{*+k \alpha}(R M) \stackrel{ \cong}{ \rightarrow} H_{01}^{*+2+(k+1) \alpha}(R \Omega_r M)$,
\end{enumerate}
where $* \in \Z$.
\end{pro}

\begin{proof}
Observe that the kernel of $F \surj M$ is $\Omega_r M$.
By $Q_0$-acyclicity, Proposition \ref{pro_selh01} applies to the short exact sequence $ 0 \rightarrow  \Omega_r M \rightarrow F \rightarrow M \rightarrow 0$, to provide long exact sequences
$$ \hdots \rightarrow H_{01}^{ \star}( R \Omega_r M) \rightarrow H_{01}^{ \star}( R F) \rightarrow H_{01}^{ \star}( R M)  \rightarrow H_{01}^{ \star+2+ \alpha}( R \Omega_r M) \rightarrow \hdots .$$
Moreover, sparcity of $H_{01}^{ \star}( R F)$, for $F$ free, as expressed in Proposition \ref{pro_h01libre} identify many terms in this long exact sequence to zero.

Observe also that, for any $ \ste(1)$-module $N$, one has $ RN^{* - \alpha} = 0$, and thus $H_{01}^{ * - \alpha}(RN) = 0$. \\

Consequently, in the portion of long exact sequence \\
$$ \hdots \rightarrow H_{01}^{ *-2- \alpha}( R M) \rightarrow H_{01}^{ *}( R \Omega_r M) \rightarrow H_{01}^{ *}( R F) $$ $$  \rightarrow H_{01}^{*}( R M)  \rightarrow H_{01}^{ *+2+ \alpha}( R \Omega_r M) \rightarrow H_{01}^{ *+2+ \alpha}( R F) \rightarrow \hdots, $$ \\
the terms $H_{01}^{ *-2- \alpha}( R M)$ and $H_{01}^{ *+2+ \alpha}( R F)$ are trivial (by Proposition \ref{pro_h01libre} for the second one), providing a $4$-terms exact sequence $ 0 \rightarrow H_{01}^{*}(R \Omega_r M) \rightarrow H_{01}^{*}(R F) \rightarrow H_{01}^{*}(R M) \rightarrow H_{01}^{*+2+ \alpha}(R \Omega_r M) \rightarrow 0 $. We now show that it splits into two isomorphisms. \\

Suppose that the morphism $H_{01}^{*}(R \Omega_r M) \rightarrow H_{01}^{*}(R F)$ is not surjective. Proposition \ref{pro_h01libre} imply the existence of an element in the class of $Sq^2 Sq^2 Sq^2 v \in H_{01}^{ \star}(F)$, for some $v \in F$ which is not in the image of $ \Omega_r M \rightarrow F$. For degree reasons, the class of $Sq^2 Sq^2 Sq^2 v$ contains one element, since $ \qi$ acts trivially in this degree. Thus $Sq^2 Sq^2 Sq^2 v$ is send to some $Sq^2 Sq^2 Sq^2 m$ in $M$ by the $ \ste(1)$-module morphism  $F \surj M$, and a copy of $ \ste(1)$ splits off: $(F \surj M) = f \oplus Id_{ \ste(1)} : F' \oplus \ste(1) \rightarrow M' \oplus \ste(1)$, thus $M$ is not reduced. Contradiction. We conclude that $H_{01}^{*}(R \Omega_r M) \rightarrow H_{01}^{*}(R F)$ is surjective, providing the isomorphisms $(1)$ and $(3)$ when $k = 0$. \\

Now, consider the portion
$$ \hdots \rightarrow H_{01}^{ *-3 \alpha}( R F) \rightarrow H_{01}^{ *-3 \alpha}( R M) \rightarrow H_{01}^{ *+2-2 \alpha}( R \Omega_r M) $$ $$ \rightarrow H_{01}^{ *+2-2 \alpha}( R F) \rightarrow H_{01}^{*+2-2 \alpha}( R M)  \rightarrow H_{01}^{ *+4- \alpha}( R \Omega_r M) \rightarrow  \hdots, $$
of the previous exact sequence. Once again, the terms $ H_{01}^{ *-3 \alpha}( R F)$ and $ H_{01}^{ *+4- \alpha}( R \Omega_r M)$ are trivial, giving a $4$-term exact sequence $$ 0 \rightarrow H_{01}^{*-3 \alpha}(R M) \rightarrow H_{01}^{*+2-2 \alpha}(R \Omega_r M) \rightarrow H_{01}^{*+2-2 \alpha}(R F) \rightarrow H_{01}^{*+2-2 \alpha}(R M) \rightarrow 0.$$ \\

The fact that  $M$ is split and $F \surj M$ minimal implies that $Ker(F \surj M) =  \Omega_r M$ is reduced. Consequently, the short exact sequence
$$ M^{ \vee} \inj F^{ \vee} \surj ( \Omega_r M)^{ \vee}$$
satisfy the hypothesis of the previous point, providing a $4$-term exact sequence
$$ 0 \rightarrow H_{01}^{*}(R  M^{ \vee}) \rightarrow H_{01}^{*}(R F^{ \vee}) \rightarrow H_{01}^{*}(R ( \Omega M)^{ \vee}) \rightarrow H_{01}^{*+2+ \alpha}(R ( \Omega M)^{ \vee}) \rightarrow 0 .$$
Consequently, there are isomorphisms
$$ H_{01}^{*}(R (M^{ \vee})) \cong H_{01}^{*+2+ \alpha}(R (( \Omega_r M)^{ \vee}))$$
and Lemma \ref{lemma_calculh01} together with Proposition \ref{pro_commh01dualite} yields isomorphisms
$$ H_{01}^{*-3 \alpha}(R M) \cong H_{01}^{*+2-2 \alpha}(R \Omega_r M).$$

We finish the proof by the easier cases. Let $k \not\in \{ -3,-2,-1,0 \}$, Proposition \ref{pro_h01libre} directly gives that both $H_{01}^{ *+k \alpha}( R F)$ and $H_{01}^{ *+(k+1) \alpha}( R F)$ are trivial, and the long exact sequence
$$ H_{01}^{ *+k \alpha}( R F) \rightarrow H_{01}^{ *+k \alpha}( R M) \rightarrow H_{01}^{ *+2+(k+1) \alpha}( R \Omega_r M) \rightarrow H_{01}^{ *+2+(k+1) \alpha}( R F)$$
gives the desired isomorphisms $ H_{01}^{*+k \alpha}(R M) \stackrel{ \cong}{ \rightarrow} H_{01}^{*+2+(k+1) \alpha}(R \Omega_r M)$. \\
\end{proof}

\begin{nota}
 Denote $ \kera : \ste(1)\dashmod mod \rightarrow \F\dashmod mod$ the socle, {\em i.e.} the functor $Hom_{ \ste(1)}( \F, -)$.
\end{nota}

\begin{pro} \label{pro_calculh01}
Let $M$ be a reduced $Q_0$-acyclic $ \ste(1)$-module.
 For all $n \geq 0$,
\begin{itemize}
 \item $H_{01}^{*+n \alpha}(RM) \cong \Sigma^{2n} \kera( \Omega_r^{-n}(M))$
 \item $H_{01}^{*-(n+2) \alpha}(RM) \cong  \Sigma^{-2n-5} \kera( \Omega_r^{n+2}(M))$
 \item $H_{01}^{*- \alpha}(RM) =0$.
\end{itemize}
\end{pro}

\begin{proof}
Let $M$ be a reduced $ \ste(1)$-module. Choose a minimal free module $F$ such that there is an epimorphism $F \surj M$. In these conditions, there is a short exact sequence
$$ \Omega_r(M) \inj F \surj M.$$

Consequently, Proposition \ref{corr_h01loop} apply to
$$ \Omega_r M \inj F \surj M.$$

The first step is to compute $H_{01}^{ \star}(RM)$ in integer grading: the $\F$-vector space $H_{01}^{* + 0 \alpha}(RM)$. By sparsity, no element of $M \cong 1 \otimes M \inj H \mf^{ \star}_{\gr} \otimes M \cong RM$ can be hit by $Im_{ \qi}(RM)$. Consequently $$H_{01}^{*+0 \alpha}(RM) = \left( Ker_{ \qo}(RM) \cap Ker_{ \qi}(RM) \right)^{*+0 \alpha}.$$ Now, Proposition \ref{thm:formula_qoi} give $$ \left( Ker_{ \qo}(RM) \cap Ker_{ \qi}(RM) \right)^{*+0 \alpha} = Ker_{Sq^1}(M) \cap Ker_{ Sq^2}(M) \subset RM$$ by definition of the action of $\qi$ on $RM$. The ring $ \ste(1)$ being generated by $Sq^1$ and $Sq^2$, we have $Ker_{Sq^1}(M) \cap Ker_{ Sq^2}(M) = \kera(M)$. \\

We now show $H_{01}^{*+n \alpha}(RM) = \Sigma^{2n} \kera( \Omega_r^{-n}(M))$ by induction on $n$. Let $n \geq 1$.
\begin{itemize}
 \item For $n=1$, it is contained in Proposition \ref{corr_h01loop}.
 \item For $n \geq 1$, $(3)$ and the last assertion of corollary \ref{corr_h01loop} applied to the short exact sequence $$ M \inj F \surj \Omega_r^{-1}(M) $$ gives $$H_{01}^{*+n \alpha}(RM)  \cong  H_{01}^{*-2+(n-1) \alpha}(R \Omega_r^{-1} M) \cong \Sigma^{2n} \kera( \Omega_r^{-n}(M)),$$ where the last isomorphism is provided by the induction hypothesis.
\end{itemize}

\vspace{0.3cm}

The last isomorphism follows by a similar argument.

%
%

We already knew that $H_{01}^{*- \alpha}(RM) = 0$ by the structure of the coefficient ring $H\mf^{\star}_{\gr}$ and the definition of $R$. \\
\end{proof}

\begin{pro} \label{pro_amodsurh01}
Let $M$ be a reduced $Q_0$-acyclic $ \ste(1)$-module. Denote $ \sigma^2 M[ \sigma]$ the $ \E$-module $ Ker(a : \rn(M) \rightarrow \rn(M))$. There is a natural isomorphism:
$$ H_{ 01}^{ \star}( \sigma^2 M[ \sigma]) =  \sigma^2 M/(Im_{Q_1}(M) + Im_{ Q_0}(M)) \oplus \sigma^{3}H_{01}^*(M)[ \sigma].$$

Moreover, applying  $H_{01}^{ \star}$ provides two exact sequences
$$ H_{01}^{ \star- \alpha}( \rp M) \stackrel{a}{ \rightarrow} H_{01}^{ \star}( \rp M) \rightarrow  Ker_{Q_1} Ker_{ Q_0}(M) \oplus \sigma^{-1}H_{01}^*(M)[ \sigma^{-1}]$$
and 
$$ \sigma^2 M/(Im_{Q_1}(M) + Im_{ Q_0}(M)) \oplus \sigma^{3}H_{01}^*(M)[ \sigma] \rightarrow H_{01}^{ \star}( \rn M) \stackrel{a}{ \rightarrow} H_{01}^{ \star+ \alpha}( \rn M).$$
\end{pro}

\begin{proof}
The first identification follows from the formula $ \qi(m) = \sigma^{-1} Q_1m + a Sq^2m$ for $m \in M \subset RM$ and Cartan formulae analogously to Lemma \ref{lemma_h01msigma}. 

Now, the two desired exact sequences are
\begin{itemize}
\item the long exact sequence provided by Lemma \ref{lemma_sel_xa},
\item the long exact sequence obtained by applying $H_{01}^{ \star}$ to the  short exact sequence of $\qo$-acyclic $ \E$-modules
$$ \sigma^2 M[ \sigma] \inj \rn(M) \stackrel{a}{ \surj} \rn(M) $$
which is a short $( \E, \lo)$-exact sequence of $ \E$-module (same argument as in the proof of Lemma \ref{lemma_sel_xa}).
\end{itemize}
\end{proof}

\section{A computation of $\mathcal{H}^{\star}(V)$}

In this section, we compute $\mathcal{H}^{\star}(V)$, given in Corollary \ref{cor:hv}.

\subsection{The stable equivalence class of $ \widetilde{H \F}^*(BV)$}

\begin{lemma} \label{lemma_kunneth_hbv}
Let $P $ be the $ \ste(1)$-module $ \widetilde{H \F}^*(B \gr) = x\F[x]$ for a class $x$ in degree one.
There is a $ \ste(1)$-module isomorphism
$$ \widetilde{H \F}^*(BV_n) \cong \bigoplus_{i = 1}^{n} \left( P^{ \otimes i} \right)^{ \oplus  \begin{pmatrix} n \\ i \end{pmatrix}} . $$
\end{lemma}

\begin{proof}
Use $BV_n = (B \gr)^{ \times n}$ donne ${BV_n}_+ = (B \gr)_+^{ \wedge n}$, and the Künneth formula.
\end{proof}

The study of the stable equivalence class of $P^{ \otimes i}$ was done in \cite{BrOssa}. We now recall the results we use in our computation.

\begin{pro}[\emph{\cite[Corollary 3.3, Theorem 4.3]{BrOssa}}]  \label{pro_pnomega} \label{pro_periodiciteomega}
In the stable $ \ste(1)$-module category, $P^{ \otimes(n+1)}  \simeq \Omega^n \Sigma^{-n} P$, and $ \Omega^4P \simeq \Sigma^{12} P$.
\end{pro}

\begin{figure}
 
\definecolor{qqqqff}{rgb}{0,0,1}
\definecolor{cqcqcq}{rgb}{0.75,0.75,0.75}

\begin{tikzpicture}[line cap=round,line join=round,>=triangle 45,x=1.0cm,y=1.0cm]
\clip(1.65,-3.42) rectangle (13,-0.84);
\draw (2,3)-- (4,3);
\draw (6,3)-- (8,4);
\draw (8,2)-- (10,3);
\draw (12,3)-- (14,3);
\draw [shift={(4,1.5)}] plot[domain=0.64:2.5,variable=\t]({1*2.5*cos(\t r)+0*2.5*sin(\t r)},{0*2.5*cos(\t r)+1*2.5*sin(\t r)});
\draw [shift={(8,7.36)}] plot[domain=4.28:5.14,variable=\t]({1*4.8*cos(\t r)+0*4.8*sin(\t r)},{0*4.8*cos(\t r)+1*4.8*sin(\t r)});
\draw [shift={(12,4.5)}] plot[domain=3.79:5.64,variable=\t]({1*2.5*cos(\t r)+0*2.5*sin(\t r)},{0*2.5*cos(\t r)+1*2.5*sin(\t r)});
\draw [shift={(7,6.5)}] plot[domain=4:4.93,variable=\t]({1*4.61*cos(\t r)+0*4.61*sin(\t r)},{0*4.61*cos(\t r)+1*4.61*sin(\t r)});
\draw [shift={(9,-0.5)}] plot[domain=0.86:1.79,variable=\t]({1*4.61*cos(\t r)+0*4.61*sin(\t r)},{0*4.61*cos(\t r)+1*4.61*sin(\t r)});
\draw (2.79,3.08) node[anchor=north west] {$\huge{Sq^1}$};
\draw (3.52,4.93) node[anchor=north west] {$\huge{Sq^2}$};
\draw (2,-2)-- (3,-2);
\draw (4,-2)-- (5,-2);
\draw (6,-2)-- (7,-2);
\draw (8,-2)-- (9,-2);
\draw (10,-2)-- (11,-2);
\draw (12,-2)-- (13,-2);
\draw (14,-2)-- (15,-2);
\draw (16,-2)-- (17,-2);
\draw [shift={(4,-2.79)}] plot[domain=0.67:2.48,variable=\t]({1*1.27*cos(\t r)+0*1.27*sin(\t r)},{0*1.27*cos(\t r)+1*1.27*sin(\t r)});
\draw [shift={(5,-1.32)}] plot[domain=3.74:5.68,variable=\t]({1*1.21*cos(\t r)+0*1.21*sin(\t r)},{0*1.21*cos(\t r)+1*1.21*sin(\t r)});
\draw [shift={(8,-2.83)}] plot[domain=0.69:2.45,variable=\t]({1*1.3*cos(\t r)+0*1.3*sin(\t r)},{0*1.3*cos(\t r)+1*1.3*sin(\t r)});
\draw [shift={(9,-1.35)}] plot[domain=3.72:5.71,variable=\t]({1*1.19*cos(\t r)+0*1.19*sin(\t r)},{0*1.19*cos(\t r)+1*1.19*sin(\t r)});
\draw [shift={(12,-2.87)}] plot[domain=0.72:2.43,variable=\t]({1*1.33*cos(\t r)+0*1.33*sin(\t r)},{0*1.33*cos(\t r)+1*1.33*sin(\t r)});
\draw [shift={(13,-1.28)}] plot[domain=3.77:5.66,variable=\t]({1*1.23*cos(\t r)+0*1.23*sin(\t r)},{0*1.23*cos(\t r)+1*1.23*sin(\t r)});
\draw [shift={(16,-2.83)}] plot[domain=0.69:2.45,variable=\t]({1*1.3*cos(\t r)+0*1.3*sin(\t r)},{0*1.3*cos(\t r)+1*1.3*sin(\t r)});
\draw [shift={(17.01,-1.46)}] plot[domain=3.63:5.79,variable=\t]({1*1.14*cos(\t r)+0*1.14*sin(\t r)},{0*1.14*cos(\t r)+1*1.14*sin(\t r)});
\draw (1.9,-1.88) node[anchor=north west] {$\large{x}$};
\draw (2.83,-1.91) node[anchor=north west] {$\large{x^2}$};
\draw (3.77,-1.88) node[anchor=north west] {$\large{x^3}$};
\end{tikzpicture}

\caption{The $ \ste(1)$-module $H \F^*(B \gr)$} \label{fig_hbz2}
\end{figure}

\begin{de} \label{de_pn}
Denote $P_{n+1} = \left( \Omega^n \Sigma^{-n} P \right)^{red}$. 
\end{de}
%

\begin{pro}[\emph{\cite[figure 2 p.6]{BrOssa}}] \label{pro_identpn}
For $i=0,1,2,3$ and $4$ the $\ste(1)$-module $P_i$ is given by
\\
\\
\scalebox{.62}{
$$
\xymatrix{
P_0: &
x^{-1}
\ar[r]
\ar@/^1pc/[rr]
&
1
&
x
\ar[r]
& x^2
\ar@/^1pc/[rr]
& x^3
\ar[r]
\ar@/_1pc/[rr]
& x^4
& x^5
\ar[r]
& x^6
\ar@/_1pc/[rr]
& x^7
\ar[r]
\ar@/^1pc/[rr]
& x^8
& x^9
\ar[r]
&
x^{10}
&
\cdots
&&&
\\
\\
}
$$
} 
\\
\\
\scalebox{.62}{
$$
\xymatrix{
P_1: &
x
\ar[r]
& x^2
\ar@/^1pc/[rr]
& x^3
\ar[r]
\ar@/_1pc/[rr]
& x^4
& x^5
\ar[r]
& x^6
\ar@/_1pc/[rr]
& x^7
\ar[r]
\ar@/^1pc/[rr]
& x^8
& x^9
\ar[r]
&
x^{10}
\ar@/^1pc/[rr]
&
x^{11}
\ar[r]
\ar@/_1pc/[rr]
&
x^{12}
&
\cdots
&&&
\\
\\
}
$$
} 
\\
\\
\scalebox{.62}{
$$
\xymatrix@R=9pt{
P_2: &
y_2
\ar[r]
\ar@/_1pc/[rr]
& y_3
\ar@/^1pc/[rr]
& x^4
\ar@/_1pc/[rr]
& y_5
\ar[r]
& y_6
&
&
&
&
&
&
&
\\
&
&
x^{3}
\ar[ru]
\ar@/_1pc/[rr]
&
& x^5
\ar[r]
& x^6
\ar@/_1pc/[rr]
& x^7
\ar[r]
\ar@/^1pc/[rr]
& x^8
& x^9
\ar[r]
& x^{10}
\ar@/^1pc/[rr]
& x^{11}
\ar[r]
\ar@/_1pc/[rr]
& x^{12}
&
\cdots
&&
\\
\\
}
$$
} 
\\
\\
\scalebox{.62}{
$$
\xymatrix@R=9pt{
P_3: & y_3 \ar[r] \ar@/_1pc/[rr] & y_4 \ar@/^1pc/[rr] & x^{5} \ar[rd] \ar@/_1pc/[rr] & y_6 \ar[r] & y_7 \\
& & & & x^{6} \ar@/_1pc/[rr] & x^{7} \ar[r] \ar@/^1pc/[rr] & x^8  & x^9 \ar[r] & x^{10} \ar@/^1pc/[rr] & x^{11} \ar[r] \ar@/_1pc/[rr] & x^{12} & x^{13} \ar[r] & x^{14}  & \cdots \\ \\ } $$
} 
\\
\\
where $x^n$ and $y_n$ are in degree $n$.
\end{pro}

\begin{lemma} \label{lemma_kerapn}
 There are identifications
\begin{itemize}
\item $ \kera(P_{0}) = \F[x^4]$,
\item $ \kera(P_{1}) = x^4 \F[x^4]$,
\item $ \kera(P_{2}) = y^2 \F \oplus x^8 \F[x^4]$,
\item and  $ \kera(P_{3}) = y^2 \F \oplus x^8 \F[x^4]$
\end{itemize}
with the notations of Proposition \ref{pro_identpn}.
\end{lemma}

\begin{proof}
The result follows from Proposition \ref{pro_identpn}.
\end{proof}

\subsection{The $\Z[a]$-module $\mathcal{H}^{\star}(V)$}

The computation of $H_{01}^{ \star}(H \mf^{ \star}_{\gr}(BV))$ goes as follows: 
\begin{itemize}
\item understand $H_{01}^{ \star}(R(P^{ \otimes n})))$ as a $\F$-vector space with Proposition \ref{pro_calculh01},
\item compute the $ \F[a]$-module structure by Proposition \ref{pro_amodsurh01},
\item assemble the results with Lemma \ref{lemma_kunneth_hbv}.
\end{itemize}

The first and most difficult step is Theorem \ref{pro_calculh01pn} which gives $H_{01}^{ \star}(RP_n)$.

\begin{nota} \label{nota_HPstar}
Denote $HP^{ \star}$ the $RO( \gr)$-graded $ \F[a, \sigma^{-4}]$-module $ \{ 1, x^4 \} \F \otimes_{ \F} \F[ a, \sigma^{-4}, v]/( a^3, av)$, with grading $|x^4| = 4$, $|a|= \alpha$, $| \sigma^{-4}| = -4+4 \alpha$ and $|v| = 1+ \alpha$ (see figure \ref{figure_hpstar}).
\end{nota}

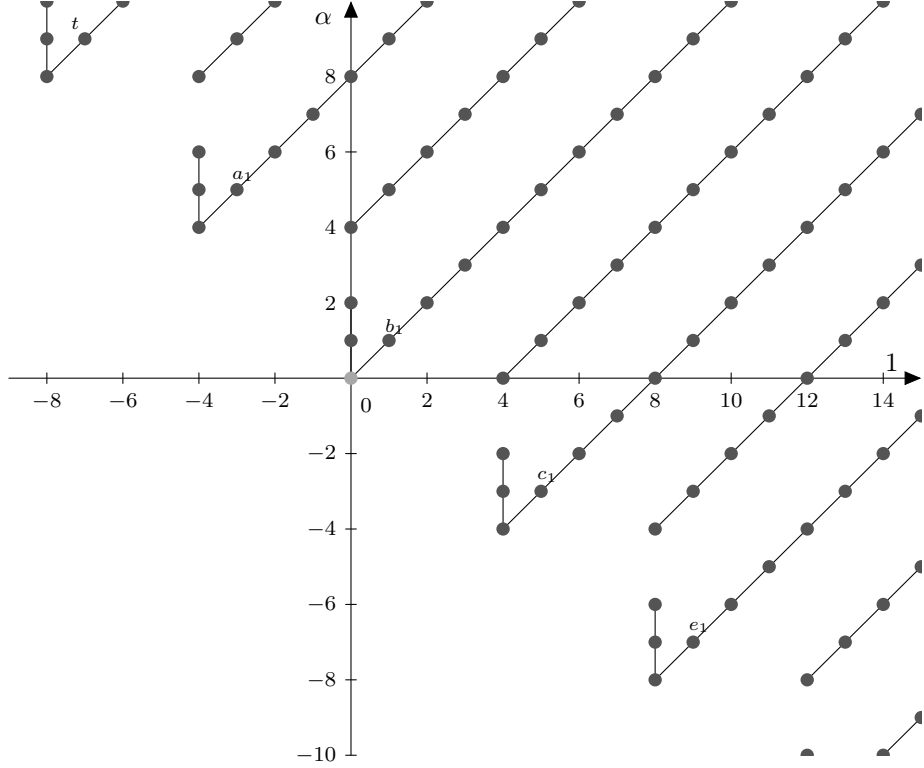
\begin{figure}

\definecolor{qqqqff}{rgb}{0.33,0.33,0.33}
\definecolor{xdxdff}{rgb}{0.66,0.66,0.66}
\definecolor{cqcqcq}{rgb}{0.75,0.75,0.75}

\begin{tikzpicture}[line cap=round,line join=round,>=triangle 45,x=0.5cm,y=0.5cm]
\draw[->,color=black] (-9,0) -- (15,0);
\foreach \x in {-8,-6,-4,-2,2,4,6,8,10,12,14}
\draw[shift={(\x,0)},color=black] (0pt,2pt) -- (0pt,-2pt) node[below] {\footnotesize $\x$};
\draw[->,color=black] (0,-10) -- (0,10);
\foreach \y in {-10,-8,-6,-4,-2,2,4,6,8}
\draw[shift={(0,\y)},color=black] (2pt,0pt) -- (-2pt,0pt) node[left] {\footnotesize $\y$};
\draw[color=black] (0pt,-10pt) node[right] {\footnotesize $0$};
\clip(-10,-10) rectangle (15,10);
\draw (-1.2,9.94) node[anchor=north west] {$ \alpha$};
\draw (13.80,0.90) node[anchor=north west] {1};
\draw [domain=0.0:15.0] plot(\x,{(-0--23*\x)/23});
\draw [domain=4.0:15.0] plot(\x,{(-92--23*\x)/23});
\draw [domain=4.0:15.0] plot(\x,{(-184--23*\x)/23});
\draw [domain=8.0:15.0] plot(\x,{(-96--6*\x)/6});
\draw [domain=12.0:15.0] plot(\x,{(-360--18*\x)/18});
\draw [domain=12.0:15.0] plot(\x,{(-432--18*\x)/18});
\draw [domain=16.0:15.0] plot(\x,{(-56--2*\x)/2});
\draw [domain=0.0:15.0] plot(\x,{(--8--2*\x)/2});
\draw [domain=-4.0:15.0] plot(\x,{(--16--2*\x)/2});
\draw [domain=-4.0:15.0] plot(\x,{(--24--2*\x)/2});
\draw [domain=-8.0:15.0] plot(\x,{(--32--2*\x)/2});
\draw [domain=-8.0:15.0] plot(\x,{(--40--2*\x)/2});
\draw [domain=-12.0:15.0] plot(\x,{(--48--2*\x)/2});
\draw [domain=-12.0:15.0] plot(\x,{(--56--2*\x)/2});
\draw [domain=-16.0:15.0] plot(\x,{(--64--2*\x)/2});
\draw [domain=-16.0:15.0] plot(\x,{(--72--2*\x)/2});
\draw (-16,16)-- (-16,18);
\draw (-12,12)-- (-12,14);
\draw (-8,8)-- (-8,10);
\draw (-4,4)-- (-4,6);
\draw (0,0)-- (0,2);
\draw (4,-4)-- (4,-2);
\draw (8,-4)-- (30,18);
\draw (8,-8)-- (8,-6);
\begin{scriptsize}
\fill [color=xdxdff] (0,0) circle (2.5pt);
\fill [color=qqqqff] (1,1) circle (2.5pt);
\fill [color=qqqqff] (2,2) circle (2.5pt);
\fill [color=qqqqff] (3,3) circle (2.5pt);
\fill [color=qqqqff] (4,4) circle (2.5pt);
\fill [color=qqqqff] (5,5) circle (2.5pt);
\fill [color=qqqqff] (6,6) circle (2.5pt);
\fill [color=qqqqff] (7,7) circle (2.5pt);
\fill [color=qqqqff] (8,8) circle (2.5pt);
\fill [color=qqqqff] (9,9) circle (2.5pt);
\fill [color=qqqqff] (10,10) circle (2.5pt);
\fill [color=qqqqff] (11,11) circle (2.5pt);
\fill [color=qqqqff] (12,12) circle (2.5pt);
\fill [color=qqqqff] (13,13) circle (2.5pt);
\fill [color=qqqqff] (14,14) circle (2.5pt);
\fill [color=qqqqff] (15,15) circle (2.5pt);
\fill [color=qqqqff] (16,16) circle (2.5pt);
\fill [color=qqqqff] (17,17) circle (2.5pt);
\fill [color=qqqqff] (18,18) circle (2.5pt);
\fill [color=qqqqff] (19,19) circle (2.5pt);
\fill [color=qqqqff] (20,20) circle (2.5pt);
\fill [color=qqqqff] (21,21) circle (2.5pt);
\fill [color=qqqqff] (22,22) circle (2.5pt);
\fill [color=qqqqff] (23,23) circle (2.5pt);
\fill [color=qqqqff] (4,0) circle (2.5pt);
\fill [color=qqqqff] (5,1) circle (2.5pt);
\fill [color=qqqqff] (6,2) circle (2.5pt);
\fill [color=qqqqff] (7,3) circle (2.5pt);
\fill [color=qqqqff] (8,4) circle (2.5pt);
\fill [color=qqqqff] (9,5) circle (2.5pt);
\fill [color=qqqqff] (10,6) circle (2.5pt);
\fill [color=qqqqff] (11,7) circle (2.5pt);
\fill [color=qqqqff] (12,8) circle (2.5pt);
\fill [color=qqqqff] (13,9) circle (2.5pt);
\fill [color=qqqqff] (14,10) circle (2.5pt);
\fill [color=qqqqff] (15,11) circle (2.5pt);
\fill [color=qqqqff] (16,12) circle (2.5pt);
\fill [color=qqqqff] (17,13) circle (2.5pt);
\fill [color=qqqqff] (18,14) circle (2.5pt);
\fill [color=qqqqff] (19,15) circle (2.5pt);
\fill [color=qqqqff] (20,16) circle (2.5pt);
\fill [color=qqqqff] (21,17) circle (2.5pt);
\fill [color=qqqqff] (22,18) circle (2.5pt);
\fill [color=qqqqff] (23,19) circle (2.5pt);
\fill [color=qqqqff] (24,20) circle (2.5pt);
\fill [color=qqqqff] (25,21) circle (2.5pt);
\fill [color=qqqqff] (26,22) circle (2.5pt);
\fill [color=qqqqff] (27,23) circle (2.5pt);
\fill [color=qqqqff] (4,-4) circle (2.5pt);
\fill [color=qqqqff] (5,-3) circle (2.5pt);
\fill [color=qqqqff] (6,-2) circle (2.5pt);
\fill [color=qqqqff] (7,-1) circle (2.5pt);
\fill [color=qqqqff] (8,0) circle (2.5pt);
\fill [color=qqqqff] (9,1) circle (2.5pt);
\fill [color=qqqqff] (10,2) circle (2.5pt);
\fill [color=qqqqff] (11,3) circle (2.5pt);
\fill [color=qqqqff] (12,4) circle (2.5pt);
\fill [color=qqqqff] (13,5) circle (2.5pt);
\fill [color=qqqqff] (14,6) circle (2.5pt);
\fill [color=qqqqff] (15,7) circle (2.5pt);
\fill [color=qqqqff] (16,8) circle (2.5pt);
\fill [color=qqqqff] (17,9) circle (2.5pt);
\fill [color=qqqqff] (18,10) circle (2.5pt);
\fill [color=qqqqff] (19,11) circle (2.5pt);
\fill [color=qqqqff] (20,12) circle (2.5pt);
\fill [color=qqqqff] (21,13) circle (2.5pt);
\fill [color=qqqqff] (22,14) circle (2.5pt);
\fill [color=qqqqff] (23,15) circle (2.5pt);
\fill [color=qqqqff] (24,16) circle (2.5pt);
\fill [color=qqqqff] (25,17) circle (2.5pt);
\fill [color=qqqqff] (26,18) circle (2.5pt);
\fill [color=qqqqff] (27,19) circle (2.5pt);
\fill [color=qqqqff] (8,-4) circle (2.5pt);
\fill [color=qqqqff] (9,-3) circle (2.5pt);
\fill [color=qqqqff] (10,-2) circle (2.5pt);
\fill [color=qqqqff] (11,-1) circle (2.5pt);
\fill [color=qqqqff] (12,0) circle (2.5pt);
\fill [color=qqqqff] (13,1) circle (2.5pt);
\fill [color=qqqqff] (14,2) circle (2.5pt);
\fill [color=qqqqff] (15,3) circle (2.5pt);
\fill [color=qqqqff] (16,4) circle (2.5pt);
\fill [color=qqqqff] (17,5) circle (2.5pt);
\fill [color=qqqqff] (18,6) circle (2.5pt);
\fill [color=qqqqff] (19,7) circle (2.5pt);
\fill [color=qqqqff] (20,8) circle (2.5pt);
\fill [color=qqqqff] (21,9) circle (2.5pt);
\fill [color=qqqqff] (22,10) circle (2.5pt);
\fill [color=qqqqff] (23,11) circle (2.5pt);
\fill [color=qqqqff] (24,12) circle (2.5pt);
\fill [color=qqqqff] (25,13) circle (2.5pt);
\fill [color=qqqqff] (26,14) circle (2.5pt);
\fill [color=qqqqff] (27,15) circle (2.5pt);
\fill [color=qqqqff] (28,16) circle (2.5pt);
\fill [color=qqqqff] (29,17) circle (2.5pt);
\fill [color=qqqqff] (30,18) circle (2.5pt);
\fill [color=qqqqff] (8,-8) circle (2.5pt);
\fill [color=qqqqff] (9,-7) circle (2.5pt);
\fill [color=qqqqff] (10,-6) circle (2.5pt);
\fill [color=qqqqff] (11,-5) circle (2.5pt);
\fill [color=qqqqff] (12,-4) circle (2.5pt);
\fill [color=qqqqff] (13,-3) circle (2.5pt);
\fill [color=qqqqff] (14,-2) circle (2.5pt);
\fill [color=qqqqff] (15,-1) circle (2.5pt);
\fill [color=qqqqff] (16,0) circle (2.5pt);
\fill [color=qqqqff] (17,1) circle (2.5pt);
\fill [color=qqqqff] (18,2) circle (2.5pt);
\fill [color=qqqqff] (19,3) circle (2.5pt);
\fill [color=qqqqff] (20,4) circle (2.5pt);
\fill [color=qqqqff] (21,5) circle (2.5pt);
\fill [color=qqqqff] (22,6) circle (2.5pt);
\fill [color=qqqqff] (23,7) circle (2.5pt);
\fill [color=qqqqff] (24,8) circle (2.5pt);
\fill [color=qqqqff] (25,9) circle (2.5pt);
\fill [color=qqqqff] (26,10) circle (2.5pt);
\fill [color=qqqqff] (27,11) circle (2.5pt);
\fill [color=qqqqff] (28,12) circle (2.5pt);
\fill [color=qqqqff] (29,13) circle (2.5pt);
\fill [color=qqqqff] (30,14) circle (2.5pt);
\fill [color=qqqqff] (12,-8) circle (2.5pt);
\fill [color=qqqqff] (13,-7) circle (2.5pt);
\fill [color=qqqqff] (14,-6) circle (2.5pt);
\fill [color=qqqqff] (15,-5) circle (2.5pt);
\fill [color=qqqqff] (16,-4) circle (2.5pt);
\fill [color=qqqqff] (17,-3) circle (2.5pt);
\fill [color=qqqqff] (18,-2) circle (2.5pt);
\fill [color=qqqqff] (19,-1) circle (2.5pt);
\fill [color=qqqqff] (20,0) circle (2.5pt);
\fill [color=qqqqff] (21,1) circle (2.5pt);
\fill [color=qqqqff] (22,2) circle (2.5pt);
\fill [color=qqqqff] (23,3) circle (2.5pt);
\fill [color=qqqqff] (24,4) circle (2.5pt);
\fill [color=qqqqff] (25,5) circle (2.5pt);
\fill [color=qqqqff] (26,6) circle (2.5pt);
\fill [color=qqqqff] (27,7) circle (2.5pt);
\fill [color=qqqqff] (28,8) circle (2.5pt);
\fill [color=qqqqff] (29,9) circle (2.5pt);
\fill [color=qqqqff] (30,10) circle (2.5pt);
\fill [color=qqqqff] (12,-12) circle (2.5pt);
\fill [color=qqqqff] (13,-11) circle (2.5pt);
\fill [color=qqqqff] (14,-10) circle (2.5pt);
\fill [color=qqqqff] (15,-9) circle (2.5pt);
\fill [color=qqqqff] (16,-8) circle (2.5pt);
\fill [color=qqqqff] (17,-7) circle (2.5pt);
\fill [color=qqqqff] (18,-6) circle (2.5pt);
\fill [color=qqqqff] (19,-5) circle (2.5pt);
\fill [color=qqqqff] (20,-4) circle (2.5pt);
\fill [color=qqqqff] (21,-3) circle (2.5pt);
\fill [color=qqqqff] (22,-2) circle (2.5pt);
\fill [color=qqqqff] (23,-1) circle (2.5pt);
\fill [color=qqqqff] (24,0) circle (2.5pt);
\fill [color=qqqqff] (25,1) circle (2.5pt);
\fill [color=qqqqff] (26,2) circle (2.5pt);
\fill [color=qqqqff] (27,3) circle (2.5pt);
\fill [color=qqqqff] (28,4) circle (2.5pt);
\fill [color=qqqqff] (29,5) circle (2.5pt);
\fill [color=qqqqff] (30,6) circle (2.5pt);
\fill [color=qqqqff] (16,-12) circle (2.5pt);
\fill [color=qqqqff] (17,-11) circle (2.5pt);
\fill [color=qqqqff] (18,-10) circle (2.5pt);
\fill [color=qqqqff] (19,-9) circle (2.5pt);
\fill [color=qqqqff] (20,-8) circle (2.5pt);
\fill [color=qqqqff] (21,-7) circle (2.5pt);
\fill [color=qqqqff] (22,-6) circle (2.5pt);
\fill [color=qqqqff] (23,-5) circle (2.5pt);
\fill [color=qqqqff] (24,-4) circle (2.5pt);
\fill [color=qqqqff] (25,-3) circle (2.5pt);
\fill [color=qqqqff] (26,-2) circle (2.5pt);
\fill [color=qqqqff] (27,-1) circle (2.5pt);
\fill [color=qqqqff] (28,0) circle (2.5pt);
\fill [color=qqqqff] (29,1) circle (2.5pt);
\fill [color=qqqqff] (30,2) circle (2.5pt);
\fill [color=qqqqff] (0,4) circle (2.5pt);
\fill [color=qqqqff] (1,5) circle (2.5pt);
\fill [color=qqqqff] (2,6) circle (2.5pt);
\fill [color=qqqqff] (3,7) circle (2.5pt);
\fill [color=qqqqff] (4,8) circle (2.5pt);
\fill [color=qqqqff] (5,9) circle (2.5pt);
\fill [color=qqqqff] (6,10) circle (2.5pt);
\fill [color=qqqqff] (7,11) circle (2.5pt);
\fill [color=qqqqff] (8,12) circle (2.5pt);
\fill [color=qqqqff] (9,13) circle (2.5pt);
\fill [color=qqqqff] (10,14) circle (2.5pt);
\fill [color=qqqqff] (11,15) circle (2.5pt);
\fill [color=qqqqff] (12,16) circle (2.5pt);
\fill [color=qqqqff] (13,17) circle (2.5pt);
\fill [color=qqqqff] (14,18) circle (2.5pt);
\fill [color=qqqqff] (15,19) circle (2.5pt);
\fill [color=qqqqff] (16,20) circle (2.5pt);
\fill [color=qqqqff] (17,21) circle (2.5pt);
\fill [color=qqqqff] (18,22) circle (2.5pt);
\fill [color=qqqqff] (19,23) circle (2.5pt);
\fill [color=qqqqff] (-4,4) circle (2.5pt);
\fill [color=qqqqff] (-3,5) circle (2.5pt);
\fill [color=qqqqff] (-2,6) circle (2.5pt);
\fill [color=qqqqff] (-1,7) circle (2.5pt);
\fill [color=qqqqff] (0,8) circle (2.5pt);
\fill [color=qqqqff] (1,9) circle (2.5pt);
\fill [color=qqqqff] (2,10) circle (2.5pt);
\fill [color=qqqqff] (3,11) circle (2.5pt);
\fill [color=qqqqff] (4,12) circle (2.5pt);
\fill [color=qqqqff] (5,13) circle (2.5pt);
\fill [color=qqqqff] (6,14) circle (2.5pt);
\fill [color=qqqqff] (7,15) circle (2.5pt);
\fill [color=qqqqff] (8,16) circle (2.5pt);
\fill [color=qqqqff] (9,17) circle (2.5pt);
\fill [color=qqqqff] (10,18) circle (2.5pt);
\fill [color=qqqqff] (11,19) circle (2.5pt);
\fill [color=qqqqff] (12,20) circle (2.5pt);
\fill [color=qqqqff] (13,21) circle (2.5pt);
\fill [color=qqqqff] (14,22) circle (2.5pt);
\fill [color=qqqqff] (15,23) circle (2.5pt);
\fill [color=qqqqff] (-4,8) circle (2.5pt);
\fill [color=qqqqff] (-3,9) circle (2.5pt);
\fill [color=qqqqff] (-2,10) circle (2.5pt);
\fill [color=qqqqff] (-1,11) circle (2.5pt);
\fill [color=qqqqff] (0,12) circle (2.5pt);
\fill [color=qqqqff] (1,13) circle (2.5pt);
\fill [color=qqqqff] (2,14) circle (2.5pt);
\fill [color=qqqqff] (3,15) circle (2.5pt);
\fill [color=qqqqff] (4,16) circle (2.5pt);
\fill [color=qqqqff] (5,17) circle (2.5pt);
\fill [color=qqqqff] (6,18) circle (2.5pt);
\fill [color=qqqqff] (7,19) circle (2.5pt);
\fill [color=qqqqff] (8,20) circle (2.5pt);
\fill [color=qqqqff] (9,21) circle (2.5pt);
\fill [color=qqqqff] (10,22) circle (2.5pt);
\fill [color=qqqqff] (11,23) circle (2.5pt);
\fill [color=qqqqff] (-8,8) circle (2.5pt);
\fill [color=qqqqff] (-7,9) circle (2.5pt);
\fill [color=qqqqff] (-6,10) circle (2.5pt);
\fill [color=qqqqff] (-5,11) circle (2.5pt);
\fill [color=qqqqff] (-4,12) circle (2.5pt);
\fill [color=qqqqff] (-3,13) circle (2.5pt);
\fill [color=qqqqff] (-2,14) circle (2.5pt);
\fill [color=qqqqff] (-1,15) circle (2.5pt);
\fill [color=qqqqff] (0,16) circle (2.5pt);
\fill [color=qqqqff] (1,17) circle (2.5pt);
\fill [color=qqqqff] (2,18) circle (2.5pt);
\fill [color=qqqqff] (3,19) circle (2.5pt);
\fill [color=qqqqff] (4,20) circle (2.5pt);
\fill [color=qqqqff] (5,21) circle (2.5pt);
\fill [color=qqqqff] (6,22) circle (2.5pt);
\fill [color=qqqqff] (7,23) circle (2.5pt);
\fill [color=qqqqff] (-8,12) circle (2.5pt);
\fill [color=qqqqff] (-7,13) circle (2.5pt);
\fill [color=qqqqff] (-6,14) circle (2.5pt);
\fill [color=qqqqff] (-5,15) circle (2.5pt);
\fill [color=qqqqff] (-4,16) circle (2.5pt);
\fill [color=qqqqff] (-3,17) circle (2.5pt);
\fill [color=qqqqff] (-2,18) circle (2.5pt);
\fill [color=qqqqff] (-1,19) circle (2.5pt);
\fill [color=qqqqff] (0,20) circle (2.5pt);
\fill [color=qqqqff] (1,21) circle (2.5pt);
\fill [color=qqqqff] (2,22) circle (2.5pt);
\fill [color=qqqqff] (3,23) circle (2.5pt);
\fill [color=qqqqff] (-12,12) circle (2.5pt);
\fill [color=qqqqff] (-11,13) circle (2.5pt);
\fill [color=qqqqff] (-10,14) circle (2.5pt);
\fill [color=qqqqff] (-9,15) circle (2.5pt);
\fill [color=qqqqff] (-8,16) circle (2.5pt);
\fill [color=qqqqff] (-7,17) circle (2.5pt);
\fill [color=qqqqff] (-6,18) circle (2.5pt);
\fill [color=qqqqff] (-5,19) circle (2.5pt);
\fill [color=qqqqff] (-4,20) circle (2.5pt);
\fill [color=qqqqff] (-3,21) circle (2.5pt);
\fill [color=qqqqff] (-2,22) circle (2.5pt);
\fill [color=qqqqff] (-1,23) circle (2.5pt);
\fill [color=qqqqff] (-12,16) circle (2.5pt);
\fill [color=qqqqff] (-11,17) circle (2.5pt);
\fill [color=qqqqff] (-10,18) circle (2.5pt);
\fill [color=qqqqff] (-9,19) circle (2.5pt);
\fill [color=qqqqff] (-8,20) circle (2.5pt);
\fill [color=qqqqff] (-7,21) circle (2.5pt);
\fill [color=qqqqff] (-6,22) circle (2.5pt);
\fill [color=qqqqff] (-5,23) circle (2.5pt);
\fill [color=qqqqff] (-16,16) circle (2.5pt);
\fill [color=qqqqff] (-15,17) circle (2.5pt);
\fill [color=qqqqff] (-14,18) circle (2.5pt);
\fill [color=qqqqff] (-13,19) circle (2.5pt);
\fill [color=qqqqff] (-12,20) circle (2.5pt);
\fill [color=qqqqff] (-11,21) circle (2.5pt);
\fill [color=qqqqff] (-10,22) circle (2.5pt);
\fill [color=qqqqff] (-9,23) circle (2.5pt);
\fill [color=qqqqff] (-16,20) circle (2.5pt);
\fill [color=qqqqff] (-15,21) circle (2.5pt);
\fill [color=qqqqff] (-14,22) circle (2.5pt);
\fill [color=qqqqff] (-13,23) circle (2.5pt);
\fill [color=qqqqff] (12,-11) circle (2.5pt);
\fill [color=qqqqff] (12,-10) circle (2.5pt);
\fill [color=qqqqff] (8,-7) circle (2.5pt);
\fill [color=qqqqff] (8,-6) circle (2.5pt);
\fill [color=qqqqff] (4,-3) circle (2.5pt);
\fill [color=qqqqff] (4,-2) circle (2.5pt);
\fill [color=qqqqff] (0,1) circle (2.5pt);
\fill [color=qqqqff] (0,2) circle (2.5pt);
\fill [color=qqqqff] (-4,5) circle (2.5pt);
\fill [color=qqqqff] (-4,6) circle (2.5pt);
\fill [color=qqqqff] (-8,9) circle (2.5pt);
\fill [color=qqqqff] (-8,10) circle (2.5pt);
\fill [color=qqqqff] (-12,13) circle (2.5pt);
\fill [color=qqqqff] (-12,14) circle (2.5pt);
\fill [color=qqqqff] (-16,17) circle (2.5pt);
\fill [color=qqqqff] (-16,18) circle (2.5pt);
\draw[color=black] (-15.15,17.33) node {$r$};
\draw[color=black] (-11.11,13.38) node {$s$};
\draw[color=black] (-7.25,9.4) node {$t$};
\draw[color=black] (-2.84,5.36) node {$a_1$};
\draw[color=black] (1.16,1.37) node {$b_1$};
\draw[color=black] (5.16,-2.61) node {$c_1$};
\draw[color=black] (19.92,6.86) node {$d_1$};
\draw[color=black] (9.16,-6.6) node {$e_1$};
\end{scriptsize}
\end{tikzpicture}

\caption{The $RO( \gr)$-graded $ \F[a, \sigma^{-4}]$-module $HP^{ \star}$. Vertical lines represents the product by the Euler class $a$.} \label{figure_hpstar}
\end{figure}

\begin{thm} \label{pro_calculh01pn}
 There is a $RO( \gr)$-graded $ \F[a, \sigma^{-4}]$-module isomorphism
$$ H_{01}^{ \star}(RP_n) =( \Sigma^{-n(1+ \alpha)}HP^{ \star})_{twist \geq 0} \oplus ( \Sigma^{-n(1+ \alpha)-1}HP^{ \star})_{twist \leq -2}$$
where the functors $ (-)_{ twist \geq i}$ and $ (-)_{ twist \leq i}$ are truncation in degrees of the form $k+l\alpha$ for $l \geq i$ and $l \leq i$ respectively,  for  $i \in \mathbb{Z}$.
\end{thm}

Before passing to the proof, we need some intermediate results.

\begin{lemma} \label{lemma_fdevh01pn}
There is a $ \F$-vector space isomorphism
$$ H_{01}^{ \star}(RP_n) = \bigoplus_{i \geq 0} \Sigma^{i(1+ \alpha)} \kera(P_{n-i}) \oplus \bigoplus_{i \leq -2}  \Sigma^{i(1+ \alpha)-1} \kera(P_{n-i}).$$
\end{lemma}

\begin{proof}
By Proposition \ref{pro_calculh01} and Definition \ref{de_pn} we have isomorphisms, for all $i \geq 0$, $$H_{01}^{ * + i \alpha}(RP_n) \cong \Sigma^{2i+ i \alpha} \kera( \Omega_r^{-i}(P_n)) \cong \Sigma^{2i+ i \alpha} \kera( \Sigma^{-i}(P_{n-i})),$$
and for all $ i \geq 2$, $$H_{01}^{ * - i \alpha}(RP_n) \cong \Sigma^{-2i-1- i \alpha} \kera( \Omega_r^{-i}(P_n)) \cong \Sigma^{-2i-1- i \alpha} \kera( \Sigma^{-i}(P_{n-i})),$$ the result follows.
\end{proof}

We now conclude the proof of Proposition \ref{pro_calculh01pn} by determining the $ \F[a, \sigma^{-4}]$-module structure on $ H_{01}^{ \star}(RP_n)$.

\begin{proof}[Proof of Theorem \ref{pro_calculh01pn} ]
Lemmas \ref{lemma_fdevh01pn} and \ref{lemma_kerapn} provide a $ \F$-vector space isomorphism
$$ H_{01}^{ \star}(RP_n) =( \Sigma^{-n(1+ \alpha)}HP^{ \star})_{twist \geq 0} \oplus ( \Sigma^{-n(1+ \alpha)-1}HP^{ \star})_{twist \leq -2}.$$

We use Proposition \ref{pro_amodsurh01}.
For degree reason, the only possible elements in $Coker(a)$ among the elements of positive twist are, via the identification given in Proposition \ref{pro_calculh01}, $1 \in \kera(P_0)$ and $ y^2 \in \kera (P_3)$.

For the negative twisted part, the only elements possibly in $im(a)$ are, via the identification of Proposition \ref{pro_calculh01}, $\sigma^2 1 \in \sigma^2(P_0/(Im_{Sq^2}(  Ker_{Q_0}P_0) + Im_{Q_1}( Ker_{Q_0}P_0)))$ and $ x^2 \in \sigma^2(P_1/(Im_{Sq^2}(  Ker_{Q_0}P_1) + Im_{Q_1}( Ker_{Q_0}P_1)))$. \\

We already computed the vector space structure of $H_{01}^{ \star}(R P_0)$ and from Proposition \ref{pro_identpn}, Lemma \ref{lemma_h01msigma} and Proposition \ref{pro_amodsurh01}, we get $H_{01}^{ \star}( P_0[ \sigma^{-1}] = \F[x^2][ \sigma^{-1}]$ and $H_{01}^{ \star}( \sigma^2 P_0[ \sigma]) = \sigma^2  \F[x^2][ \sigma]$. \\

Now, the short exact sequences provided by Proposition \ref{pro_amodsurh01} give a $RO( \gr)$-graded vector space isomorphism $$ \Sigma^{-2} Ker_a( H_{01}^{ \star}( \rp P_0)) \oplus Coker_a( H_{01}^{ \star}( \rp P_0))  \cong  \F[x^2][ \sigma^{-1}].$$ Consequently $1 \in \kera(P_0)$ and $ y^2 \in \kera (P_3)$ belong to $Coker(a)$ since they are the only elements of $H_{01}^{ \star}( RP_0)$ in the appropriate grading.

For the negatively twisted part, this is analogous. There is a $RO( \gr)$-graded $\F$-vector space isomorphism $$ Ker_a( H_{01}^{ \star}( \rn P_0)) \oplus  \Sigma^{-2} Coker_a( H_{01}^{ \star}( \rn P_0))  \cong  \sigma^2  \F[x^2][ \sigma],$$ which forces the elements $$\sigma^2 1 \in \sigma^2(Ker_{Q_0}(P_0)/(Im_{Sq^2}(  Ker_{Q_0}P_0) + Im_{Q_1}( Ker_{Q_0}P_0)))$$ and $$ x^2 \in\sigma^2 Ker_{Q_0}(P_1)/(Im_{Sq^2}(  Ker_{Q_0}P_1) + Im_{Q_1}( Ker_{Q_0}(P_1)))$$ to be in $Im(a)$. \\

To finish, we determine the $ \sigma^{-4}$ action on $H_{01}^{ \star}(RP_0)$. Consider the long exact sequence obtained by applying $H_{01}^{ \star}$ to the short $( \E, \lo)$-exact sequence
$$ \sigma^{-4} \rp P_0 \inj \rp P_0 \surj  \rp P_0 / (  \sigma^{-4} \rp P_0).$$

We will show that $ H_{01}^{ \star}( \rp P_0)$  is a free $ \F[ \sigma^{-4}]$-module. In each degree, the rank of $ H_{01}^{ \star}( \rp P_0)$ as a  $ \F[ \sigma^{-4}]$-module is at most one, so it is sufficient to determine the $ \F[ \sigma^{-4}]$-module structure of $ H_{01}^{ \star}( \rp P_0)$. \\

To this end, we show that the edge of the previously considered long exact sequence is trivial. It is sufficient to see that, for all $i \leq 3, \ j \geq 0$ and $m \in P_0$, $ \qi( a^j \sigma^{-i} m) \not\in ( \sigma^{-4} \rp P_0) - \{ 0 \}$.
The Cartan formulae give
\begin{eqnarray*}
  & & \qi( a^j \sigma^{-i} m)  \\ & = & a^j \qi(  \sigma^{-i}) m + a^{j+1} \qo( \sigma^{-i}) \qo(x) + a^j \sigma^{-i} \qi( m ) \\
& = & a^j \qi(  \sigma^{-i}) m + a^{j+1} \qo( \sigma^{-i}) Q_0(x) + a^{j+1} \sigma^{-i} Sq^2(m) + a^{j} \sigma^{-i-1} Q_1(m).
\end{eqnarray*}
For $ \qi( a^j \sigma^{-i} m) $ to be divisible by $ \sigma^{-4}$, the two following points must be satisfied
\begin{itemize}
\item $  \qi(  \sigma^{-i}) = 0$, so that $i = 0$ or $i = 1$,
\item $a^{j+1} \qo( \sigma^{-i}) Q_0(x) + a^{j+1} \sigma^{-i} Sq^2(m) + a^{j} \sigma^{-i-1} Q_1(m)$ is multiple of $ \sigma^{-4}$.
\end{itemize}
These are only simultaneously satisfied when $ \qi( a^j \sigma^{-i} m) = 0$. The result follows for the positively twisted part. \\

For $ H_{01}^{ \star}( \rn P_0)$, consider the long exact sequence obtained by applying $H_{01}^{ \star}$ to the short $( \E, \lo)$-exact sequence
$$ K \inj \rn P_0 \surj \Sigma^{| \sigma^{-4}|} \rn P_0 $$
where $ K = Ker( P_0 \surj \Sigma^{| \sigma^{-4}|} \rn P_0$.
We again show that its edge is trivial. Let $ \Sigma^{| \sigma^{-4}|}  x$ representing a class in $ H_{01}^{ \star}( \Sigma^{| \sigma^{-4}|} \rn P_0 )$. The element $ \sigma^{-4} x \in \rn P_0$ is a lift of $ \sigma^{-4} x$. But $ \qi( \sigma^{-4} x) = \sigma^{-4} \qi(x) = 0$, therefore, the product by $ \sigma^{-4}$ is surjective on $ H_{01}^{ \star}( \rn P_0)$. For dimensional reasons, it suffices to determine the $ \F[  \sigma^{-4}]$-module structure on $ H_{01}^{ \star}( \rn P_0)$. \\

To finish with $P_0$, observe that the $ \F[ \sigma^{-4}]$-module structure defined on $HP^{ \star}$ induces a $ \F[ \sigma^{-4}]$-module structure on $$( \Sigma^{-n(1+ \alpha)}HP^{ \star})_{twist \geq 0} \oplus ( \Sigma^{-n(1+ \alpha)-1}HP^{ \star})_{twist \leq -2}$$ which satisfies the properties
\begin{itemize}
 \item the product by $ \sigma^{-4}$ on $( \Sigma^{-n(1+ \alpha)}HP^{ \star})_{twist \geq 0}$ is injective,
\item the product by $ \sigma^{-4}$ on $( \Sigma^{-n(1+ \alpha)-1}HP^{ \star})_{twist \leq -2}$ is surjective.
\end{itemize}

We showed the result for $P_0$. 
To conclude, the same result is true for each $P_i$ since, in degrees $*+k \alpha$ for $k$ big enough (positively and negatively), the isomorphisms provided by Proposition \ref{pro_calculh01} assemble together in a $ \F[a, \sigma^{-4}]$-module isomorphism by \ref{lemma_H01R} since these isomorphisms are obtained by applying  $H_{01}^{ \star} \circ R$ to a $ \ste(1)$-module morphism.
\end{proof}

\begin{ex}
The $ \F[a, \sigma^{-4}]$-module $H_{01}^{ \star}(H \mf^{ \star}_{\gr}(B \gr))$ is represented in figure \ref{fig_h01bgr}. 
\end{ex}

\begin{figure}

\definecolor{xdxdff}{rgb}{0.66,0.66,0.66}
\definecolor{qqqqff}{rgb}{0.33,0.33,0.33}
\definecolor{cqcqcq}{rgb}{0.75,0.75,0.75}

\begin{tikzpicture}[line cap=round,line join=round,>=triangle 45,x=0.5cm,y=0.5cm]
\draw[->,color=black] (-9,0) -- (15,0);
\foreach \x in {-8,-6,-4,-2,2,4,6,8,10,12,14}
\draw[shift={(\x,0)},color=black] (0pt,2pt) -- (0pt,-2pt) node[below] {\footnotesize $\x$};
\draw[->,color=black] (0,-10) -- (0,10);
\foreach \y in {-10,-8,-6,-4,-2,2,4,6,8}
\draw[shift={(0,\y)},color=black] (2pt,0pt) -- (-2pt,0pt) node[left] {\footnotesize $\y$};
\draw[color=black] (0pt,-10pt) node[right] {\footnotesize $0$};
\clip(-10,-10) rectangle (15,10);
\draw (-1.2,9.94) node[anchor=north west] {$ \alpha$};
\draw (13.80,0.90) node[anchor=north west] {1};
\draw [domain=1.0:15.0] plot(\x,{(-0--23*\x)/23});
\draw [domain=5.0:15.0] plot(\x,{(-92--23*\x)/23});
\draw [domain=1.0:15.0] plot(\x,{(--8--2*\x)/2});
\draw [domain=-3.0:15.0] plot(\x,{(--16--2*\x)/2});
\draw [domain=-3.0:15.0] plot(\x,{(--24--2*\x)/2});
\draw [domain=-7.0:15.0] plot(\x,{(--32--2*\x)/2});
\draw [domain=-7.0:15.0] plot(\x,{(--40--2*\x)/2});
\draw [domain=-11.0:15.0] plot(\x,{(--48--2*\x)/2});
\draw [domain=-11.0:15.0] plot(\x,{(--56--2*\x)/2});
\draw [domain=-15.0:15.0] plot(\x,{(--64--2*\x)/2});
\draw [domain=-15.0:15.0] plot(\x,{(--72--2*\x)/2});
\draw (-15,17)-- (-15,19);
\draw (-11,13)-- (-11,15);
\draw (-7,9)-- (-7,11);
\draw (-3,5)-- (-3,7);
\draw (1,1)-- (1,3);
\draw (9,-7)-- (9,-5);
\draw (5,-3)-- (5,-2);
\draw (5,-3)-- (6,-2);
\draw (9,-3)-- (10,-2);
\draw (9,-7)-- (14,-2);
\draw (13,-7)-- (18,-2);
\draw (13,-11)-- (13,-9);
\draw (13,-11)-- (22,-2);
\draw (17,-11)-- (26,-2);
\draw [domain=8.0:15.0] plot(\x,{(-160--20*\x)/20});
\draw [domain=12.0:15.0] plot(\x,{(-144--12*\x)/12});
\draw [domain=16.0:15.0] plot(\x,{(-192--12*\x)/12});
\draw [domain=20.0:15.0] plot(\x,{(-100--5*\x)/5});
\draw [domain=24.0:15.0] plot(\x,{(-48--2*\x)/2});
\begin{scriptsize}
\fill [color=qqqqff] (2,2) circle (2.5pt);
\fill [color=qqqqff] (3,3) circle (2.5pt);
\fill [color=qqqqff] (4,4) circle (2.5pt);
\fill [color=qqqqff] (5,5) circle (2.5pt);
\fill [color=qqqqff] (6,6) circle (2.5pt);
\fill [color=qqqqff] (7,7) circle (2.5pt);
\fill [color=qqqqff] (8,8) circle (2.5pt);
\fill [color=qqqqff] (9,9) circle (2.5pt);
\fill [color=qqqqff] (10,10) circle (2.5pt);
\fill [color=qqqqff] (11,11) circle (2.5pt);
\fill [color=qqqqff] (12,12) circle (2.5pt);
\fill [color=qqqqff] (13,13) circle (2.5pt);
\fill [color=qqqqff] (14,14) circle (2.5pt);
\fill [color=qqqqff] (15,15) circle (2.5pt);
\fill [color=qqqqff] (16,16) circle (2.5pt);
\fill [color=qqqqff] (17,17) circle (2.5pt);
\fill [color=qqqqff] (18,18) circle (2.5pt);
\fill [color=qqqqff] (19,19) circle (2.5pt);
\fill [color=qqqqff] (20,20) circle (2.5pt);
\fill [color=qqqqff] (21,21) circle (2.5pt);
\fill [color=qqqqff] (22,22) circle (2.5pt);
\fill [color=qqqqff] (23,23) circle (2.5pt);
\fill [color=qqqqff] (24,24) circle (2.5pt);
\fill [color=qqqqff] (5,1) circle (2.5pt);
\fill [color=qqqqff] (6,2) circle (2.5pt);
\fill [color=qqqqff] (7,3) circle (2.5pt);
\fill [color=qqqqff] (8,4) circle (2.5pt);
\fill [color=qqqqff] (9,5) circle (2.5pt);
\fill [color=qqqqff] (10,6) circle (2.5pt);
\fill [color=qqqqff] (11,7) circle (2.5pt);
\fill [color=qqqqff] (12,8) circle (2.5pt);
\fill [color=qqqqff] (13,9) circle (2.5pt);
\fill [color=qqqqff] (14,10) circle (2.5pt);
\fill [color=qqqqff] (15,11) circle (2.5pt);
\fill [color=qqqqff] (16,12) circle (2.5pt);
\fill [color=qqqqff] (17,13) circle (2.5pt);
\fill [color=qqqqff] (18,14) circle (2.5pt);
\fill [color=qqqqff] (19,15) circle (2.5pt);
\fill [color=qqqqff] (20,16) circle (2.5pt);
\fill [color=qqqqff] (21,17) circle (2.5pt);
\fill [color=qqqqff] (22,18) circle (2.5pt);
\fill [color=qqqqff] (23,19) circle (2.5pt);
\fill [color=qqqqff] (24,20) circle (2.5pt);
\fill [color=qqqqff] (25,21) circle (2.5pt);
\fill [color=qqqqff] (26,22) circle (2.5pt);
\fill [color=qqqqff] (27,23) circle (2.5pt);
\fill [color=qqqqff] (28,24) circle (2.5pt);
\fill [color=qqqqff] (5,-3) circle (2.5pt);
\fill [color=qqqqff] (6,-2) circle (2.5pt);
\fill [color=qqqqff] (8,0) circle (2.5pt);
\fill [color=qqqqff] (9,1) circle (2.5pt);
\fill [color=qqqqff] (10,2) circle (2.5pt);
\fill [color=qqqqff] (11,3) circle (2.5pt);
\fill [color=qqqqff] (12,4) circle (2.5pt);
\fill [color=qqqqff] (13,5) circle (2.5pt);
\fill [color=qqqqff] (14,6) circle (2.5pt);
\fill [color=qqqqff] (15,7) circle (2.5pt);
\fill [color=qqqqff] (16,8) circle (2.5pt);
\fill [color=qqqqff] (17,9) circle (2.5pt);
\fill [color=qqqqff] (18,10) circle (2.5pt);
\fill [color=qqqqff] (19,11) circle (2.5pt);
\fill [color=qqqqff] (20,12) circle (2.5pt);
\fill [color=qqqqff] (21,13) circle (2.5pt);
\fill [color=qqqqff] (22,14) circle (2.5pt);
\fill [color=qqqqff] (23,15) circle (2.5pt);
\fill [color=qqqqff] (24,16) circle (2.5pt);
\fill [color=qqqqff] (25,17) circle (2.5pt);
\fill [color=qqqqff] (26,18) circle (2.5pt);
\fill [color=qqqqff] (27,19) circle (2.5pt);
\fill [color=qqqqff] (28,20) circle (2.5pt);
\fill [color=qqqqff] (9,-3) circle (2.5pt);
\fill [color=qqqqff] (10,-2) circle (2.5pt);
\fill [color=qqqqff] (12,0) circle (2.5pt);
\fill [color=qqqqff] (13,1) circle (2.5pt);
\fill [color=qqqqff] (14,2) circle (2.5pt);
\fill [color=qqqqff] (15,3) circle (2.5pt);
\fill [color=qqqqff] (16,4) circle (2.5pt);
\fill [color=qqqqff] (17,5) circle (2.5pt);
\fill [color=qqqqff] (18,6) circle (2.5pt);
\fill [color=qqqqff] (19,7) circle (2.5pt);
\fill [color=qqqqff] (20,8) circle (2.5pt);
\fill [color=qqqqff] (21,9) circle (2.5pt);
\fill [color=qqqqff] (22,10) circle (2.5pt);
\fill [color=qqqqff] (23,11) circle (2.5pt);
\fill [color=qqqqff] (24,12) circle (2.5pt);
\fill [color=qqqqff] (25,13) circle (2.5pt);
\fill [color=qqqqff] (26,14) circle (2.5pt);
\fill [color=qqqqff] (27,15) circle (2.5pt);
\fill [color=qqqqff] (28,16) circle (2.5pt);
\fill [color=qqqqff] (29,17) circle (2.5pt);
\fill [color=qqqqff] (30,18) circle (2.5pt);
\fill [color=qqqqff] (31,19) circle (2.5pt);
\fill [color=qqqqff] (9,-7) circle (2.5pt);
\fill [color=qqqqff] (10,-6) circle (2.5pt);
\fill [color=qqqqff] (11,-5) circle (2.5pt);
\fill [color=qqqqff] (12,-4) circle (2.5pt);
\fill [color=qqqqff] (13,-3) circle (2.5pt);
\fill [color=qqqqff] (14,-2) circle (2.5pt);
\fill [color=qqqqff] (16,0) circle (2.5pt);
\fill [color=qqqqff] (17,1) circle (2.5pt);
\fill [color=qqqqff] (18,2) circle (2.5pt);
\fill [color=qqqqff] (19,3) circle (2.5pt);
\fill [color=qqqqff] (20,4) circle (2.5pt);
\fill [color=qqqqff] (21,5) circle (2.5pt);
\fill [color=qqqqff] (22,6) circle (2.5pt);
\fill [color=qqqqff] (23,7) circle (2.5pt);
\fill [color=qqqqff] (24,8) circle (2.5pt);
\fill [color=qqqqff] (25,9) circle (2.5pt);
\fill [color=qqqqff] (26,10) circle (2.5pt);
\fill [color=qqqqff] (27,11) circle (2.5pt);
\fill [color=qqqqff] (28,12) circle (2.5pt);
\fill [color=qqqqff] (29,13) circle (2.5pt);
\fill [color=qqqqff] (30,14) circle (2.5pt);
\fill [color=qqqqff] (31,15) circle (2.5pt);
\fill [color=qqqqff] (13,-7) circle (2.5pt);
\fill [color=qqqqff] (14,-6) circle (2.5pt);
\fill [color=qqqqff] (15,-5) circle (2.5pt);
\fill [color=qqqqff] (16,-4) circle (2.5pt);
\fill [color=qqqqff] (17,-3) circle (2.5pt);
\fill [color=qqqqff] (18,-2) circle (2.5pt);
\fill [color=qqqqff] (20,0) circle (2.5pt);
\fill [color=qqqqff] (21,1) circle (2.5pt);
\fill [color=qqqqff] (22,2) circle (2.5pt);
\fill [color=qqqqff] (23,3) circle (2.5pt);
\fill [color=qqqqff] (24,4) circle (2.5pt);
\fill [color=qqqqff] (25,5) circle (2.5pt);
\fill [color=qqqqff] (26,6) circle (2.5pt);
\fill [color=qqqqff] (27,7) circle (2.5pt);
\fill [color=qqqqff] (28,8) circle (2.5pt);
\fill [color=qqqqff] (29,9) circle (2.5pt);
\fill [color=qqqqff] (30,10) circle (2.5pt);
\fill [color=qqqqff] (31,11) circle (2.5pt);
\fill [color=qqqqff] (13,-11) circle (2.5pt);
\fill [color=qqqqff] (14,-10) circle (2.5pt);
\fill [color=qqqqff] (15,-9) circle (2.5pt);
\fill [color=qqqqff] (16,-8) circle (2.5pt);
\fill [color=qqqqff] (17,-7) circle (2.5pt);
\fill [color=qqqqff] (18,-6) circle (2.5pt);
\fill [color=qqqqff] (19,-5) circle (2.5pt);
\fill [color=qqqqff] (20,-4) circle (2.5pt);
\fill [color=qqqqff] (21,-3) circle (2.5pt);
\fill [color=qqqqff] (22,-2) circle (2.5pt);
\fill [color=qqqqff] (24,0) circle (2.5pt);
\fill [color=qqqqff] (25,1) circle (2.5pt);
\fill [color=qqqqff] (26,2) circle (2.5pt);
\fill [color=qqqqff] (27,3) circle (2.5pt);
\fill [color=qqqqff] (28,4) circle (2.5pt);
\fill [color=qqqqff] (29,5) circle (2.5pt);
\fill [color=qqqqff] (30,6) circle (2.5pt);
\fill [color=qqqqff] (31,7) circle (2.5pt);
\fill [color=qqqqff] (17,-11) circle (2.5pt);
\fill [color=qqqqff] (18,-10) circle (2.5pt);
\fill [color=qqqqff] (19,-9) circle (2.5pt);
\fill [color=qqqqff] (20,-8) circle (2.5pt);
\fill [color=qqqqff] (21,-7) circle (2.5pt);
\fill [color=qqqqff] (22,-6) circle (2.5pt);
\fill [color=qqqqff] (23,-5) circle (2.5pt);
\fill [color=qqqqff] (24,-4) circle (2.5pt);
\fill [color=qqqqff] (25,-3) circle (2.5pt);
\fill [color=qqqqff] (26,-2) circle (2.5pt);
\fill [color=qqqqff] (28,0) circle (2.5pt);
\fill [color=qqqqff] (29,1) circle (2.5pt);
\fill [color=qqqqff] (30,2) circle (2.5pt);
\fill [color=qqqqff] (31,3) circle (2.5pt);
\fill [color=qqqqff] (1,5) circle (2.5pt);
\fill [color=qqqqff] (2,6) circle (2.5pt);
\fill [color=qqqqff] (3,7) circle (2.5pt);
\fill [color=qqqqff] (4,8) circle (2.5pt);
\fill [color=qqqqff] (5,9) circle (2.5pt);
\fill [color=qqqqff] (6,10) circle (2.5pt);
\fill [color=qqqqff] (7,11) circle (2.5pt);
\fill [color=qqqqff] (8,12) circle (2.5pt);
\fill [color=qqqqff] (9,13) circle (2.5pt);
\fill [color=qqqqff] (10,14) circle (2.5pt);
\fill [color=qqqqff] (11,15) circle (2.5pt);
\fill [color=qqqqff] (12,16) circle (2.5pt);
\fill [color=qqqqff] (13,17) circle (2.5pt);
\fill [color=qqqqff] (14,18) circle (2.5pt);
\fill [color=qqqqff] (15,19) circle (2.5pt);
\fill [color=qqqqff] (16,20) circle (2.5pt);
\fill [color=qqqqff] (17,21) circle (2.5pt);
\fill [color=qqqqff] (18,22) circle (2.5pt);
\fill [color=qqqqff] (19,23) circle (2.5pt);
\fill [color=qqqqff] (20,24) circle (2.5pt);
\fill [color=qqqqff] (-3,5) circle (2.5pt);
\fill [color=qqqqff] (-2,6) circle (2.5pt);
\fill [color=qqqqff] (-1,7) circle (2.5pt);
\fill [color=qqqqff] (0,8) circle (2.5pt);
\fill [color=qqqqff] (1,9) circle (2.5pt);
\fill [color=qqqqff] (2,10) circle (2.5pt);
\fill [color=qqqqff] (3,11) circle (2.5pt);
\fill [color=qqqqff] (4,12) circle (2.5pt);
\fill [color=qqqqff] (5,13) circle (2.5pt);
\fill [color=qqqqff] (6,14) circle (2.5pt);
\fill [color=qqqqff] (7,15) circle (2.5pt);
\fill [color=qqqqff] (8,16) circle (2.5pt);
\fill [color=qqqqff] (9,17) circle (2.5pt);
\fill [color=qqqqff] (10,18) circle (2.5pt);
\fill [color=qqqqff] (11,19) circle (2.5pt);
\fill [color=qqqqff] (12,20) circle (2.5pt);
\fill [color=qqqqff] (13,21) circle (2.5pt);
\fill [color=qqqqff] (14,22) circle (2.5pt);
\fill [color=qqqqff] (15,23) circle (2.5pt);
\fill [color=qqqqff] (16,24) circle (2.5pt);
\fill [color=qqqqff] (-3,9) circle (2.5pt);
\fill [color=qqqqff] (-2,10) circle (2.5pt);
\fill [color=qqqqff] (-1,11) circle (2.5pt);
\fill [color=qqqqff] (0,12) circle (2.5pt);
\fill [color=qqqqff] (1,13) circle (2.5pt);
\fill [color=qqqqff] (2,14) circle (2.5pt);
\fill [color=qqqqff] (3,15) circle (2.5pt);
\fill [color=qqqqff] (4,16) circle (2.5pt);
\fill [color=qqqqff] (5,17) circle (2.5pt);
\fill [color=qqqqff] (6,18) circle (2.5pt);
\fill [color=qqqqff] (7,19) circle (2.5pt);
\fill [color=qqqqff] (8,20) circle (2.5pt);
\fill [color=qqqqff] (9,21) circle (2.5pt);
\fill [color=qqqqff] (10,22) circle (2.5pt);
\fill [color=qqqqff] (11,23) circle (2.5pt);
\fill [color=qqqqff] (12,24) circle (2.5pt);
\fill [color=qqqqff] (-7,9) circle (2.5pt);
\fill [color=qqqqff] (-6,10) circle (2.5pt);
\fill [color=qqqqff] (-5,11) circle (2.5pt);
\fill [color=qqqqff] (-4,12) circle (2.5pt);
\fill [color=qqqqff] (-3,13) circle (2.5pt);
\fill [color=qqqqff] (-2,14) circle (2.5pt);
\fill [color=qqqqff] (-1,15) circle (2.5pt);
\fill [color=qqqqff] (0,16) circle (2.5pt);
\fill [color=qqqqff] (1,17) circle (2.5pt);
\fill [color=qqqqff] (2,18) circle (2.5pt);
\fill [color=qqqqff] (3,19) circle (2.5pt);
\fill [color=qqqqff] (4,20) circle (2.5pt);
\fill [color=qqqqff] (5,21) circle (2.5pt);
\fill [color=qqqqff] (6,22) circle (2.5pt);
\fill [color=qqqqff] (7,23) circle (2.5pt);
\fill [color=qqqqff] (8,24) circle (2.5pt);
\fill [color=qqqqff] (-7,13) circle (2.5pt);
\fill [color=qqqqff] (-6,14) circle (2.5pt);
\fill [color=qqqqff] (-5,15) circle (2.5pt);
\fill [color=qqqqff] (-4,16) circle (2.5pt);
\fill [color=qqqqff] (-3,17) circle (2.5pt);
\fill [color=qqqqff] (-2,18) circle (2.5pt);
\fill [color=qqqqff] (-1,19) circle (2.5pt);
\fill [color=qqqqff] (0,20) circle (2.5pt);
\fill [color=qqqqff] (1,21) circle (2.5pt);
\fill [color=qqqqff] (2,22) circle (2.5pt);
\fill [color=qqqqff] (3,23) circle (2.5pt);
\fill [color=qqqqff] (4,24) circle (2.5pt);
\fill [color=qqqqff] (-11,13) circle (2.5pt);
\fill [color=qqqqff] (-10,14) circle (2.5pt);
\fill [color=qqqqff] (-9,15) circle (2.5pt);
\fill [color=qqqqff] (-8,16) circle (2.5pt);
\fill [color=qqqqff] (-7,17) circle (2.5pt);
\fill [color=qqqqff] (-6,18) circle (2.5pt);
\fill [color=qqqqff] (-5,19) circle (2.5pt);
\fill [color=qqqqff] (-4,20) circle (2.5pt);
\fill [color=qqqqff] (-3,21) circle (2.5pt);
\fill [color=qqqqff] (-2,22) circle (2.5pt);
\fill [color=qqqqff] (-1,23) circle (2.5pt);
\fill [color=qqqqff] (0,24) circle (2.5pt);
\fill [color=qqqqff] (-11,17) circle (2.5pt);
\fill [color=qqqqff] (-10,18) circle (2.5pt);
\fill [color=qqqqff] (-9,19) circle (2.5pt);
\fill [color=qqqqff] (-8,20) circle (2.5pt);
\fill [color=qqqqff] (-7,21) circle (2.5pt);
\fill [color=qqqqff] (-6,22) circle (2.5pt);
\fill [color=qqqqff] (-5,23) circle (2.5pt);
\fill [color=qqqqff] (-4,24) circle (2.5pt);
\fill [color=qqqqff] (-15,17) circle (2.5pt);
\fill [color=qqqqff] (-14,18) circle (2.5pt);
\fill [color=qqqqff] (-13,19) circle (2.5pt);
\fill [color=qqqqff] (-12,20) circle (2.5pt);
\fill [color=qqqqff] (-11,21) circle (2.5pt);
\fill [color=qqqqff] (-10,22) circle (2.5pt);
\fill [color=qqqqff] (-9,23) circle (2.5pt);
\fill [color=qqqqff] (-8,24) circle (2.5pt);
\fill [color=qqqqff] (-15,21) circle (2.5pt);
\fill [color=qqqqff] (-14,22) circle (2.5pt);
\fill [color=qqqqff] (-13,23) circle (2.5pt);
\fill [color=qqqqff] (-12,24) circle (2.5pt);
\fill [color=qqqqff] (13,-10) circle (2.5pt);
\fill [color=qqqqff] (13,-9) circle (2.5pt);
\fill [color=qqqqff] (9,-6) circle (2.5pt);
\fill [color=qqqqff] (9,-5) circle (2.5pt);
\fill [color=qqqqff] (5,-2) circle (2.5pt);
\fill [color=qqqqff] (1,2) circle (2.5pt);
\fill [color=qqqqff] (1,3) circle (2.5pt);
\fill [color=qqqqff] (-3,6) circle (2.5pt);
\fill [color=qqqqff] (-3,7) circle (2.5pt);
\fill [color=qqqqff] (-7,10) circle (2.5pt);
\fill [color=qqqqff] (-7,11) circle (2.5pt);
\fill [color=qqqqff] (-11,14) circle (2.5pt);
\fill [color=qqqqff] (-11,15) circle (2.5pt);
\fill [color=qqqqff] (-15,18) circle (2.5pt);
\fill [color=qqqqff] (-15,19) circle (2.5pt);
\fill [color=qqqqff] (1,1) circle (2.5pt);
\draw[color=black] (-14.13,18.41) node {$r$};
\draw[color=black] (-10.08,14.42) node {$s$};
\draw[color=black] (-6.17,10.42) node {$t$};
\draw[color=black] (-1.73,6.43) node {$a_1$};
\draw[color=black] (2.23,2.38) node {$b_1$};
\draw[color=black] (10.24,-5.6) node {$e_1$};
\draw[color=black] (5.9,-2.09) node {$c$};
\draw[color=black] (6.18,-2.66) node {$d$};
\draw[color=black] (10.14,-2.66) node {$e$};
\draw[color=black] (12.05,-4.65) node {$f$};
\draw[color=black] (16.15,-4.65) node {$g$};
\draw[color=black] (14.19,-9.6) node {$c_1$};
\draw[color=black] (18.48,-6.65) node {$d_1$};
\draw[color=black] (22.39,-6.65) node {$f_1$};
\end{scriptsize}
\end{tikzpicture}

\caption{The $ \F[a, \sigma^{-4}]$-module $H_{01}^{ \star}(H \mf^{ \star}_{\gr}(B \gr))$.} \label{fig_h01bgr}

\end{figure}

By additivity of the functors in play, one gets the following result.

\begin{corr} \label{cor:hv}
Let $V$ be an elementary abelian $2$-group and $F$ the largest free sub-$\ste(1)$-module of $H\F^*(BV)$.
 There is a $ \F[ a, \sigma^{-4}]$-module isomorphism 
$$ \mathcal{H}^{\star}(V) \cong \bigoplus_{i=1}^n \left(( \Sigma^{-i(1+ \alpha)}HP^{ \star})_{twist \geq 0} \oplus ( \Sigma^{-i(1+ \alpha)-1}HP^{ \star})_{twist \leq -2} \right)^{ \oplus \begin{pmatrix} n \\ i \end{pmatrix}} \oplus H_{01}^{ \star}(RF) $$
\end{corr}

\section{Height $2$ detection for elementary abelian $2$-groups un $k\R$-cohomology}

The goal of this section is to prove that the slice tower for $K\R$ theory satisfies the detection of height $2$ with respect to the functor $[BV,-]^{\star}_e$.
The strategy is the following: first, we use our computation of $\mathcal{H}^{\star}(V)$ to prove the detection of height $1$ for the Borel tower associated to the slice tower for $K\R$, that is $E \gr_+ \wedge \Sigma^{ \bullet(1+ \alpha)} k\R$. Then, the fact that the geometric fixed points $\Phi^{\gr}K\R = 0$ implies that the Bott element is $a$-torsion, and so the tower $\widetilde{E \gr} \wedge \Sigma^{ \bullet(1+ \alpha)} k\R$ has trivial structure morphisms $\widetilde{E\gr} \wedge v_1$, and in particular the diagram

$$  \xymatrix@u@R=0.0cm@C=1.0cm{ \vdots \ar[r]^{E \gr_+ \wedge v_1} & E \gr_+ \wedge \Sigma^{ (n+1)(1+ \alpha)} k\R \ar[r]^{E \gr_+ \wedge v_1} \ar[d] & E \gr_+ \wedge \Sigma^{ n(1+ \alpha)} k\R \ar[r]^{E \gr_+ \wedge v_1} \ar[d] & E \gr_+ \wedge \Sigma^{ (n-1)(1+ \alpha)} k\R \ar[r]^{E \gr_+ \wedge v_1} \ar[d] & \vdots \\
\vdots \ar[r]^{v_1} &  \Sigma^{ (n+1)(1+ \alpha)} k\R \ar[r]^{v_1} \ar[d] & \Sigma^{ n(1+ \alpha)} k\R \ar[r]^{v_1} \ar[d] &  \Sigma^{ (n-1)(1+ \alpha)} k\R\ar[r]^{v_1} \ar[d] & \vdots \\
\vdots \ar[r]^{0} & \widetilde{E\gr} \wedge \Sigma^{ (n+1)(1+ \alpha)} k\R \ar[r]^{ 0} & \widetilde{E\gr} \wedge \Sigma^{ n(1+ \alpha)} k\R \ar[r]^{ 0} &  \widetilde{E\gr} \wedge \Sigma^{ (n-1)(1+ \alpha)} k\R \ar[r]^{0} & \vdots  
}$$

satisfies the hypothesis of Proposition \ref{pro:suchgreatheights}, so the tower in the middle satisfies the detection of height $(1+1)$.
The last step is to use this detection property as a computational tool via Proposition \ref{thm:chaincplx} to achieve explicit computation.

\subsection{The proof of detection of height $1$ for $E \gr_+ \wedge \Sigma^{ \bullet(1+ \alpha)} k\R$}

We show the detection of height $1$ for the Borel slice tower for $K\R$ using $(3)$ of Proposition \ref{pro:detect_1_2} together with the chain complex given by Proposition \ref{thm:chaincplx}. To this end, we first compute the object $\frac{Ker(\theta_n^*)}{Im(\theta_{n-1}^*)}$ for the tower we are considering.

\begin{lemma} \label{lemma_h01erbv}
There is an isomorphism $$ \frac{ Ker_{E \gr_+ \wedge \qi}( E \gr_+ \wedge H \Z^{ \star}_{\gr}(BV)) }{ Im_{E \gr_+ \wedge \qi}( E \gr_+ \wedge H \Z^{ \star}_{\gr}(BV)) } \cong H_{ 01}^{ \star}( (E \gr_+ \wedge H \F)^{ \star}(BV)),$$
and  $H_{ 01}^{ \star}( (E \gr_+ \wedge H \F)^{ \star}_{\gr}(BV)) = (H_{ 01}^{ \star}(  H \F^{ \star}_{\gr}(BV)))_{twist \geq 0}[ \sigma^{ \pm 4}].$
\end{lemma}

\begin{proof}
The first isomorphism is by definition of $H_{01}^{\star}$. \\

The $H \mf$-module morphism $ E \gr_+ \wedge H \mf \rightarrow H \mf$ induces a $ \ste^{ \star}$-module morphism
$$ (E \gr_+ \wedge H \mf)^{ \star}(BV) \rightarrow H \mf^{ \star}_{\gr}(BV)$$
which is part of a long exact sequence of $ \F[a]$-modules
$$ \hdots \rightarrow (E \gr_+ \wedge H \mf)^{ \star}_{\gr}(BV) \rightarrow H \mf^{ \star}_{\gr}(BV) \stackrel{(-)[a^{ \pm 1}]}{ \longrightarrow} ( \widetilde{E \gr} \wedge H \mf)^{ \star}_{\gr}(BV) $$ $$ \rightarrow  (E \gr_+ \wedge H \mf)^{ \star+1}_{\gr}(BV) \rightarrow \hdots .$$

Recall that Hu and Kriz computed $(E\gr_+ \wedge H \mf)^{\star}_{\gr} = \F[ \sigma^{ \pm 1}, a^{-1}]$ in \cite[p.370]{HK01}, denoted $H_{\star}^f$ in \textit{loc. cit.}.
The $ \F[a]$-module structure on $H \mf^{ \star}_{\gr}(BV)$ given by Proposition \ref{thm:formula_qoi} identifies two out of three terms in the sequence:

$$ \hdots \rightarrow (E \gr_+ \wedge H \mf)^{ \star}_{\gr}(BV) \rightarrow \left( R(H \F^*(BV)) \right)^{ \star} \stackrel{(-)[a^{ \pm 1}]}{ \longrightarrow} $$ $$ \left( \F[ \sigma, a^{ \pm 1}] \otimes_{ \F}  H \F^{*}(BV)  \right)^{ \star}  \rightarrow  (E \gr_+ \wedge H \mf)^{ \star+1}_{\gr}(BV) \rightarrow \hdots,$$

providing a $\F$-vector space isomorphism $ (E \gr_+ \wedge H \mf)^{ \star}_{\gr}(BV) \cong  \F[ \sigma^{ \pm 1}, a^{-1}] \otimes_{ \F} H \F^*(BV)$. \\

The $ \ste^{ \star}$-module morphism $ (E \gr_+ \wedge H \mf)^{ \star}_{\gr}(BV) \rightarrow H \mf^{ \star}_{\gr}(BV)$ gives the $ \Lambda_{ \F}( E \gr_+ \wedge \qi)$-module structure on $ (E \gr_+ \wedge H \mf)^{ \star}_{\gr}(BV)$ by the Cartan formulae since
\begin{itemize}
\item it is an isomorphism in degrees of the form $k+ n \alpha$ for all $n$ and $k \leq -2$,
\item the element $ \sigma^{-1} \in (E \gr_+ \wedge H \mf)^{ \star}_{\gr}(BV)$ is invertible, thus $ \sigma^{-4} \in (Ker_{ E \gr_+ \wedge \qi} \cap Im_{E \gr_+  \wedge \qi})  ((E \gr_+ \wedge H \mf)^{ \star}_{\gr}(BV))$ is invertible.
\end{itemize}
In particular,  $H_{ 01}^{ \star}( (E \gr_+ \wedge H \F)^{ \star}_{\gr}(BV))$ is $ \sigma^{4}$-periodic, and the morphism
$$H_{ 01}^{ \star}( (E \gr_+ \wedge H \F)^{ \star}_{\gr}(BV)) \rightarrow H_{ 01}^{ \star}(  H \F^{ \star}_{\gr}(BV))$$
induced by $ E \gr_+ \wedge H \mf \rightarrow H \mf$ is an isomorphism in degrees of the form $k+ \Z \alpha$, for $k \leq -4$ ( because $| \qi| = 2+ \alpha$).
The result follows.
\end{proof}

\begin{lemma} \label{lemma_h01erkbv}
Let $n \geq 1$ and $V$ an elementary abelian $2$-group. Then, there is a $\F[a, \sigma^{-4}]$-module isomorphism between 

$$ \frac{ Ker_{ E \gr_+ \wedge \qi}( (E \gr_+ \wedge H \F)^{ \star}_{\gr}(BV))}{ Im_{ \Sigma^{-2- \alpha}E \gr_+ \wedge \qi}( (E \gr_+ \wedge H \F)^{ \star}_{\gr}(BV))}$$
and
$$ \bigoplus_{i=1}^n \left(HP^{ \star +i(1+ \alpha)} \right)^{ \begin{pmatrix} n \\ i \end{pmatrix}}.$$
\end{lemma}

\begin{proof}
The result now follows by additivity of the functors in play and Lemma \ref{lemma_h01erbv}, using Lemma \ref{lemma_kunneth_hbv}:

\begin{eqnarray*}
 &  & H_{ 01}^{ \star}( (E \gr_+ \wedge H \F)^{ \star}_{\gr}(BV)) \\
 & \cong & (H_{ 01}^{ \star}(  H \F^{ \star}_{\gr}(BV)))_{twist \geq 0}[ \sigma^{ \pm 4}] \\
 &  \cong & \bigoplus_{i=1}^n ((H_{ 01}^{ \star}(P^{ \otimes i}))_{twist \geq 0}[ \sigma^{ \pm 4}])^{ \oplus \begin{pmatrix} n \\ i \end{pmatrix}} \\
 & \cong & \bigoplus_{i=1}^n ((HP^{ \star+i(1+ \alpha)})_{twist \geq 0}[ \sigma^{ \pm 4}])^{ \oplus \begin{pmatrix} n \\ i \end{pmatrix}} \\
& \cong & \bigoplus_{i=1}^n (HP^{ \star+i(1+ \alpha)})^{ \oplus \begin{pmatrix} n \\ i \end{pmatrix}}
\end{eqnarray*}

where the last identification comes from the $ \sigma^{-4}$-periodicity of $HP^{ \star}$.
\end{proof}

\begin{pro} \label{thm_EK_detect}
Let $V$ be an elementary abelian $2$-group.
The tower $$ \left(E \gr_+ \wedge \Sigma^{ \bullet(1+ \alpha)} \kr \right)$$ satisfies the detection of height $1$ for $[BV,-]^{\star}_e$.
\end{pro}

\begin{proof}
By the second point of Theorem \ref{thm:chaincplx}, for any $n \in \Z$, the composite $t:= \overline{c_n^*}\iota_n\overline{\delta_{n-2}^*}$ yields a $\F[a]$-module morphism of degree $3+2\alpha$ whose image is isomorphic to $F^2$. For various $n$, the different maps $t$ defined this way are suspension of one another by multiple of the regular representation as the tower of $\gr$-spectra under consideration is $v_1$-periodic. Thus, by the third point of Proposition \ref{pro:detect_1_2}, it is sufficient to show that any  $ \F[a]$-module morphism

$$t :  \mathcal{H}^{\star}(V) \rightarrow \mathcal{H}^{\star +3+2 \alpha}(V)$$

is trivial. The Lemma \ref{lemma_h01erkbv} gives an identification of the source and target $\F[a,\sigma^{-4}]$-modules of $t$. \\

Recall that $HP^{ \star} = \{ 1, x^4 \} \F \otimes_{ \F} \F[ a, \sigma^{-4}, v]/( a^3, av)$, with degrees $|x^4| = 4$, $|a|= \alpha$, $| \sigma^{-4}| = -4+4 \alpha$ and $|v| = 1+ \alpha$, so the only possibly non-trivial values for such a morphism are, $t(ax) = y$ where $x$ is an element which is not in $Ker_{ a^2}$. But $t(ax) = at(x) = a0 = 0$ for degree reasons.
Consequently, $t$ is trivial, and the result follows.
\end{proof}

Proposition \ref{pro:suchgreatheights} gives our principal result:

\begin{thm} \label{thm_2detectkr}
The slice tower for $K\R$ satisfies the detection of height $2$ for $[BV,-]^{\star}_e$.
\end{thm}

\begin{proof}
We show that this tower satisfies the hypothesis of Proposition \ref{pro:suchgreatheights}.
\begin{itemize}
\item The functor $ E \gr_+ \wedge (-)$, is exact,
\item the functor $ \widetilde{E \gr} \wedge (-)$, is exact, and the isotropy separation sequence gives natural distinguished triangles $EX \rightarrow X \rightarrow \tilde{E} X$ for all $ \gr$-spectrum $X$,
\item Let $x \in \kr^{ \star}_{\gr}(BV)$, then $v_1x$ is $a$-torsion because $a^3v_1 = 0$ in $ \kr^{ \star}_{\gr}$, so the image of $v_1x$ in $ \widetilde{E \gr} \wedge \kr^{ \star}_{\gr}(BV)$ is trivial.
\end{itemize}
The result now follows from Proposition \ref{pro:suchgreatheights} for $h=1$.
\end{proof}

\subsection{Consequences of the detection of height $2$}

Recall the complete computation of $\mathcal{H}^{\star}(V)$ presented in Corollary \ref{cor:hv}.

\begin{pro} \label{pro_hstar_ktheories}
Write $H \F^*(BV) = F \oplus (H \F^*(BV))^{red}$.
Define the  $ \F[a]$-module morphism $t : \mathcal{H}^{\star}(V) \rightarrow \mathcal{H}^{ \star+3+2 \alpha}(V)$
 by the commutative diagram
$$ \xymatrix{ \mathcal{H}^{ \star}(V) \ar@{->>}[d] \ar[r]^t & \mathcal{H}^{ \star+3+2 \alpha}(V) \\
\frac{ \Gamma_{v_1}( \kr^{ \star}_{\gr}(BV))}{Ker_{v_1}( \kr^{ \star}_{\gr}(BV))}^{ \star} \ar[r]^{\iota_0} & (v_1 Ker_{v_1^2}( \kr^{ \star}_{\gr}(BV)))^{ \star+2+ \alpha}. \ar@{^(->}[u] } $$ 
There are isomorphisms:
\begin{enumerate}
\item $$ \Sigma^{1+ \alpha} v_1 Ker_{v_1^2} \cong \Gamma_{v_1}( \kr^{ \star}_{\gr}(BV))/Ker_{v_1}( \kr^{ \star}_{\gr}(BV)) \cong Im(t),$$
\item  $$ \frac{ co\Gamma_{v_1}( \kr^{ \star}_{\gr}(BV))}{v_1 co\Gamma_{v_1}( \kr^{ \star}_{\gr}(BV))} \cong \frac{Ker(t)}{Im(t)},$$
\item and $$ \frac{Ker_{v_1}( \kr^{ \star}_{\gr}(BV))}{ v_1Ker_{v_1^2}( \kr^{ \star}_{\gr}(BV))} \cong Im_{ \qi} \circ Ker_{ \qo}( H \mf^{ \star}_{\gr}(BV)).$$
\end{enumerate}
\end{pro}

\begin{proof}
This is an explicit reformulation of Proposition \ref{thm:chaincplx} and the definition of $\iota_{n}$ in the particular case of the slice tower for $K\R$-theory, in the case of detection of height $2$ which is asserted by \ref{thm_2detectkr}.
\end{proof}

We now determine the morphism $t$ of the previous lemma.

\begin{lemma} \label{lemma_imt}
Let $$ \widetilde{\mathcal{H}}^{\star}(V) := \bigoplus_{i=1}^n \left(( \Sigma^{-i(1+ \alpha)}HP^{ \star})_{twist \geq 0} \oplus ( \Sigma^{-i(1+ \alpha)-1}HP^{ \star})_{twist \leq -2} \right)^{ \oplus \begin{pmatrix} n \\ i \end{pmatrix}},$$ the image by the functor $H_{01}^{\star} R$of the non $\ste(1)$-free part of $H \mf^{*}(BV)$.
Then, the map
$$t :\widetilde{\mathcal{H}}^{\star}(V) \oplus H_{01}^{ \star}(RF) \rightarrow \widetilde{\mathcal{H}}^{\star+3+2 \alpha}(V) \oplus H_{01}^{ \star+3+2 \alpha}(RF)$$
satisfies $t([ \sigma^{-2}Sq^1x]) = [Sq^2Sq^2Sq^2x]$, for all generator $x$ of a free sub-$\ste(1)$-module of  $H \F^*(BV)$, and $t$ takes trivial values elsewhere.
\end{lemma}

\begin{proof}
By  the first point of Lemma \ref{lemma_h01erkbv}, and by definition of $HP^{\star}$ there cannot be any non trivial morphism
$$ \widetilde{\mathcal{H}^{\star}(V)} \rightarrow \widetilde{\mathcal{H}^{\star+3+2 \alpha}(V)} .$$
Because of the cancellation of $H\mf^{*-\alpha}$, for all $* \in \Z$, $ ( \Sigma^{-1-\alpha}H \Z)^{\gr} = 0$, and thus $v_1^{\gr} : \kr^{\gr} \rightarrow \Sigma^{-1- \alpha} \kr$ is a weak auto equivalence of $ko$.

By definition of $t$, the following diagram is commutative

$$ \xymatrix{ H_{01}^{*+2-2 \alpha}(H \mf^{\star}_{\gr}(BV)) \ar[r]^{ \delta^{\gr}} & F^0_{-1} \ar[r]^{(v_1^{-1})^{\gr}} & F^2_0 \ar[r]^{c^{\gr} \ \ } & H_{01}^{*+5}(H \mf^{\star}_{\gr}(BV)) \\
\Sigma^{-2 \alpha}H \F^*(BV) \ar@{->>}[u] \ar[r]^{ \partial} &  \ ko^{*+5}(BV) \ar[r]^= & ko^{*+5}(BV) \ar[r]^{\tilde{c}} & \kera(H \F^*(BV)) \ar@{->>}[u] } $$

where $\partial$ and $\tilde{c}$ are morphisms coming from the Postnikov tower of $ko$ by

$$ \xymatrix{  ko\ar[d] \ar[r]^{\tilde{c}}& H \Z \\
KO<-4> \ar[r] & \Sigma^{-4}H \Z. \ar@{-->}[ul]^{\partial} } $$

But \cite[section A.5]{BG10} identifies $ \tilde{c} \partial$ to an integral lift of the non-equivariant Steenrod operation $Sq^2Sq^1Sq^2$.
The result follows.
\end{proof}

We are finally able to identify $\kr^{\star}_{\gr}(BV)$.

\begin{thm} \label{thm:krbvformula}
There is a $\Z[a,v_1]$-module splitting of $ \kr^{\star}_{\gr}(BV)$ as
$$ \kr^{\star}_{\gr}(BV) \cong co\Gamma_{v_1}(  \kr^{\star}_{\gr}(BV)) \oplus F^1(V) \oplus F^2(V) \otimes_{\Z} \Lambda(v_1)$$
and isomorphisms:
\begin{enumerate}
\item $F^1(V) \cong Im( \qi : H \mf^{\star}_{\gr}(BV) \rightarrow H \mf^{\star+2+\alpha}_{\gr}(BV))$,
\item $F^2(V) \cong Sq^2Sq^2Sq^2 F$ where $F$ is the largest free $ \ste(1)$-module contained in $ H \F^*(BV)$,
\item and $$ \Phi_n/\Phi_{n+1} \cong  \bigoplus_{i=1}^n \left(( \Sigma^{-i(1+ \alpha)}HP^{ \star})_{twist \geq 0} \oplus ( \Sigma^{-i(1+ \alpha)-1}HP^{ \star})_{twist \leq -2}, \right)^{ \oplus \begin{pmatrix} n \\ i \end{pmatrix}}$$
where $$\Phi_n = Im(v_1^n : co\Gamma_{v_1}( \kr^{\star+n(1+ \alpha)}_{\gr}(BV)) \rightarrow co\Gamma_{v_1}( \kr^{\star+n(1+ \alpha)}_{\gr}(BV))),$$ for $n \geq 0$ defines a decreasing exhaustive filtration of the $\Z[a,v_1]$-module $co\Gamma_{v_1}(  \kr^{\star}_{\gr}(BV))$.
\end{enumerate}
\end{thm}

\begin{proof}
There always is a splitting of the form $$ \kr^{\star}_{\gr}(BV) \cong co\Gamma_{v_1}(  \kr^{\star}_{\gr}(BV)) \oplus \Gamma_{v_1}(  \kr^{\star}_{\gr}(BV))$$
The $v_1$-torsion comes from:
\begin{itemize}
\item first point of Proposition \ref{thm:chaincplx} for $F^1(V)$,
\item Lemma \ref{lemma_imt} for $F^2(V)$,
\item by $(4)$ of Proposition \ref{pro:detect_1_2}, $\Gamma_{v_1}(  \kr^{\star}_{\gr}(BV)) \cong Ker_{v_1}/Ker_{v_1|_{Im(v_1)}} \oplus \Gamma_{v_1}/Ker_{v_1} \otimes_{\Z} \Lambda(v_1)$ by detection of height $2$ of Theorem \ref{thm_2detectkr}).
\end{itemize}
Finally, the filtration of $co\Gamma_{v_1}(  \kr^{\star}_{\gr}(BV))$ is provided by point 2 of Proposition \ref{thm:chaincplx}. The exhaustivity in each $RO(\gr)$-grading is easily checked by connectivity of $K\R$.
\end{proof}

\bibliographystyle{alpha}
\bibliography{biblio}

\begin{thebibliography}{ATJLSS03}

\bibitem[Ati66]{At66}
M.~F. Atiyah.
\newblock {$K$}-theory and reality.
\newblock {\em Quart. J. Math. Oxford Ser. (2)}, 17:367--386, 1966.

\bibitem[ATJLSS03]{AJS}
Leovigildo Alonso~Tarr{\'{\i}}o, Ana Jerem{\'{\i}}as~L{\'o}pez, and
  Mar{\'{\i}}a~Jos{\'e} Souto~Salorio.
\newblock Construction of {$t$}-structures and equivalences of derived
  categories.
\newblock {\em Trans. Amer. Math. Soc.}, 355(6):2523--2543 (electronic), 2003.

\bibitem[Ric14]{Ric3}
Nicolas Ricka.
\newblock Equivariant anderson duality and mackey functor duality.
\newblock {\em to appear in Glasgow Mathematical Journal}, 2014.

\bibitem[Ric13]{RicPhD}
Nicolas Ricka.
\newblock Sous-alg\`ebres de l'alg\`ebre de Steenrod \'equivariante et une propri\'et\'e de d\'etection en K-th\'eorie d'Atiyah.
\newblock {\em Universit\'e Paris 13, PhD thesis}, 2013.



\bibitem[BG03]{BG03}
R.~R. Bruner and J.~P.~C. Greenlees.
\newblock The connective {$K$}-theory of finite groups.
\newblock {\em Mem. Amer. Math. Soc.}, 165(785):viii+127, 2003.

\bibitem[BG10]{BG10}
Robert~R. Bruner and J.~P.~C. Greenlees.
\newblock {\em Connective real {$K$}-theory of finite groups}, volume 169 of
  {\em Mathematical Surveys and Monographs}.
\newblock American Mathematical Society, Providence, RI, 2010.

\bibitem[Boa95]{Bo95}
J.~Michael Boardman.
\newblock Stable operations in generalized cohomology.
\newblock In {\em Handbook of algebraic topology}, pages 585--686.
  North-Holland, Amsterdam, 1995.

\bibitem[{Bru}14]{BrOssa}
R.~R. {Bruner}.
\newblock{An Alpine Expedition through Algebraic Topology, Contemporary Mathematics, vol. 617, Amer. Math. Soc., Providence, RI, 2014, PP.81-108}

\bibitem[CY]{Yu}
Yu~Cherng-Yih.
\newblock The connective real k-theory of elementary abelian 2-groups ph.d.
  thesis, university of notre dame, june 1995.

\bibitem[{Dug}03]{Du03}
D.~{Dugger}.
\newblock {An Atiyah-Hirzebruch spectral sequence for KR-theory}.
\newblock {\em ArXiv Mathematics e-prints}, April 2003.

\bibitem[EJ11]{EJ11}
Edgar~E. Enochs and Overtoun M.~G. Jenda.
\newblock {\em Relative homological algebra. {V}olume 2}, volume~54 of {\em de
  Gruyter Expositions in Mathematics}.
\newblock Walter de Gruyter GmbH \& Co. KG, Berlin, 2011.

\bibitem[EM65]{EM65}
Samuel Eilenberg and J.~C. Moore.
\newblock Foundations of relative homological algebra.
\newblock {\em Mem. Amer. Math. Soc. No.}, 55:39, 1965.

\bibitem[GM95]{GM95}
J.~P.~C. Greenlees and J.~P. May.
\newblock Equivariant stable homotopy theory.
\newblock pages 277--323, 1995.

\bibitem[HHR09]{HHR}
Michael~A. Hill, Michael~J. Hopkins, and Douglas~C. Ravenel.
\newblock On the non-existence of elements of kervaire invariant one, v2, 2009.

\bibitem[HK01]{HK01}
Po~Hu and Igor Kriz.
\newblock Real-oriented homotopy theory and an analogue of the
  {A}dams-{N}ovikov spectral sequence.
\newblock {\em Topology}, 40(2):317--399, 2001.

\bibitem[Hoc56]{Hoch56}
G.~Hochschild.
\newblock Relative homological algebra.
\newblock {\em Trans. Amer. Math. Soc.}, 82:246--269, 1956.

\bibitem[HPS97]{HPS}
Mark Hovey, John~H. Palmieri, and Neil~P. Strickland.
\newblock Axiomatic stable homotopy theory.
\newblock {\em Mem. Amer. Math. Soc.}, 128(610):x+114, 1997.

\bibitem[Lew95]{Le95}
L.~Gaunce Lewis, Jr.
\newblock Change of universe functors in equivariant stable homotopy theory.
\newblock {\em Fund. Math.}, 148(2):117--158, 1995.

\bibitem[MM02]{MM}
M.~A. Mandell and J.~P. May.
\newblock Equivariant orthogonal spectra and {$S$}-modules.
\newblock {\em Mem. Amer. Math. Soc.}, 159(755):x+108, 2002.

\bibitem[FL04]{FL04}
Kevin~K. Ferland and L.~Gaunce Lewis, Jr.
\newblock The {$R{\rm O}(G)$}-graded equivariant ordinary homology of
  {$G$}-cell complexes with even-dimensional cells for {$G={\Bbb Z}/p$}.
\newblock {\em Mem. Amer. Math. Soc.}, 167(794):viii+129, 2004.

\bibitem[LMSM86]{LMS}
L.~G. Lewis, Jr., J.~P. May, M.~Steinberger, and J.~E. McClure.
\newblock {\em Equivariant stable homotopy theory}, volume 1213 of {\em Lecture
  Notes in Mathematics}.
\newblock Springer-Verlag, Berlin, 1986.
\newblock With contributions by J. E. McClure.


\bibitem[Mar83]{Mar83}
H.~R. Margolis.
\newblock {\em Spectra and the {S}teenrod algebra}, volume~29 of {\em
  North-Holland Mathematical Library}.
\newblock North-Holland Publishing Co., Amsterdam, 1983.
\newblock Modules over the Steenrod algebra and the stable homotopy category.

\bibitem[MV]{MV01}
F.~Morel and V.~Voevodsky.
\newblock $\mathcal{A}^1$-homotopy theory for schemes , inst. hautes \'etudes
  sci. publ. math., no. 90 (2001), 45--143.

\bibitem[Oss89]{Os89}
E.~Ossa.
\newblock Connective {$K$}-theory of elementary abelian groups.
\newblock In {\em Transformation groups ({O}saka, 1987)}, volume 1375 of {\em
  Lecture Notes in Math.}, pages 269--275. Springer, Berlin, 1989.

\bibitem[Pal01]{Pal01}
John~H. Palmieri.
\newblock Stable homotopy over the {S}teenrod algebra.
\newblock {\em Homology Homotopy Appl.}, 2014(1), 2001.

\bibitem[{Pow}11]{Po11}
G.~{Powell}.
\newblock {On connective K-theory of elementary abelian 2-groups and local
  duality}.
\newblock {On connective $K$-theory of elementary abelian 2-groups and local duality. 
 Homology Homotopy Appl. 16 (2014), no. 1, 215--243.}


\bibitem[{Pow}12]{Po12}
G.~{Powell}.
\newblock {On connective KO-theory of elementary abelian 2-groups}.
\newblock {\em ArXiv e-prints}, February 2014.

\bibitem[Voe03]{Voe03}
Vladimir Voevodsky.
\newblock Reduced power operations in motivic cohomology.
\newblock {\em Publ. Math. Inst. Hautes \'Etudes Sci.}, (98):1--57, 2003.

\end{thebibliography}

\end{document}